\documentclass[a4paper,11pt]{article}
\pdfoutput=1
\usepackage[english]{bricearticle}

\title{Computing harmonic maps between Riemannian manifolds}
\author{
Jonah Gaster\footnote{University of Wisconsin-Milwaukee, Department of Mathematical Sciences. Milwaukee, WI 53201-0413, USA.\newline
E-mail: \url{gaster@uwm.edu}}, ~ 
Brice Loustau\footnote{Rutgers University - Newark, Department of Mathematics. Newark, NJ 07105 USA; and
TU Darmstadt, Department of Mathematics. 64289 Darmstadt, Germany. 
E-mail: \url{loustau@mathematik.tu-darmstadt.de}}, ~and
Léonard Monsaingeon\footnote{IECL Université de Lorraine, Site de Nancy. F-54506 Vand{\oe}uvre-lès-Nancy Cedex, France; and 
GFM Universidade de Lisboa. 1749-016 Lisboa, Portugal. 
E-mail: \url{leonard.monsaingeon@univ-lorraine.fr}}}
\date{} 

\hypersetup{
    pdftitle={Computing harmonic maps between Riemannian manifolds},
    pdfauthor={Jonah Gaster, Brice Loustau, and Léonard Monsaingeon},
}

\newcommand{\Harmony}{\texttt{Harmony}}

\begin{document}

\pdfbookmark[1]{Title page, abstract}{Title}
\begin{titlepage}
\newgeometry{top=0.09\paperheight, bottom=0.11\paperheight}
\maketitle
\thispagestyle{empty}
\centering
\includegraphics[width=90mm]{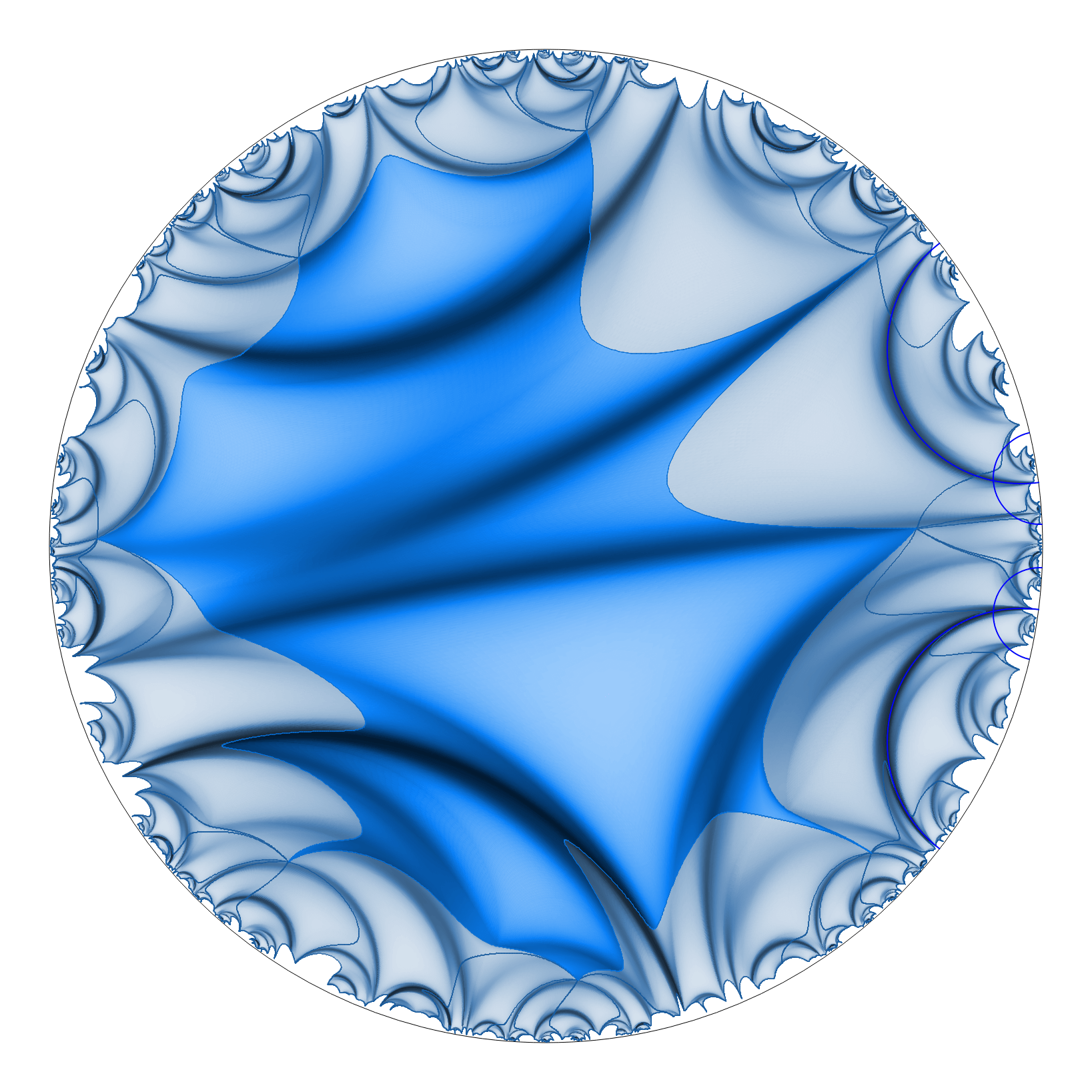}
\bigskip

\bigskip 

\begin{abstract}

In our previous paper \cite{Gaster-Loustau-Monsaingeon1}, we showed that the theory of harmonic maps between Riemannian manifolds, especially hyperbolic surfaces,
may be discretized by introducing a triangulation of the domain manifold with independent vertex and edge weights.
In the present paper, we study convergence of the discrete theory back to the smooth theory when taking finer and finer triangulations, in the general Riemannian setting. 
We present suitable conditions on the weighted triangulations that ensure convergence of discrete harmonic maps to smooth harmonic maps, introducing
the notion of (almost) asymptotically Laplacian weights. We also present a systematic method to construct such weighted triangulations in the $2$-dimensional case. 
Our computer software \Harmony{} successfully implements these methods to computes equivariant harmonic maps in the hyperbolic plane.

\bigskip \bigskip


\bigskip


\end{abstract}

\end{titlepage}

 \thispagestyle{empty}
 \setcounter{tocdepth}{2}
 \pdfbookmark[1]{Contents}{Contents}  
 {\tableofcontents}
 \addtocounter{page}{1}
 \restoregeometry

\cleardoublepage\phantomsection 
\section*{Introduction}
\addcontentsline{toc}{section}{Introduction}

Let $M$ and $N$ be Riemannian manifolds, let us assume $M$ compact and $N$ complete. A harmonic map $f \colon M \to N$ is a critical point of the energy
functional
\begin{equation}
 E(f) = \frac{1}{2} \int_M \Vert \upd f\Vert^2 \upd v\,.
\end{equation}
Equivalently, $f$ has vanishing tension field $\tau(f) = 0$, a nonlinear generalization of the Laplace operator that can be defined as the trace
of the Riemannian Hessian: $ \tau(f) = \nabla(\upd f)$. 
When $N$ is compact and has negative
sectional curvature,
there exists a  harmonic map $M \to N$ in any homotopy class of smooth maps, and it is unique unless it is constant or maps to a geodesic. This foundational result
due to Eells-Sampson \cite{MR0164306} and Hartman \cite{MR0214004} can be understood in terms of the convexity properties of the energy.
Essentially, the curvature assumption on $N$ implies that the energy functional is convex on any component of the space of smooth maps $\cC^\infty(M,N)$,
which guarantees convergence of the gradient flow--also called heat flow in this setting--from any initial smooth map to the energy minimizer.

In our previous work \cite{Gaster-Loustau-Monsaingeon1}, which mostly specialized to surfaces, we showed that the theory can be appropriately discretized 
by meshing the domain manifold
with a triangulation and assigning two independent systems of weights, on the set of vertices and edges respectively. One of the main results
 is the strong convexity of the discrete energy functional, from which we derive convergence of the discrete heat flow
to the unique discrete harmonic map. (The second focus of \cite{Gaster-Loustau-Monsaingeon1} is on center of mass methods, which we do not 
discuss in the present paper.)
While that paper was concerned with a fixed discretization, the purpose of the present paper
is to study the convergence of the discrete theory back to the smooth theory when one takes finer and finer meshes.

\bigskip

After introducing the discretization setup in \autoref{sec:setup}, in \autoref{sec:SystemsOfWeights}
we discuss special conditions on weighted triangulations in order to adequately
capture the local geometry of the domain manifold. We define \emph{Laplacian} systems of weights, which aim to produce a good approximation
of the Laplacian (\ie{} tension field) by the discrete Laplacian. As a fundamental example, we introduce our favorite \emph{volume vertex weights}
and \emph{cotangent edge weights}.

In \autoref{sec:SequencesOfMeshes}, we study fine sequences of meshes (with maximum edge length converging to zero), and the approximation of the relevant smooth objects by their discrete counterparts.
A key requirement for the sequence is to be \emph{crystalline}, meaning that all angles of the triangulation stay bounded away from zero.
We also strategically weaken the notion of Laplacian weights to \emph{(almost) asymptotically Laplacian} weights.
We show that for such sequences of weighted meshes, which we will later see can systematically be constructed, there is 
convergence of the discrete volume form, tension field, energy density, and energy to their smooth counterparts.

In \autoref{sec:Convergence}, we study the convergence of discrete maps to smooth harmonic maps.
If the discrete energy is sufficiently convex, and the sequence of meshes is almost asymptotically Laplacian, 
we prove that (the center of mass interpolations of) the discrete harmonic maps converge to the unique smooth harmonic map in $\upL^2$. 
We expect the strong convexity assumption to hold in a very broad setting, and have proved it in the $2$-dimensional case in \cite{Gaster-Loustau-Monsaingeon1}.
Pending stronger assumptions, we also show convergence in $\upL^\infty$, and in energy. Furthermore, we show that the discrete heat flow starting from any discretized map converges to the smooth harmonic map when both the time index
and the space index run to $+\infty$, provided a CFL-type condition is satisfied. This theorem may be seen as a constructive implementation of the theorem
of Eells-Sampson and Hartman.

The final section \autoref{sec:Construction} of the paper describes how to systematically construct almost asymptotically Laplacian sequences of meshes, so that
our previous theorems can apply, at least in the $2$-dimensional case.
These are quite simply constructed by iterated midpoint geodesic subdivision from an initial triangulation of the domain manifold, and taking the volume weights on vertices and cotangent weights on edges. 
Proving the required Laplacian qualities some delicate Riemannian geometry estimates, naturally building on the Euclidean case; we largely relegate these to the appendix (\autoref{sec:RiemannianEstimates}) to avoid burdening our exposition.
It is quite remarkable how the conditions for our constructed sequences to be almost asymptotically Laplacian are barely met,
and in turn how these conditions are barely sufficient for our main convergence theorem (\autoref{thm:ConvergenceL2}) to hold.

Putting together the main theorems in \autoref{sec:Convergence} and \autoref{sec:Construction} (\autoref{thm:ConvergenceL2}, \autoref{thm:MainThmSurfaces}, and \autoref{thm:DeltaSequenceLaplacian}), we obtain explicit constructions of sequences of discretizations that ensure convergence to the desired harmonic map. 
Here is a sample theorem summarizing our main results for surfaces:

\begin{theorem*}
\label{mainThm}
Let $M$ and $N$ be compact Riemannian $2$-manifolds of negative Euler characteristics, and assume $N$ has negative sectional curvature. 
Consider a sequence of meshes on $M$ obtained by iterated midpoint subdivision with all angles 
bounded away from $\pi/2$, and equip it with the area vertex weights and cotangent edge weights.
Let $\cC$ be a component of $\cC^\infty(M,N)$ of nonzero degree, 
and let
$v_n$ be the unique discrete harmonic map in the corresponding discrete homotopy class. 
Then $v_n$ converges to the unique harmonic map $w \in \cC$ 
in the $\upL^2$ topology.
\end{theorem*}

This construction and the discrete heat flow is implemented in our freely available computer software \Harmony{}, 
which is presented in our previous paper \cite{Gaster-Loustau-Monsaingeon1}.
\Harmony{} computes the unique harmonic map from the hyperbolic plane to itself that is equivariant with respect to the actions of two Fuchsian groups, which
can be selected by the user via Fenchel-Nielsen coordinates.

\bigskip

Much of the theory and techniques that we develop are well-known in the Euclidean setting, such as the 
discrete heat flow method or the cotangent weights popularized by Pinkall-Polthier \cite{MR1246481}. 
This paper builds upon the Euclidean theory by using fine meshes on Riemannian manifolds.
However, there are notable differences from the Euclidean setting:
First, the Laplace equation is linear in the Euclidean setting, allowing finite element methods. 
Second, we restrict to compact manifolds without boundary, 
in contrast to Euclidean domains where boundary conditions are prescribed. 
Finally, there are important consequences of negative curvature, 
including the strong convexity of the energy functional and the uniqueness of harmonic maps, that we exploit in the present project. 

The program to discretize the theory of harmonic maps between Riemannian manifolds, and to obtain convergence back to the smooth theory, remains unfinished. Celebrated work on the discretized theory includes \cite{MR2365833,MR1848068,MR1483983}, while convergence to the smooth harmonic map has been analyzed for submanifolds of $\bR^n$ notably by Bartels \cite{MR2629993}. 
The present paper seems to have some overlap with Bartels' work, 
though 
our setting is more intrinsic and geometric in nature. 
A perhaps more powerful approach than ours to prove convergence of discrete harmonic maps to smooth harmonic maps
would consist in finding a discrete version of Bochner's formula and possibly Moser's Harnack inequality: see \autoref{rem:Dickless}.

\bigskip

\emph{A note to the reader:} Although this paper is the sequel of \cite{Gaster-Loustau-Monsaingeon1}, the two papers can be read independently. We
also point out that \autoref{sec:Convergence} and \autoref{sec:Construction} in this paper can be read independently.

\phantomsection 
\subsection*{Acknowledgments}
\addcontentsline{toc}{section}{Acknowledgments}

The authors wish to thank David Dumas for 
his extensive advice and support with the mathematical content and the development of \Harmony{}.

The first two authors gratefully acknowledge research support from the NSF Grant DMS1107367 \emph{RNMS: GEometric structures And Representation varieties} (the \emph{GEAR Network}).
The third author was partially supported by the Portuguese Science Foundation FCT trough grant PTDC/MAT-STA/0975/2014 \emph{From Stochastic Geometric Mechanics to Mass Transportation Problems}.

\section{Setup} \label{sec:setup}

Throughout the paper, let $(M,g)$ and $(N,h)$ be smooth connected complete Riemannian manifolds. These will be our domain and target respectively.
We will typically assume that $M$ is compact and oriented, and that $N$ is Hadamard (complete, simply connected, with nonpositive sectional curvature). 
Although most of the paper holds in this generality, 
we are especially interested in the case where $S = M$ is $2$-dimensional.
For background on the smooth theory of harmonic maps $M \to N$, please refer to \cite[\S 1]{Gaster-Loustau-Monsaingeon1}.

\subsection{Discretization setup}
\label{subsec:DiscretizationSetup}

Our discretization setup is the following. (We also refer to \cite[\S 2]{Gaster-Loustau-Monsaingeon1} for more details, although it focuses on the equivariant setting
and $\bH^2$.) 
A \emph{mesh} on $M$ is any topological triangulation; we denote by $\cG$
the embedded graph that is the $1$-skeleton. A mesh (or its underlying graph) is called
\emph{geodesic} if all edges are embedded geodesic segments.

Denote $\cV = \cG^{(0)}$ and $\cE = \cG^{(1)}$ the set of vertices and
(unoriented) edges of $\cG$. 
We shall equip $\cG$ with vertex weights $(\mu_x)_{x \in \cV}$ and edge weights $(\omega_{xy})_{\{x,y\} \in \cE}$. 
For now, these weights are two arbitrary and independent collections of positive numbers.
Such a \emph{biweighted graph} allows one to develop a discrete theory of harmonic maps $M \to N$ as follows:

\medskip
\begin{enumerate}[$\bullet$, itemsep=2ex, wide, labelindent=0pt]
 \item The system of vertex weights defines a measure $\mu_\cG = (\mu_x)_{x \in \cV}$ on $\cV$. Since $\cG$ is embedded in $M$,
 $\mu_\cG$ can also be seen as a discrete measure on $M$ supported by the set of vertices.
 \item A \emph{discrete map from $M$ to $N$ along $\cG$} is a map $\cV \to N$. The space $\Map_\cG(M, N)$ of such maps 
 is a smooth finite-dimensional manifold with tangent space
 \begin{equation}
 \upT_f \Map_\cG(M, N) = \Gamma(f^* \upT N) \coloneqq \bigoplus_{x \in \cV} \upT_{f(x)} N\,.
 \end{equation}
 It carries a smooth $\upL^2$-Riemannian metric given by:
 \begin{equation} \label{eq:DiscreteL2RiemannianMetric}
\langle V, W \rangle \coloneqq \int_M \langle V_x, W_x\rangle \, \upd \mu_{\cG}(x) = \sum_{x\in \cV} \mu_x \langle V_x, W_x\rangle
\end{equation}
and an associated $\upL^2$ distance given by
\begin{equation}
 d(f,g)^2 \coloneqq \int_M d(f(x), g(x))^2 \, \upd \mu_{\cG}(x) = \sum_{x \in \cV} \mu_x d(f(x), g(x))^2 
\end{equation}
where $d(f(x), g(x))$ denotes the Riemannian distance in $N$.
 \item The \emph{discrete energy density} of a discrete map $f \in \Map_\cG(M, N)$ is the discrete 
 nonnegative function $e_\cG(f) \in \Map_\cG(M, \R)$
 defined by
\begin{equation}\label{eq:DiscreteEnergyDensity}
e_\cG(f)_{x} = \frac{1}{4\mu_x} \sum_{y\sim x} \omega_{xy} \, d(f(x),f(y))^2\,.
\end{equation}
 \item The \emph{discrete energy functional} on $\Map_\cG(M, N)$ is the map $E_\cG \colon \Map_\cG(M, N) \to \R$ given by 
\begin{equation}\label{eq:DiscreteEnergyFunctional}
\begin{split}
E_\cG(f) &= \int_M e_\cG(f) \upd \mu_\cG\\
&= \frac12 \sum_{x\sim y} \omega_{xy} \, d(f(x),f(y))^2\,.
\end{split}
\end{equation}
 A \emph{discrete harmonic map} is a critical point of $E_\cG$.
\begin{remark}
 The discrete energy functional does not depend on the choice of vertex weights, neither does the harmonicity of a discrete map.
 When $M$ is $2$-dimensional, this is reflects the fact that the energy functional $E \colon \cC^\infty(M,N) \to \R$ only depends on the conformal structure on $S$.
\end{remark}
 \item  The \emph{discrete tension field} of $f \in \Map_\cG(M, N)$ is $\tau_\cG (f) \in \Gamma(f^* \upT N)$ defined by
 \begin{equation} \label{eq:DiscreteTensionField}
 \tau_\cG(f)_{x} = \frac{1}{\mu(x)} \sum_{y \sim x} \omega_{xy} \overrightarrow{f(x) f(y)}\,.
\end{equation}
\begin{notation}
 Throughout the paper, we abusively denote $\overrightarrow{xy} \coloneqq \exp_{x}^{-1}(y)$ (whenever well-defined), where
 $\exp_x$ is the Riemannian exponential map.
\end{notation}
In \cite[Prop.\,2.21]{Gaster-Loustau-Monsaingeon1} we show the discrete first variational formula:
\begin{equation}
\tau_\cG(f) = - \grad E_\cG(f)\,.
\end{equation}
In particular, $f$ is harmonic if and only if $\tau_\cG(f) = 0$. This is equivalent to the property that 
for all $x\in \cV$, $f(x)$ is the center of mass of its neighbors values (more precisely of the system $\{f(y), \omega_{xy}\}$ for $y$ adjacent to $x$
\cite[Prop.\,2.22]{Gaster-Loustau-Monsaingeon1}).
\item Given $u_0 \in \Map_\cG(M, N)$ and $t >0$, the \emph{discrete heat flow with fixed stepsize $t$} is the sequence
$(u_k)_{k \geqslant 0}$ defined by
\begin{equation}
 u_{k+1} = \exp(t \, \tau_\cG (u_k))\,.
\end{equation}
The discrete heat flow is precisely the fixed stepsize gradient descend method for the discrete energy functional $E_\cG$.
\end{enumerate}

\medskip

One of the main theorems of \cite{Gaster-Loustau-Monsaingeon1} is that if  $S = M$ and $N$ are closed oriented surfaces of negative Euler characteristics
and $u_0$ has nonzero degree, then the discrete heat flow
converges as $k \to +\infty$ to the unique minimizer of $E_\cG$ in the same homotopy class 
with exponential convergence rate.
See \cite[Theorem 4.5]{Gaster-Loustau-Monsaingeon1} for more details.

\subsection{Midpoint subdivision of a mesh} \label{subsec:MidpointSubdivision}

Assume $(M,g)$ is equipped with a geodesic mesh and denote by $\cG$ the associated graph.
One can define a new mesh called the \emph{midpoint subdivision} (or \emph{refinement}) as follows. 
For comfort, let us assume $M = S$ is $2$-dimensional; the definition is easily generalized.
Define a new geodesic graph $\cG'$ by adding to the vertex set of $\cG$
all the midpoints of edges of $\cG$, and adding new edges so that every triangle in $\cG$ is subdivided as $4$ triangles in $\cG'$
(see \cite[Definition 2.2]{Gaster-Loustau-Monsaingeon1}). This clearly defines a new geodesic triangulation of $S$ whose $1$-skeleton is $\cG'$.
See \autoref{fig:Meshes} for an illustration of an invariant mesh in $\H^2$ and its refinement generated by the software
\Harmony{}. 

Evidently, this subdivision process may be iterated, thus one can define the \emph{refinement of order $n$} of a geodesic mesh. 
Meshes obtained by successive midpoint refinements will be our standard support for approximating 
a smooth manifold by discrete data. Properties of such meshes will be further discussed in \autoref{sec:Construction}.

\begin{figure}
\centering
\begin{subfigure}{.5\textwidth}
  \centering
  \includegraphics[width=.98\textwidth]{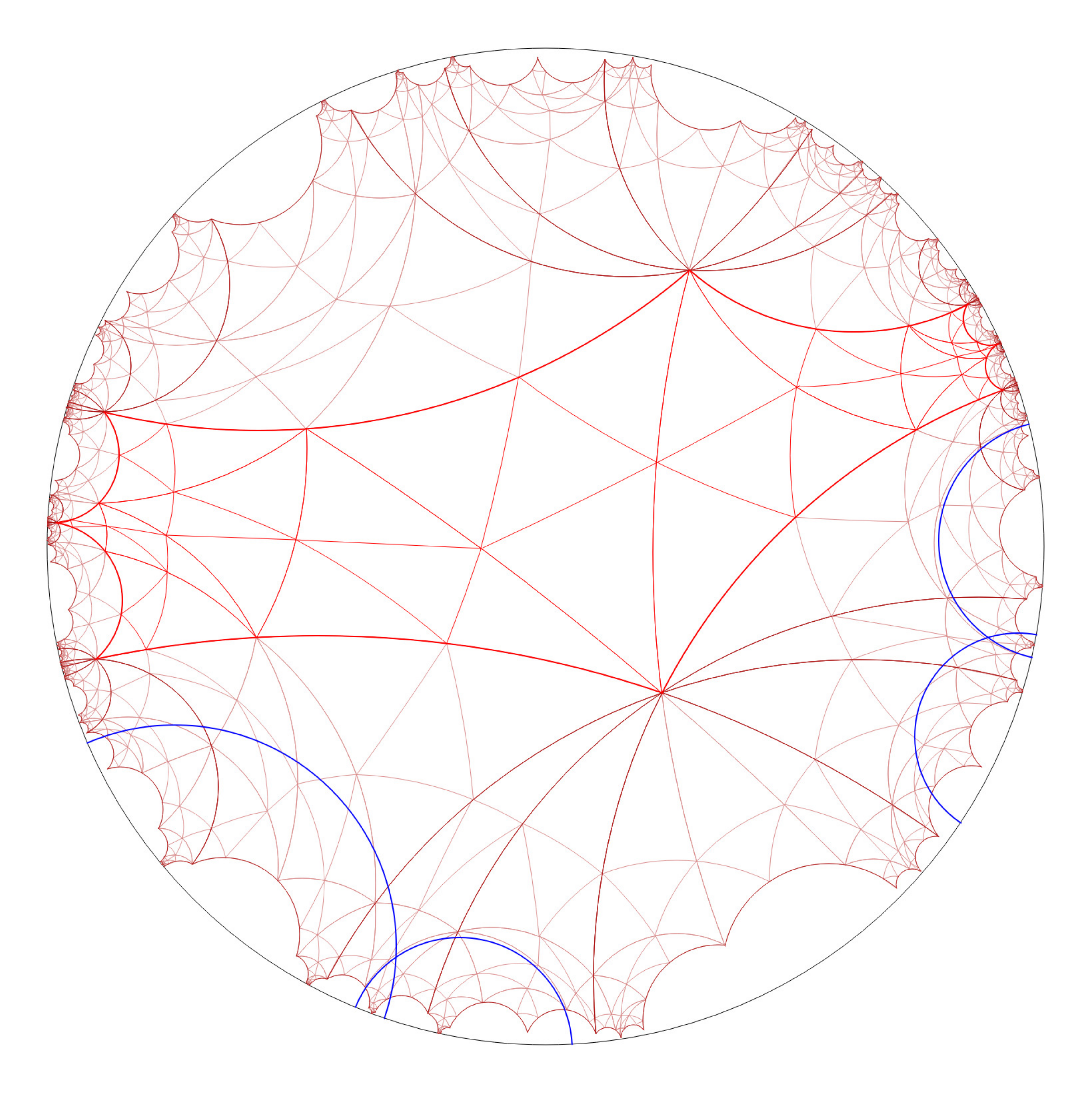}
  \caption{A mesh of $\bH^2$}
  \label{fig:Mesh0}
\end{subfigure}%
\begin{subfigure}{.5\textwidth}
  \centering
  \includegraphics[width=.98\textwidth]{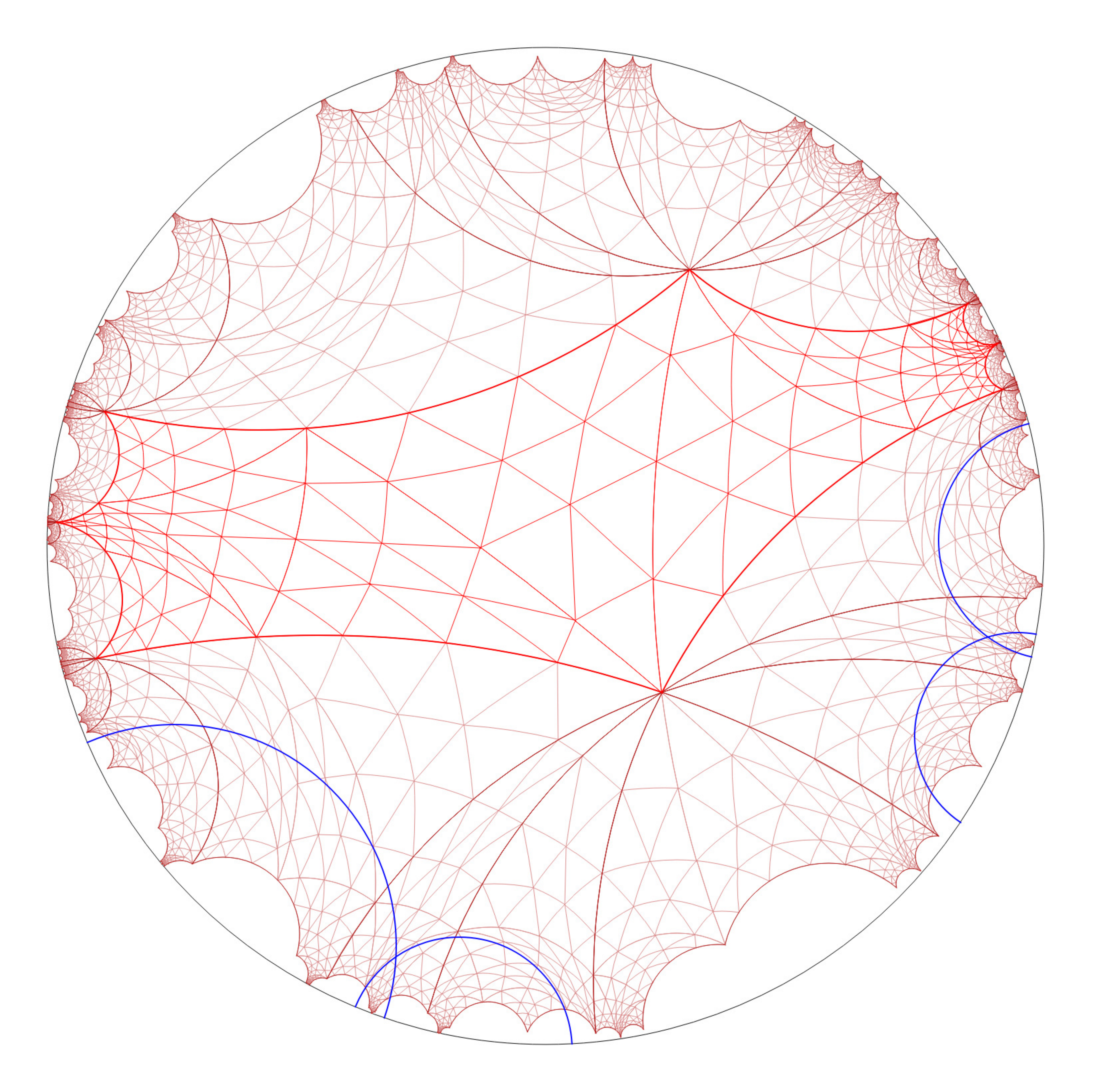}
  \caption{Midpoint refinement}
  \label{fig:Mesh1}
\end{subfigure}
\caption{A mesh of the Poincaré disk model of $\bH^2$ on the left, its midpoint refinement on the right.
Both are invariant under the action of a Fuchsian group $\Gamma$, yielding meshes on a closed hyperbolic surface $S$ of genus $2$.
The brighter central region is a fundamental domain. The blue circle arcs are the axes of the generators of $\Gamma \approx \pi_1 S$.}
\label{fig:Meshes}
\end{figure}

\subsection{Interpolation}

\subsubsection{Generalities}

Assume $(M,g)$ is equipped with a geodesic mesh and denote by $\cG$ the associated graph.
A continuous map $f \colon M \to N$ is \emph{piecewise smooth along $\cG$} if $f$ is smooth in restriction to any simplex of the mesh.

Note that there is a forgetful (restriction) map
\begin{equation}
 \pi_\cG \colon \cC(M,N) \to \Map_\cG(M, N)
\end{equation}
which assigns to any continuous map $f \colon M \to N$ its restriction to the vertex set of $\cG$. 
A first definition of an \emph{interpolation scheme} would be 
a right inverse $\iota_\cG$ of the map $\pi_\cG$.

Of course, a natural requirement to add is that $\iota_\cG$ is a continuous map 
whose image is contained in 
the subspace of piecewise smooth maps along $\cG$.
In the Euclidean setting, there is one canonical choice for interpolation, namely linear interpolation. In the general Riemannian
setting there is no such obvious choice. For our purposes we will view center of mass interpolation as the preferred interpolation,
though there are other natural options (\eg{} harmonic interpolation), which we will not discuss.

There is a subtle deficiency in the above definition of interpolation scheme when $N$ is not simply connected: 
one would like to require that $\iota_\cG \circ \pi_\cG$ preserves homotopy classes of maps, but that is not possible.
This problem can be solved by defining an interpolation scheme as attached to the choice of a homotopy class:
\begin{definition}
Let $\cC$ be a connected component of $\cC(M,N)$. An interpolation scheme $\iota_\cG$ is a continuous right inverse of $\pi_\cG$ restricted to $\cC$, 
whose image consists of piecewise smooth maps along $\cG$.
\end{definition}
Note that this definition still does not allow one to define the homotopy class of a discrete map.
A more elegant way to deal with deficiency, which we favored in \cite{Gaster-Loustau-Monsaingeon1}, is to work equivariantly
in the universal covers.

\subsubsection{Working equivariantly}
\label{subsubsec:WorkingEquivariantly}


Fix a homotopy class $\cC$ of a continuous 
map $M \to N$
, which induces a group homomorphism $\rho \colon \pi_1 M \to \pi_1 N$. Recall that any $f \in \cC$
admits a $\rho$-equivariant lift between universal covers $\tilde{f} \colon \tilde{M} \to \tilde{N}$.
The mesh $\cM$ on $M$ also lifts to a $\pi_1 M$-invariant geodesic mesh $\tilde{\cM}$ of $\tilde{M}$.
As usual, one has to take more care with basepoints on $M$ and $N$--and use more notation--to make this story complete.

\begin{definition} \label{def:DiscreteHomotopyClass}
 The discrete homotopy class $\cC_\cG \coloneqq \Map_{\tilde{\cG}, \rho}(\tilde{M}, \tilde{N})$ is defined as the space of $\rho$-equivariant discrete maps $\tilde{M} \to \tilde{N}$ along $\tilde{\cG}$.
\end{definition}
One can then define an interpolation theme as a continuous right inverse of $\pi_\cG$ on  $\cC_\cG$.
For the purposes of this paper, however, all of the convergence analysis can be performed on the quotient manifolds.
The presentation is chosen with ease in mind, and so we overlook the subtlety above. 
Nevertheless, we point out that there are other benefits
to the equivariant setting:
\begin{itemize}
 \item It allows one to 
consider equivariance with respect to group homomorphisms $\rho \colon \pi_1 M \to \Isom(\tilde{N})$ that are not necessarily
 induced by continuous maps from $M$ to a quotient of $\tilde{N}$, \eg{} non-discrete representations $\rho$.
 \item Computationally, it is easier to work in the universal covers. This is the point of view that we chose when coding the software \Harmony.
\end{itemize}
This explains our present change in perspective from the equivariance throughout \cite{Gaster-Loustau-Monsaingeon1}.

\subsubsection{Center of mass interpolation}

We refer to \cite[\S 5.1]{Gaster-Loustau-Monsaingeon1} for generalities on centers of mass, also called barycenters, in metric spaces and Riemannian manifolds.

For comfort, let us assume that $S = M$ is $2$-dimensional; it is quite straightforward to generalize what follows to higher dimensions.
First we describe interpolation between triples of points.
Let $A, B, C$ be three points on the surface $(S,g)$. We assume that these three points are sufficiently close, more precisely that
they lie in a strongly convex geodesic ball $B$, \ie{} any two points of $B$ are joined by a unique minimal geodesic segment in 
$S$ and this segment is contained in $B$. In particular, there is a uniquely defined triangle $T \subseteq S$ with vertices $A$, $B$, $C$ and with geodesic boundary.
Any point $P \in T$ can uniquely be written as the center of mass of $\{(A, \alpha), (B, \beta), (C, \gamma)\}$, where $\alpha, \beta, \gamma \in [0,1]$
and $\alpha + \beta + \gamma = 1$.
Let similarly $A'$, $B'$, $C'$ be three sufficiently close points in the Riemannian manifold $(N, h)$. Then there is a unique
\emph{center of mass interpolation} map $f \colon A B C \to N$ such that for any point $P \in T$ as above, $f(P)$ is the center of mass
of $\{(A', \alpha), (B', \beta), (C', \gamma)\}$.
In other words, $f$ is the identity map in barycentric coordinates.

Clearly, given a discrete map $f \in \Map_\cG(S, N)$, one can define its center of mass interpolation triangle by triangle following the procedure above.
Although there seems to be a restriction on the size of the triangles in $S$ and their images by $f$ in $N$ for the interpolation to be well-defined, 
one can work equivariantly in the universal covers as explained in \autoref{subsubsec:WorkingEquivariantly} and the restriction disappears as long as $S$ has nonpositive sectional curvature, or $\cG$ is sufficiently fine \ie{} has small maximum edge length, and $N$ has nonpositive sectional curvature.

\begin{definition}
 Assume $(M,g)$ has nonpositive sectional curvature, or $\cG$ is sufficiently fine, and $N$ has nonpositive sectional curvature.
 The discussion above yields a \emph{center of mass interpolation} scheme
 \begin{equation}
  \iota_\cG  \colon  \Map_\cG(M, N) \to \cC(M,N)\,.
 \end{equation}
 We denote $\widehat{f} \coloneqq \iota_\cG(f)$ the center of mass interpolation of a discrete map $f \in \Map_\cG(M, N)$.
\end{definition}

\begin{theorem} \label{thm:CoMOInterpolation}
Assume $M$ has nonpositive sectional curvature, or $\cG$ is sufficiently fine, and $N$ has nonpositive sectional curvature.
Then
\begin{enumerate}[(i)]
 \item \label{thm:CoMOInterpolationitemi} For any $f \in \Map_\cG(M, N)$, the interpolation $\widehat{f}$ maps each edge of $\cG$ to a geodesic segment in $M$ (and does so with constant speed).
 \item \label{thm:CoMOInterpolationitemii} For any $f \in \Map_\cG(M, N)$, the interpolation $\widehat{f}$ is piecewise smooth along $\cG$.
 \item \label{thm:CoMOInterpolationitemiii} The map $\iota_\cG  \colon  \Map_\cG(M, N) \to \cC(M,N)$ is $1$-Lipschitz
 for the $\upL^\infty$ distance on both spaces.
\end{enumerate}
\end{theorem}

\begin{proof}
For comfort, let us write the proof when $M = S$ is $2$-dimensional.
The proof of \ref{thm:CoMOInterpolationitemi} is immediate. 
For \ref{thm:CoMOInterpolationitemii}, recall that the center of mass $P$ as above is characterized by
\begin{equation}
\alpha \overrightarrow{PA} + \beta \overrightarrow{PB} + \gamma \overrightarrow{PC} = \vec{0}
\end{equation}
(see \cite[Eq. (37)]{Gaster-Loustau-Monsaingeon1}), where we denote $\overrightarrow{PA} \coloneqq \exp_{P}^{-1}(A)$ etc. It follows
from the implicit function theorem that $(\alpha, \beta, \gamma)$ provide smooth \emph{barycentric coordinates} on $T$ (resp. $T'$). Conclude
by observing that $\widehat{f}$ is the identity map in barycentric coordinates.

The proof of \ref{thm:CoMOInterpolationitemiii} is a little more delicate, and crucially relies on 
$N$ having nonpositive sectional curvature. Let $f_1, f_2 \in \Map_\cG(S, N)$, we want to show that
$d_\infty(\widehat{f_1}, \widehat{f_2}) \leqslant d_\infty(f_1, f_2)$. Consider any triangle in $\cG$ with vertices $A, B, C \in S$.
Let $p \in S$ be any point inside or on the boundary of the triangle $ABC \subseteq S$. 
We denote $A_i = f_i(A)$, $B_i = f_i(B)$, $C_i = f_i(C)$, $P_i = \widehat{f_i}(P)$ for 
$i\in \{1, 2\}$. Since $p$ is an arbitrary point on $S$, we win if we show that $d(P_1, P_2) \leqslant d_\infty(f_1, f_2)$.
By definition of the center of mass interpolation, $P_i$ is the center of mass of $\{(A_i, \alpha), (B_i, \beta), (C_i, \gamma)\}$, where $\alpha, \beta, \gamma \in [0,1]$ is some triple with $\alpha + \beta + \gamma = 1$ (namely, the unique triple such that $M$ is the center of mass of $\{(A, \alpha), (B, \beta), (C, \gamma)\}$).
Let $\vec{V}_i = \alpha \overrightarrow{P_i A_i} + \beta \overrightarrow{P_i B_i} + \gamma \overrightarrow{P_i C_i}$ and let 
$\vec{W} = \alpha \overrightarrow{P_1 A_2} + \beta \overrightarrow{P_1 B_2} + \gamma \overrightarrow{P_1 C_2}$, where we denote
$\overrightarrow{P_i A_i} = \exp_{P_i}^{-1}(A_i)$, etc. By definition of the center of mass $\vec{V}_i = \vec{0}$, so we can write
$\vec{W} = \vec{W} - \vec{V}_1$:
\begin{equation} \label{eq:InterpolationLipschitzProof}
 \vec{W} = \alpha \left( \overrightarrow{P_1 A_2} - \overrightarrow{P_1 A_1} \right)
         + \beta \left( \overrightarrow{P_1 B_2} - \overrightarrow{P_1 B_1} \right)
         + \gamma \left( \overrightarrow{P_1 C_2} - \overrightarrow{P_1 C_1} \right)
\end{equation}
Since $N$ has nonpositive sectional curvature, the exponential map $\exp_{P_1} \colon \upT_{P_1} N \to N$ is distance nondecreasing (for this argument
to be completely rigorous, we may need to pass to universal covers), so that 
$\Vert \overrightarrow{P_1 A_2} - \overrightarrow{P_1 A_1} \Vert \leqslant d(A_1, A_2)$, etc. Using the triangle inequality in 
\eqref{eq:InterpolationLipschitzProof} we find $\Vert \vec{W} \Vert \leqslant d_\infty(f_1, f_2)$.
This shows that $d(P_1, P_2) \leqslant d_\infty(f_1, f_2)$ by \cite[Lemma 5.3]{Gaster-Loustau-Monsaingeon1}.
\end{proof}

\section{Systems of weights}
\label{sec:SystemsOfWeights}

We follow the discretization setup of \autoref{sec:setup} and seek systems of vertex and edge weights on $\cG$ that adequately
capture the local geometry of $M$, in the sense that they ensure a good approximation of the 
theory of smooth harmonic maps from $M$ to any other Riemannian manifold.

Throughout this section $(M,g)$ is any Riemannian manifold equipped with a geodesic mesh $\cM$.
We denote as usual $\cG$ the associated graph.

\subsection{Laplacian weights}
\label{subsec:Laplacian}

\begin{definition} \label{def:LaplacianWeights}
 A system of vertex weights $(\mu_x)_{x \in \cV}$ and edge weights $(\omega_{xy})_{\{x,y\} \in \cE}$ on the graph $\cG$ is 
 called \emph{Laplacian (to third order)} at a vertex $x\in \cV$ if, for any linear form $L \in \upT_x^*M$:
 \begin{enumerate}[(1)]
  \item (First-order condition) \label{itemdef:LinearCondition}
   \begin{equation}
  \frac{1}{\mu_x}\sum_{y \sim x} \omega_{xy} \, \overrightarrow{xy}  = 0\,.
 \end{equation}
 \item (Second-order condition) \label{itemdef:QuadraticCondition}
 \begin{equation}
  \frac{1}{\mu_x}\sum_{y \sim x} \omega_{xy} \; L(\overrightarrow{xy})^2 = 2 \Vert L \Vert ^2\,.
 \end{equation}
 \item (Third-order condition) \label{itemdef:CubicCondition}
 \begin{equation}
  \frac{1}{\mu_x}\sum_{y \sim x} \omega_{xy} \; L(\overrightarrow{xy})^3 = 0\,.
  \end{equation} 
  \end{enumerate}
  The biweighted graph $(\cG, (\mu_x), (\omega_{xy}))$ is called \emph{Laplacian} if it is Laplacian at any vertex.
\end{definition}

\noindent Recall that we denote $\overrightarrow{xy} \coloneqq \exp_x^{-1}y \in \upT_xM$.

\begin{remark}
As we shall see, the defining properties of Laplacian weights (or their characterization \autoref{prop:LaplacianWeights})
are remarkably versatile. Perhaps the most obvious motivation for their definition is \autoref{thm:ConvergenceTensionField}, but 
we will also use it in different ways, \eg{} for \autoref{lem:SumWeightsCrystalline} or \autoref{thm:ConvergenceEnergyDensity}. 
\end{remark}

\begin{remark}  \label{rem:LaplacianCenterOfMass}
A biweighted graph being Laplacian to first order, \ie{} satisfying condition \ref{itemdef:LinearCondition}, 
is equivalent to the the fact that each vertex of $\cG$ is the weighted barycenter of its neighbors.
\autoref{thm:ExampleLaplacianFirstOrder} provides many examples of Laplacian graphs to first order.
\end{remark}

\begin{theorem} \label{thm:ExampleLaplacianFirstOrder}
 Assume $M = S$ is $2$-dimensional and has nonpositive curvature. Any biweighted graph $\cG$ underlying a topological triangulation of $S$ admits a unique map 
 to $S$ that is Laplacian to first order, \ie{} whose image graph equipped with the same weights is Laplacian to first order.
\end{theorem}

\begin{proof}
Note that a map $f \colon \cG \to S$ being Laplacian to first order is equivalent to $f$ having zero discrete tension field, \ie{} $f$ being discrete harmonic.
By \cite[Theorem 3.20]{Gaster-Loustau-Monsaingeon1}, the discrete energy functional in this setting is strongly convex, in particular it has a unique critical point.
\end{proof}

The following seemingly stronger characterization of Laplacian weights is immediate:

\begin{proposition} \label{prop:LaplacianWeights}
 A system of weights on $\cG$ is Laplacian at $x \in \cV$ if and only if for any finite-dimensional vector space $W$:
 \begin{enumerate}[(1)]
  \item \label{itemprop:LinearCondition} For any  linear map $L \colon \upT_x M \to W$:
   \begin{equation}
  \sum_{y \sim x} \omega_{xy} \; L(\overrightarrow{xy}) = 0\,.
 \end{equation}
 \item \label{itemprop:QuadraticCondition} For any quadratic form $q$ on $\upT_x M$ with values in $W$:
 \begin{equation}
  \frac{1}{\mu_x}\sum_{y \sim x} \omega_{xy} \; q(\overrightarrow{xy}) = 2 \tr q\,.
 \end{equation}
 \item \label{itemprop:CubicCondition} For any cubic form $\sigma$ on $\upT_xM$ with values in $W$:
  \begin{equation}
  \sum_{y \sim x} \omega_{xy} \;\sigma(\overrightarrow{xy}) = 0\,.
 \end{equation}
  \end{enumerate}
\end{proposition}

Note that we use the metric (inner product) in $\upT_x M$ to define $\tr q$. By definition, $\tr q$ is the trace of the self-adjoint
endomorphism associated to $q$.

\subsection{Preferred vertex weights: the volume weights} \label{subsec:VolumeWeights}

In this paper we favor one system of 
vertex weights associated to any mesh of any Riemannian manifold, the so-called \emph{volume weights}.

For comfort assume $(M, g) = S$ is $2$-dimensional, although what follows is evidently generalized to higher dimensions.
Let $x$ be a vertex of the triangulation and consider the polygon $P_x \subseteq S$ equal to the union of the triangles adjacent to $x$. 
We define the weight of the vertex $x$ by
\begin{equation}
 \mu_x \coloneqq \frac{1}{3} \Area(P_x)
\end{equation}
where $\Area(P_x)$ denotes the Riemannian volume (area) of $P_x$. This clearly defines a system of positive vertex weights $\mu_\cG \coloneqq (\mu_x)_{x \in \cV}$.
We alternatively see $\mu_\cG$ as a discrete measure on $S$ supported by the set of vertices, which is meant to approximate the volume density $v_g$ 
of the Riemannian metric: see \autoref{subsec:ConvergenceVolumeForm}. Note that the choice of the constant $\frac{1}{1 +\dim M} = \frac{1}{3}$ in the definition of $\mu_x$ is motivated by the fact that each triangle is counted $3$ times when integrating over $S$. The next proposition is almost trivial:

\begin{proposition} \label{prop:SumVertexWeights}
 Let $(M,g)$ be a closed manifold with an embedded graph $\cG$ associated to a geodesic mesh. 
 Let $\mu_\cG$ be the discrete measure on $S$ defined by the volume weights. Then
 \begin{equation}
  \sum_{x \in \cV} \mu_x = \int_M \upd \mu_\cG = \int_M \upd v_g = \Vol(M,g)\,.
 \end{equation}
\end{proposition}

Recall that any system of vertex weights endows the space of discrete maps $\Map_{\cG}(M,N)$ with 
an $\upL^2$ distance (see \autoref{subsec:DiscretizationSetup}).

\begin{theorem} \label{thm:InterpolationL2}
Let $N$ be any Riemannian manifold of nonpositive sectional curvature.
Equip the space of discrete maps $\Map_\cG(M, N)$ with the $\upL^2$ distance associated to the volume weights.
Then the center of mass interpolation map $\iota_\cG  \colon  \Map_\cG(M, N) \to \cC(M,N)$
is $L$-Lipschitz with respect to the $\upL^2$ distance on both spaces, with $L = \sqrt{1+\dim M}$.
When $M$ is Euclidean (flat), the Lipschitz constant can be upgraded to $L =1$.
\end{theorem}

\begin{proof}
 Let us assume $M = S$ is $2$-dimensional for comfort. Let $f, g \in \Map_\cG(M, N)$, denote
 by $\widehat{f} \coloneqq \iota_\cG(f)$ and $\widehat{g} \coloneqq \iota_\cG(g)$ their center of mass interpolations.
 By definition of the $\upL^2$ distance on $\cC(M,N)$, 
\begin{equation}
 d(\widehat{f}, \widehat{g})^2 = \int_M d(\widehat{f}(x), \widehat{g}(x))^2 \, \upd v_g(x)~.
\end{equation}
Denote by $\cT$ the set of triangles in the mesh. The integral is rewritten
\begin{equation} \label{eq:InterpolationL2eq1}
 d(\widehat{f}, \widehat{g})^2 = \sum_{T \in \cT} \int_T d(\widehat{f}(x), \widehat{g}(x))^2 \, \upd v_g(x)~.
\end{equation}
Let $T = ABC$ be any triangle in $\cT$. Following the proof of \autoref{thm:CoMOInterpolation} \ref{thm:CoMOInterpolationitemiii}, for all $x \in T$
there exists $\alpha, \beta, \gamma \in [0,1]$ such that $\alpha + \beta + \gamma = 1$ and
\begin{equation}
 d(\widehat{f}(x), \widehat{g}(x)) \leqslant \alpha d(f(A), g(A)) + \beta d(f(B), g(B)) + \gamma d(f(C), g(C))\,.
\end{equation}
By convexity of the square function, it follows
\begin{equation} \label{eq:InterpolationL2eq2}
 d(\widehat{f}(x), \widehat{g}(x))^2 \leqslant \alpha d(f(A), g(A))^2 + \beta d(f(B), g(B))^2 + \gamma d(f(C), g(C))^2
\end{equation}
hence
\begin{equation} \label{eq:InterpolationL2eq3}
 d(\widehat{f}(x), \widehat{g}(x))^2 \leqslant  d(f(A), g(A))^2 + d(f(B), g(B))^2 + d(f(C), g(C))^2\,.
\end{equation}
Therefore we may derive from \eqref{eq:InterpolationL2eq1}
\begin{equation}
\begin{split}
 d(\widehat{f}, \widehat{g})^2 &\leqslant \sum_{T \in \cT} \left[d(f(A), g(A))^2 + d(f(B), g(B))^2 + d(f(C), g(C))^2\right] \Area(T)\\
 &\leqslant \sum_{x \in \cV} \sum_{T \in \cT_x} d(f(x), g(x))^2 \, \Area(T_x)
\end{split}
\end{equation}
where $\cT_x$ denotes the set of triangles adjacent to $x$. Finally this is rewritten
\begin{equation}
 d(\widehat{f}, \widehat{g})^2 \leqslant \sum_{x \in \cV} 3 \mu_x \, d(f(x), g(x))^2
\end{equation}
where $\mu_x$ is the volume weight at $x$, \ie{}  $d(\widehat{f}, \widehat{g})^2 \leqslant 3 d(f, g)^2$.

If $M$ is Euclidean (flat), the proof can be upgraded to obtain a Lipschitz constant $L = 1$
by keeping the finer estimate \eqref{eq:InterpolationL2eq2} instead of \eqref{eq:InterpolationL2eq3}, and computing the triangle integral.
\end{proof}

\subsection{Preferred edge weights: the cotangent weights}
\label{subsec:CotangentWeights}

We also have a favorite system of edge weights, the so-called \emph{cotangent weights}, although they have the following restrictions:
\begin{enumerate}[(1)]
 \item We only define them for $2$-dimensional Riemannian manifolds, 
 though they have higher-dimensional analogs.
 \item They are only positive for triangulations having the ``Delaunay angle property''. (This includes any acute triangulation.)
\end{enumerate}

These weights have a simple definition in terms of the cotangents of the (Riemannian) angles between edges in the triangulation, and coincide with the weights
of Pinkall-Polthier \cite{MR1246481} in the Euclidean case. For more background on the cotangent weights in the Euclidean setting and a formula for their higher-dimensional
analogs, please see \cite{KeenanCraneCotangent}.

The following result noticed by Pinkall-Polthier \cite{MR1246481} is an elementary exercise of plane Euclidean geometry:
\begin{lemma} \label{lem:Pinkall-Polthier}
Let $T = ABC$ and $T' = A'B'C'$ be triangles in the Euclidean plane. Denote by $f \colon T \to T'$ the unique affine map
such that $f(A) = A'$, etc. Then the energy of $f$ is given by
\begin{equation}
 \begin{split}
  E(f) &\coloneqq \frac{1}{2}\int_T \Vert \upd f \Vert^2 \upd v\\
  &= \frac{1}{4} \left(a'^2 \cot \alpha + b'^2 \cot \beta + c'^2 \cot \gamma  \right)
 \end{split}
\end{equation}
where $\alpha$, $\beta$, $\gamma$ denote the unoriented angles of the triangle $ABC$ 
and $a'$, $b'$, $c'$ denote the side lengths of the triangle $A'B'C'$ as in \autoref{fig:triangleMap}.
\end{lemma}

\begin{figure}[ht]
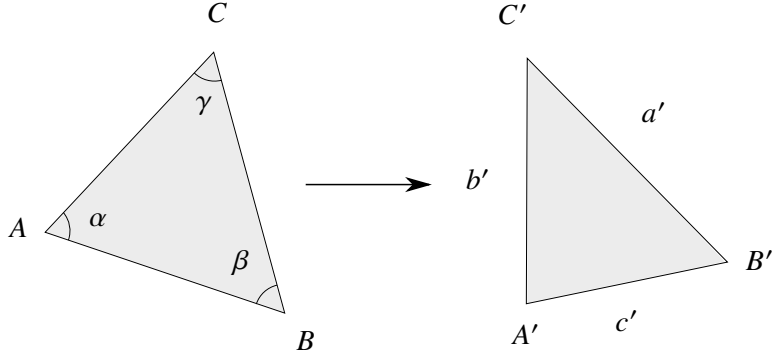

\centering
\vspace{.5cm}
\begin{lpic}{triangleMap(9cm)}
	\lbl[]{40,70;$\alpha$}
	\lbl[]{-20,64;$A$}
	\lbl[]{145,38;$\beta$}
	\lbl[]{193,-20;$B$}
	\lbl[]{118,155;$\gamma$}
	\lbl[]{128,224;$C$}
	\lbl[]{320,100;$b'$}
	\lbl[]{355,-15;$A'$}
	\lbl[]{450,150;$a'$}
	\lbl[]{528,40;$B'$}
	\lbl[]{430,-5;$c'$}
	\lbl[]{346,223;$C'$}
\end{lpic}
\vspace{.5cm}
\caption{A triangle map in $\R^2$.}
\label{fig:triangleMap}
\end{figure}


In view of \autoref{lem:Pinkall-Polthier}, given a surface $(S,g)$ equipped with a geodesic mesh,
we define the weight of an edge $e$ by considering the two angles $\alpha$ and $\beta$ opposite to $e$ in the two triangles adjacent to $e$
(see \autoref{fig:edgeWeight}), and we put 
\begin{equation} \label{eq:CotangentWeight}
\omega_e \coloneqq \frac{1}{2}(\cot \alpha + \cot \beta)\,. 
\end{equation}

\begin{figure}[ht]
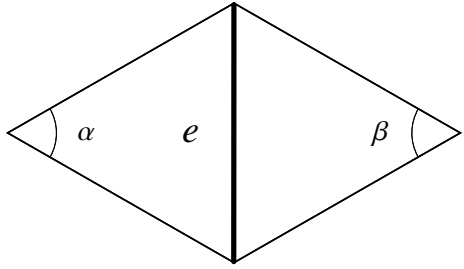

\centering
\begin{lpic}{edgeWeight(6cm)}
	\lbl[]{35,57;$\alpha$}
	\lbl[]{162,57;$\beta$}
	\Large
	\lbl[]{80,57;$e$}
\end{lpic}
\caption{The weight $\omega_e$ of the edge $e$ is defined in terms of the opposite angles $\alpha$ and $\beta$.}
\label{fig:edgeWeight}
\end{figure}
Note that we use the Riemannian metric $g$ to define the geodesic edges
of the graph and the angles between edges. 

\begin{definition}
Let $(S,g)$ be a Riemannian surface equipped with a geodesic mesh with underlying graph $\cG$.
The edge weights on $\cG$ defined as in \eqref{eq:CotangentWeight} are the system of \emph{cotangent weights}.
\end{definition}

As a direct application of \autoref{lem:Pinkall-Polthier}, we obtain:
\begin{proposition} \label{prop:CotangentWeightsExactPL}
 Let $(S,g)$ be a flat surface with a geodesic mesh. Let $\cG$ be the underlying graph equipped 
 with the cotangent edge weights. For any piecewise affine map $f \colon S \to \R^n$,
 the smooth energy $E(f) \coloneqq \frac{1}{2}\int_S \Vert \upd f \Vert^2 \upd v$ coincides with the discrete energy $E_\cG(f)$
 defined in \eqref{eq:DiscreteEnergyFunctional}.
\end{proposition}

Note that a priori, the cotangent weights are not necessarily positive. Clearly, they are positive for acute triangulations (all of whose triangles are acute).
More generally, the cotangent weights are positive if and only if the triangulation has the property that, for any edge $e$, the two opposite angles add to less than $\pi$. This is simply because 
\begin{equation}
\omega_e = \frac{1}{2}(\cot \alpha + \cot \beta) = \frac{\sin(\alpha + \beta)}{2 \sin \alpha \sin \beta}\,. 
\end{equation}
We call this the \emph{Delaunay angle property}. 
In the Euclidean setting (for a flat surface), this property is equivalent to the triangulation being \emph{Delaunay}, \ie{} the circumcircle of any triangle does not contain any vertex in its interior 
\cite[Lemma 9, Prop. 10]{MR2365833}. 
%

\subsection{Laplacian qualities of cotangent weights}

In the $2$-dimensional Euclidean setting, in addition to \autoref{prop:CotangentWeightsExactPL}, the cotangent weights enjoy some good--and other not so good--Laplacian properties, although this is much less obvious.

\begin{proposition}
\label{prop:EuclideanCotangentWeightsLaplacianFirstOrder}
Suppose that $(S,g)$ is a flat surface. Then the cotangent weights associated to any triangulation of $S$ are Laplacian to first order.
\end{proposition}

\begin{proof}
Let $x$ be a vertex and consider the polygon $P = P_x$ equal to the union of the triangles adjacent to $x$. Since in the flat case the exponential map
$\exp_x$ is a local isometry, without loss of generality we can assume that $P$ is contained in the Euclidean plane $\upT_x S \approx \R^2$ and $x = O$.

Suppose that the vertices of $P$ are given in cyclic order by $(A_i)$, and that we have angles 
$\alpha_i$, $\beta_i$, $\gamma_i$ as in 
\autoref{fig:trianglesAtO}.
By definition, the weight of the edge $O A_i$ is given by $\omega_i \coloneqq \frac{1}{2}\left(\cot \beta_{i-1} + \cot \gamma_i\right)$.

\begin{figure}[ht]
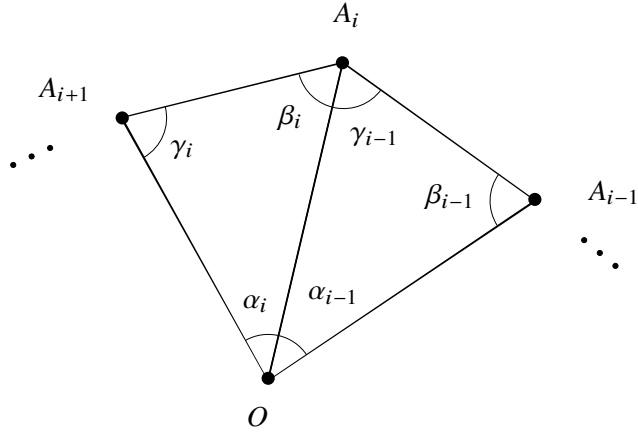

\centering
\bigskip \bigskip
\begin{lpic}{trianglesAtO(8cm)}
	\lbl[]{115,-15;$O$}
	\lbl[]{25,138;$A_{i+1}$}
	\lbl[]{156,173;$A_i$}
	\lbl[]{280,90;$A_{i-1}$}
	\lbl[]{114,38;$\alpha_i$}
	\lbl[]{80,110;$\gamma_i$}
	\lbl[]{130,125;$\beta_i$}
	\lbl[]{169,116;$\gamma_{i-1}$}
	\lbl[]{204,88;$\beta_{i-1}$}
	\lbl[]{150,42;$\alpha_{i-1}$}
\end{lpic}
\vspace{.5cm}
\caption{The triangles of $P$ at $O$.}
\label{fig:trianglesAtO}
\end{figure}

Now consider the identity map $f \colon P \to \R^2$. It has constant energy density $e(f) = 2$, therefore the total
energy of $f$ is $E = 2 \Area(P)$. On the other hand, $E$ is the sum of the energies of $f$ in restriction to the triangles forming $P$.
By \autoref{lem:Pinkall-Polthier} this is
\begin{equation} \label{eq:EnergyPolygon}
 E = \frac{1}{4} \sum_i  \left[\cot \alpha_i \Vert \overrightarrow{A_i A_{i+1}} \Vert^2 + \cot \beta_i \Vert \overrightarrow{O A_{i+1}} \Vert^2 +
 \cot \gamma_i \Vert \overrightarrow{ O A_i} \Vert^2 \right]\,.
\end{equation}
So far we assumed that $O$ is the origin in $\R^2$, but of course the argument is valid if $O$ is any point. In fact, let us see the energy $E$ above as a function
of $O \in \R^2$ when all the other points $A_i \in \R^2$ are fixed. We compute the infinitesimal variation of $E$ under a variation $O$. On the one hand,
$\dot{E}(O) = 0$ since $E(O) = 2 \Area(P)$ is constant. On the other hand, \eqref{eq:EnergyPolygon} yields
\begin{equation} \label{eq:PolygonEnergyVariation}
\begin{split}
\dot{E}(O) &= -\frac{1}{4} \sum_i  \left[\frac{\dot{\alpha}_i}{\sin^2 \alpha_i} \Vert \overrightarrow{A_i A_{i+1}} \Vert^2 
+ \frac{\dot{\beta}_i}{\sin^2 \beta_i} \Vert \overrightarrow{O A_{i+1}} \Vert^2
+ \frac{\dot{\gamma}_i}{\sin^2 \gamma_i} \Vert \overrightarrow{ O A_i} \Vert^2 \right]\\
&\quad - \frac{1}{2} \sum_i \left\langle \dot{O} \, , \, \cot \beta_i \overrightarrow{O A_{i+1}} + \cot \gamma_i \overrightarrow{O A_{i}} \right\rangle\,.
\end{split}
\end{equation}
We claim that the first sum in \eqref{eq:PolygonEnergyVariation} vanishes. Indeed, first observe that the law of sines yields
\begin{equation} 
\frac{\Vert \overrightarrow{A_i A_{i+1}} \Vert^2 }{\sin^2 \alpha_i} = \frac{\Vert \overrightarrow{O A_{i+1}} \Vert^2}{\sin^2 \beta_i}
= \frac{ \Vert \overrightarrow{ O A_i} \Vert^2}{\sin^2 \gamma_i} = \frac{1}{D^2}
\end{equation}
where $D$ is the diameter of the triangle $O A_i A_{i+1}$'s circumcircle, so the first sum is rewritten
\begin{equation}
\begin{split}
\sum_i \left[ \frac{1}{D^2} \left(\dot{\alpha}_i + \dot{\beta}_i + \dot{\gamma}_i\right) \right]
\end{split}
\end{equation}
and $\dot{\alpha}_i + \dot{\beta}_i + \dot{\gamma}_i = 0$ since $\alpha_i + \beta_i + \gamma_i = \pi$ is constant.
Thus  \eqref{eq:PolygonEnergyVariation} is rewritten
\begin{equation} \label{eq:PolygonEnergyVariation2}
\begin{split}
\dot{E}(O) &= - \frac{1}{2} \sum_i \left\langle \dot{O} \, , \, \cot \beta_i \overrightarrow{O A_{i+1}} + \cot \gamma_i \overrightarrow{O A_{i}} \right\rangle\\
&=- \left\langle \dot{O} \, , \,\sum_i \omega_i \overrightarrow{O A_{i}} \right\rangle\,.
\end{split}
\end{equation}
In other words: $\grad E(O) = - \sum_i \omega_i \overrightarrow{O A_{i}}$. Since this must be zero (recall that $E(O)$ is constant), $O$ is indeed the barycenter of its 
weighted neighbors $\{A_i, \omega_i\}$.
\end{proof}

It is not true in general that cotangent weights are Laplacian to second order. 
However, for triangulations obtained by midpoint refinement, it is true for almost all vertices:

\begin{proposition}
\label{prop:EuclideanCotangentWeightsLaplacianSecondOrder}
Let $(S,g)$ be a flat surface. Let $(\cG_n)_{n\in \N}$ be a sequence of graphs obtained by iterated midpoint subdivision from a given initial triangulation.
Equip $\cG_n$ with the area vertex weights and cotangent edge weights. Then $\cG_n$ satisfies the 2nd-order Laplacian condition at any vertex except maybe at the vertices of
of $\cG_0$. 
\end{proposition}
The proof is based on the observation that any vertex of $\cG_n$ is either an \emph{initial vertex} (vertices of $\cG_0$),
 a \emph{boundary vertex} (vertices that are located on edges of the initial triangulation) or an \emph{interior vertex} (all other vertices),
and that the latter two satisfy a strong symmetry condition, which we call (semi-)hexaparallel symmetry:

\begin{definition} \label{def:Hexaparallel}
Consider a vertex $x$ with valence six in a Euclidean graph. 
\begin{itemize}[$\bullet$]
 \item We say that $x$ has \emph{hexaparallel symmetry} if
the set of vectors $\{\vec{xy} ~\colon~ y \sim x\}$ is in the $\GL(2,\bR)$-orbit of $\{\pm(1,0),\pm(1,1),\pm(0,1)\}$.
Equivalently, the neighbors of $x$ are the vertices of a hexagon whose opposite sides are pairwise parallel and of the same length. 
See \autoref{subfig:hexaparallel}. 
\item We say that $x$ has \emph{semi-hexaparallel symmetry} if the neighbors may be cyclically labeled 
$\{y_1, \dots, y_6\}$ and divided into two overlapping sets $\{y_1, y_2, y_3, y_4\}$ and $\{y_4, y_5, y_6, y_1\}$, each being part of a potential hexaparallel configuration.
See \autoref{subfig:semihexaparallel}. 
\end{itemize}
\end{definition}

\begin{figure}[!ht]
\centering
\setlength{\lineskip}{\medskipamount}
\subcaptionbox{Hexaparallel configuration.\label{subfig:hexaparallel}}{\includegraphics[width=0.45\textwidth]{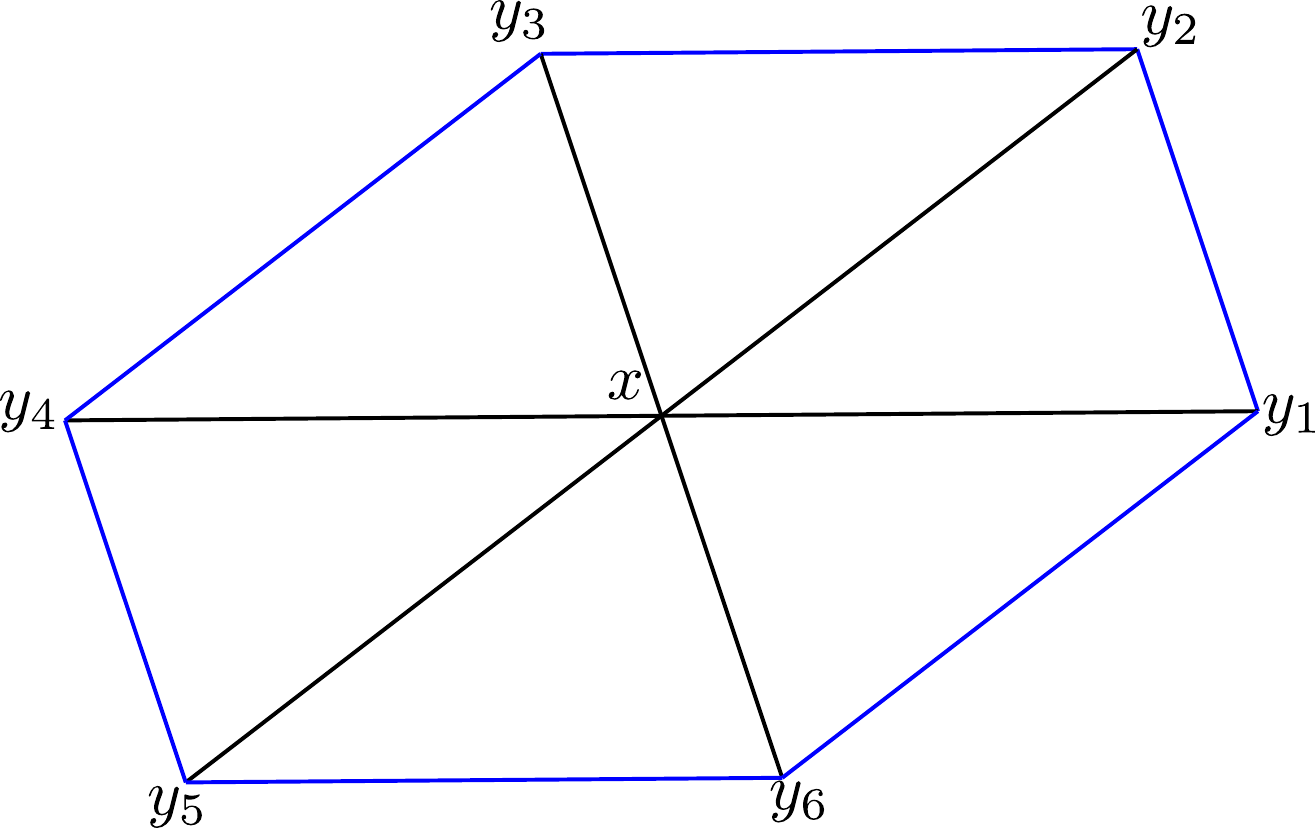}} \hspace{0.02\textwidth}
\subcaptionbox{Semi-hexaparallel configuration.\label{subfig:semihexaparallel}}{\includegraphics[width=0.45\textwidth]{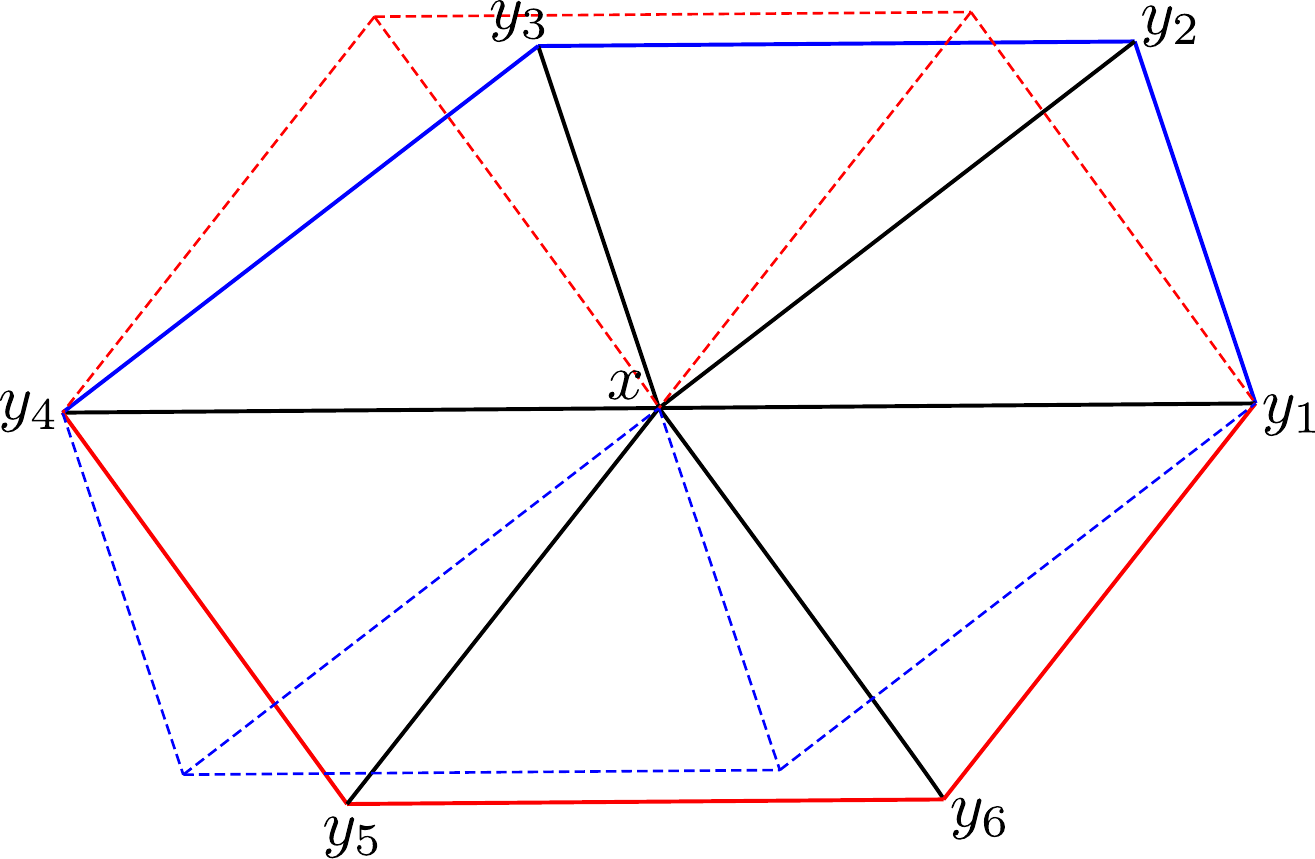}}
\caption{Hexaparallel and semi-hexaparallel symmetry.} \label{fig:hexaparallel}
\end{figure}


It is straightforward to check by induction that a plane Euclidean graph obtained by iterated midpoint subdivision
is hexaparallel at any interior vertex and semi-hexaparallel at any boundary vertex. Thus \autoref{prop:EuclideanCotangentWeightsLaplacianSecondOrder} reduces to:

\begin{lemma}
\label{lem:HexaparrallelVertices}
Any geodesic graph $\cG$ in $\R^2$ equipped with the area vertex weights and cotangent edge weights 
satisfies the second-order Laplacian condition at any (semi-)hexaparallel vertex $x$.
\end{lemma}

\begin{proof}
We need to show the second-order condition: for any quadratic form $q$ on $\R^2$,
\begin{equation} \label{eq:lemHXquad}
\frac{1}{\mu_x} \sum_{y\sim x} \omega_{xy} \ q(y-x) = 2 \ \tr q\,.
\end{equation}
First we argue that the semi-hexaparallel case derives from the hexaparallel case. Note that the left-hand side of \eqref{eq:lemHXquad} is invariant
by the central symmetry at $x$, since a quadratic function is even. If $x$ has semi-hexaparallel symmetry, we can create two hexaparallel configurations
as in \autoref{subfig:semihexaparallel}, both satisfying \eqref{eq:lemHXquad}. Taking the half-sum of the two equations then yields the desired result.

Assume from now on that $x$ has hexaparallel symmetry. Denote $y_1, \dots, y_6$ the neighbors in cyclic order.
We may choose a complex coordinate on $\R^2 \approx \C$ so that $x=0$ and $y_1 = 1$. Denote $z=a+bi$ the coordinate of $y_2$.
The hexaparallel condition implies that $y_3 = z-1$, $y_4 = -1$, $y_5 = -z$, and $y_6 = 1-z$. 
Let the oriented angles $\angle(y_1,y_2)$, $\angle(y_2,y_3)$, and $\angle(y_3,y_4)$ be denoted by $\alpha$, $\beta$, and $\gamma$, respectively.
For any $w \in \C$, we have $\cot(\arg w) = \frac{\operatorname{Re}(w)}{\operatorname{Im}(w)}$. Therefore we may compute:
\begin{equation}
\begin{split}
\cot \alpha &= \frac{\operatorname{Re}}{\operatorname{Im}}(z)=\frac a b \\
\cot \beta &= \frac{\operatorname{Re}}{\operatorname{Im}}\left(\frac{z-1}{z}\right) = \frac {a^2+b^2-a}{b} \\
\cot \gamma &= \frac{\operatorname{Re}}{\operatorname{Im}}\left(\frac{1}{1-z}\right) = \frac{1-a}{b}\,.
\end{split}
\end{equation}
Since $\mu_x =\frac13 ( 6\cdot b/2)=b$, we get
\begin{equation}
\begin{split}
\frac {1}{\mu_x} \sum_{y\sim x} \omega_{xy} \ q(y-x) &= \frac 2{b} \left( \cot \alpha \cdot q(z-1) + \cot \beta \cdot q(1) + \cot \gamma \cdot q(z) \right) \\
& = \frac 2b \left( \frac ab \cdot q(z-1) + \frac{a^2+b^2-a}{b} \cdot q(1) + \frac{1-a}{b} \cdot q(z) \right)\,.
\end{split}
\end{equation}
The latter is equal to $2$, $0$, and $2$ when $q=\upd x^2$, $\upd x\upd y$, or $\upd y^2$, respectively, as desired.
\end{proof}

\begin{corollary}
Suppose that $(S,g)$ is a flat surface. Let $(\cG_n)_{n\in \N}$ be a sequence of graphs obtained by iterated midpoint subdivision of an initial triangulation $\cG_0$. Equip $\cG_n$ with the area vertex weights and the cotangent edge weights. Then $\cG_n$ is Laplacian at any interior vertex.
\end{corollary}

\begin{proof}
The first-order and third-order conditions are trivial due to central symmetry of the neighbors around the vertex $x$ and the fact that linear and cubic functions
are odd. (Alternatively, the first-order condition holds by \autoref{prop:EuclideanCotangentWeightsLaplacianFirstOrder}.) The second-order condition holds
by \autoref{lem:HexaparrallelVertices}.
\end{proof}

\begin{remark}
We shall see in \autoref{sec:Construction} that in the general Riemannian setting, the cotangent weights will satisfy similar
Laplacian properties asymptotically for very fine meshes.
\end{remark}

\begin{remark}
While being the best choice of edge weights, the cotangent weights generally do not satisfy the second-order Laplacian condition at vertices with no (semi-)hexaparallel symmetry. 
Taking finer and finer triangulations will not help with this defect. 
At such vertices, which generically exist for topological reasons, the discrete Laplacian of a smooth function 
can \emph{not} be expected to approximate its Laplacian.
This is somewhat unsettling, but it is an intrinsic difficulty to the discretization of the Laplacian.
Providing suitable assumptions that neverthless guarantee convergence of discrete harmonic maps to smooth harmonic maps is the central aim of this paper.
\end{remark}

\section{Sequences of meshes}
\label{sec:SequencesOfMeshes}

In this section, we enhance the previous section by considering sequences of meshes on a Riemannian manifold $(M,g)$. 
The idea is to capture the local geometry of $M$ sufficiently well provided the mesh is sufficiently fine.
This allows a relaxation of the Laplacian weights conditions, which are too stringent for a fixed mesh of an arbitrary Riemannian manifold.
We introduce the notions of \emph{asymptotically Laplacian} and \emph{almost asymptotically Laplacian} systems of weights, with the aim that these weakened conditions 
can still be used to demonstrate the convergence theorems we are after.

\subsection{Fine and crystalline sequences of meshes}

Let $(\cM_n)_{n \in \N}$ be a sequence of geodesic meshes of a Riemannian manifold $(M,g)$. Denote by $r_n$ the ``mesh size'', \ie{} the longest edge length of $\cM_n$. 
Following \cite{AnalysisSitusPL}, we define:

\begin{definition}
 The sequence $(\cM_n)_{n \in \N}$ is called \emph{fine} provided $\lim\limits_{n \to +\infty} r_n = 0$.
\end{definition}
\begin{notation}
For the remainder of the paper, we drop the subscript $r \coloneqq r_n$ for ease in notation.
\end{notation}

Given a bounded subset $D \subseteq M$, one calls:
\begin{itemize}
 \item \emph{diameter of $D$} the supremum of the distance between two points of $D$, denoted $\diam(D)$.
 \item \emph{radius of $D$} the distance from the center of mass of $D$ to its boundary, denoted $\radius(D)$.
 \item \emph{thickness of $D$} the ratio of its radius and diameter, denoted $\thick(D)$:
 \begin{equation}
\thick(D) \coloneqq \frac{\radius(D)}{\diam(D)}\,.
 \end{equation}
\end{itemize}

\begin{definition}
 The sequence $(\cM_n)_{n \in \N}$ is called \emph{crystalline} if there exists a uniform lower bound for the thickness of simplices in $\cM_n$.
\end{definition}

\begin{example}
In \autoref{thm:SubdivisionCrystalline}, we will show that any sequence of meshes obtained by midpoint subdivision is fine and crystalline, a crucial
fact for the strategy of this paper.
\end{example}

\begin{proposition} \label{prop:CrystallineAngles}
 Let $(\cM_n)_{n \in \N}$ be a fine sequence of meshes. The following are equivalent:
 \begin{enumerate}[(i)]
  \item \label{item:CrystallineAnglesi} The sequence $(\cM_n)_{n \in \N}$ is crystalline.
  \item \label{item:CrystallineAnglesii}There exists a uniform positive lower bound for all angles between adjacent edges in $\cM_n$.
  \item \label{item:CrystallineAnglesiii} There exists a uniform positive lower bound for the ratio of any two edge lengths in $\cM_n$.
 \end{enumerate}
\end{proposition}

\begin{proof}[Proof sketch]
For brevity, we only sketch the proof; the detailed proof would include proper Riemannian estimates: 
see \autoref{sec:RiemannianEstimates}.

First one checks that \ref{item:CrystallineAnglesi} $\Leftrightarrow$ \ref{item:CrystallineAnglesii} in the Euclidean setting. This is an elementary calculation: for a single triangle (or $n$-simplex),
one can bound its radius in terms of its smallest angle. One then generalizes to an arbitrary Riemannian manifold $M$ by arguing that a very small triangle (or $n$-simplex) in $M$ has almost the same radius and angles as its Euclidean counterpart in a normal chart. The fact that we only consider fine sequences of meshes means that we can assume that all simplices are arbitrarily small, making the previous argument conclusive. The proof of \ref{item:CrystallineAnglesii} $\Leftrightarrow$ \ref{item:CrystallineAnglesiii} is conducted similarly.
\end{proof}

\begin{theorem} \label{thm:CrystallineProperties}
Assume that $M$ is compact and the sequence of meshes $(\cM_n)_{n \in \N}$ on $M$ is fine and crystalline. 
Denote by $\cG_n$ the graph underlying $\cM_n$ and $r = r_n$ its maximum edge length.
\begin{enumerate}[(i)]
 \item \label{item:CrystallinePropertiesi} The volume vertex weights $\mu_{x, n}$ of $\cG_n$ are $\Theta\left(r^{\dim M}\right)$ (uniformly in $x$).
 \item \label{item:CrystallinePropertiesii} The number of vertices of $\cG_n$ is $\left| \cV_n \right| = \Theta\left(r^{-\dim M}\right)$. More generally, the number of $k$-simplices of
 $\cG_n$ is $\Theta(r^{-\dim M})$.
 \item \label{item:CrystallinePropertiesiii} The combinatorial diameter $\diam \cG_n$ of the graph $\cG_n$ is $\Theta\left(r^{-1}\right)$.
 \item \label{item:CrystallinePropertiesiv} The combinatorial 
surjectivity radius $\surj \rad \cG_n$ (see below) of the graph $\cG_n$ is $\Theta\left(r^{-1}\right)$.
\end{enumerate}
\end{theorem}

The 
\emph{surjectivity radius at a vertex} $x$ of a graph $\cG$ 
is the smallest integer $k \in \N$ such that 
there exists a vertex at combinatorial distance $k$ from $x$ all of whose neighbors are at combinatorial distance $\leqslant k$ from $x$. 
The 
\emph{surjectivity radius of the graph} $\cG$, denoted $\surj\rad \cG$, is the minimum of its surjectivity radii over all vertices.

\begin{notation} \label{not:BigONotations}
 In this paper, we use the notation $f = \bigO(g)$ and $f = \littleo(g)$ in the usual sense, we use the notation $f = \Omega(g)$ for $g = \bigO(f)$, 
 and $f = \Theta(g)$ for [$f = \bigO(g)$ and $f = \Omega(g)$].
\end{notation}

\begin{proof}[Proof of \autoref{thm:CrystallineProperties}]
 For \ref{item:CrystallinePropertiesi}, recall that the volume vertex weight at $x$ is the sum of the volumes of the simplices adjacent to $x$ (divided by $\dim M$).
 Since the sequence is fine, the diameter of all simplices is going to $0$ uniformly in $x$. On first approximation, the volume of any such vertex is approximately
 equal to its Euclidean counterpart (say, in a normal chart). Since the lengths of all edges are within $[\alpha r, r]$ for some constant $\alpha >0$
 and all angles are bounded below by \autoref{prop:CrystallineAngles}, this volume is $\Theta(r^{\dim M})$.
 
 For \ref{item:CrystallinePropertiesii}, simply notice that $\sum_{x \in \cV_n} \mu_{x, n} = \Vol(M)$ by \autoref{prop:SumVertexWeights} and use 
 \ref{item:CrystallinePropertiesi}. The generalization to $k$-simplices is immediate since the total number of $k$-simplices is clearly 
 $\Theta\left(\left|\cV_n\right|\right)$.
 
 For \ref{item:CrystallinePropertiesiii}, let us first show that $\diam \cG_n = \Omega\left(r^{-1}\right)$.
 Let $x$ and $y$ be two fixed points in $M$ and denote $L$ the distance between them. For all $n\in \N$, there exists vertices $x_n$ and $y_n$ in $\cV_n$
 that are within distance $r$ of $x$ and $y$ respectively, so their distance in $M$ is $d(x_n, y_n) \geqslant L - 2r$. Denoting
 $k_n$ the combinatorial distance between $x_n$ and $y_n$, one has $d(x_n, y_n) \leqslant k_n r$ by the triangle inequality. We thus find
 that $k_n r \geqslant L - 2 r$, hence $\diam \cG_n \geqslant k_n \geqslant L r^{-1} - 2$ so that $\diam \cG_n = \Omega\left(r^{-1}\right)$.
 Finally, let us show that $\diam \cG_n = \bigO\left(r^{-1}\right)$. Let $x_n$ and $y_n$ be two vertices that achieve $\diam \cG_n$.
 Let $\gamma_n$ be a length-minimizing geodesic from $x_n$ to $y_n$. Of course, the length of $\gamma_n$ is bounded above by the diameter of $M$.
 There is a sequence of simplices $\Delta_1, \dots, \Delta_{k_n}$ such that $x\in\Delta_1$, $y\in \Delta_{k_n}$, and any two consecutive simplices are adjacent.
 Since the valence of any vertex is uniformly bounded (because of a lower bound on all angles), the number of simplices
 within a distance $\leqslant r_{\min}$ of any point of $M$ is bounded above by a constant $C$. This implies  
 $k_n \leqslant C L(\gamma)/r_{\min}$, so that $k_n \leqslant C (\diam M) \alpha \, r^{-1}$. 
 Following edges along the simplices $\Delta_i$, one finds a path of length $(\dim M - 2) k_n$ 
 from $x$ to $y$, therefore $\diam \cG_n \leqslant (\dim M -2) C (\diam M) \alpha \, r^{-1}$.
 
For the proof of \ref{item:CrystallinePropertiesiv}, the injectivity radius of $M$ provides a lower bound for $\surj\rad \cG_n$ of the form $\Omega(r^{-1})$, and $\diam \cG_n$ provides an upper bound. The details are left to the reader.
\end{proof}

For a continuous map $f \colon M \to \R$, denote $f_n \coloneqq \pi_n(f)\in \Map_{\cG_n}(M, N)$ the discretization of $f$: this is just the restriction of $f$ to the vertex set of $\cG_n$. As in \cite{AnalysisSitusPL} we have:

\begin{lemma} \label{lem:CVFunctionsCrystalline}
 If $(\cM_n)_{n\in\N}$ is a sequence of meshes that is fine and crystalline, then for any piecewise smooth function $f \colon M \to \R$, 
 the center of mass interpolation $\widehat{f_n}$ converges to $f$ for the piecewise $\cC^1$ topology.
\end{lemma}

\begin{proof}[Proof sketch]
As for \autoref{prop:CrystallineAngles}, the proof can be conducted in two steps: First in the Euclidean setting,
where the center of mass interpolation $\widehat{f_n}$ is just the piecewise linear approximation of $f_n$. This proof is done in \eg{} \cite{AnalysisSitusPL}.
One then generalizes to an arbitrary Riemannian manifold $M$ by arguing that for very fine triangulations, the center of mass interpolation $\widehat{f_n}$
is very close to the piecewise linear approximation of $f_n$ in a normal chart.
\end{proof}

\begin{remark}
Any interpolation scheme satisfying the conclusion of \autoref{lem:CVFunctionsCrystalline}, as well as \autoref{thm:CoMOInterpolation} 
and \autoref{thm:InterpolationL2} (or asymptotic versions thereof), would make the machinery work to prove our upcoming main theorems. One could therefore enforce
these properties as the definition of a good sequence of interpolation schemes.
\end{remark}

\begin{corollary} \label{cor:CVFunctionsCrystallineCompact}
 Let $f \colon M \to N$ be a $\cC^1$ map between Riemannian manifolds. Assume that $M$ is compact and equipped with a fine and crystalline
 sequence of meshes $(\cM_n)_{n\in \N}$. The center of mass interpolation $\widehat{f_n}$ converges to $f$ in $\upL^\infty(M,N)$ and $E(f) = \lim_{n\to +\infty} E(\widehat{f_n})$.
\end{corollary}

\begin{remark} \label{rem:Sobolev}
 One would like to say that $\widehat{f_n}$ converges to $f$ in the Sobolev space $\upH^1(M,N)$, but this space is not well-defined. 
 Actually, $\upH^1(M,N)$ may be defined as the subspace of $\upL^2(M,N)$ consisting of $\upL^2$ maps with finite energy, but it is unclear how to define the $\upH^1$ topology.
 Nevertheless we can say something in that direction: $\widehat{f_n} \to f$ in $\upL^2(M,N)$ and $E(\widehat{f_n}) \to E(f)$. One should think of the energy
 as the $\upL^2$ norm of the derivative, but this ``norm'' does not induce a distance.
\end{remark}

The following lemma will be useful in \autoref{subsec:ConvergenceDiscreteTensionField} and again in \autoref{subsec:ConvergenceEnergy}.

\begin{lemma}
\label{lem:SumWeightsCrystalline}
Assume that the sequence of meshes $(\cM_n)_{n \in \N}$ on $M$ is fine and crystalline. 
Let $\cG_n$ be the graph underlying $\cM_n$ and $r = r_n$ its maximum edge length.
If $\cG_n$ is equipped with a system of vertex and edge weights
that is Laplacian at some vertex $x$, then
\begin{equation}
\frac{1}{\mu_x}\sum_{y\sim x}\omega_{xy} = \bigO\left(r^{-2}\right)\,.
\end{equation}
\end{lemma}

\begin{remark}
 For ease of notation, we drop the dependence in $n$ when writing $\mu_x$ and $\omega_{xy}$ above.
\end{remark}

\begin{remark}
 Before writing the proof, let us clarify the quantifiers in  \autoref{lem:SumWeightsCrystalline} (as well as 
\autoref{thm:ConvergenceTensionField} and \autoref{thm:ConvergenceEnergyDensity}): The statement is that there exists a constant $M>0$ independent of $n$
such that at any vertex $x$ of $\cG_n$ where the system of weights is Laplacian, $\frac{1}{\mu_x}\sum_{y\sim x}\omega_{xy} \leqslant M r^{-2}$.
\end{remark}

\begin{proof}
Apply condition \ref{itemprop:QuadraticCondition} of \autoref{prop:LaplacianWeights} to the quadratic form $q = \Vert \cdot \Vert^2$:
\begin{equation} \label{eq:SumW1}
\frac{1}{\mu_x}\sum_{y\sim x}\omega_{xy} \, d(x,y)^2 = 2m
\end{equation}
where $m = \dim M$.
The fact that the sequence of meshes is fine and crystalline implies that there exists a uniform lower bound for the ratio of lengths in the triangulation.
Thus there exists a constant $\alpha > 0$ such that for  any neighbor vertices $x$ and $y$ in $\cG_n$:
\begin{equation} \label{eq:SumW2}
 \alpha \, r \leqslant d(x,y) \leqslant r\,.
\end{equation} 
It follows from \eqref{eq:SumW1} and \eqref{eq:SumW2} that 
\begin{equation}
\frac{1}{\mu_x}\sum_{y\sim x}\omega_{xy} \leqslant \frac{2m}{\alpha^2 r^2}\,.
\end{equation}
\end{proof}

\subsection{Convergence of the volume form} \label{subsec:ConvergenceVolumeForm}

Let $(M,g)$ be a Riemannian manifold, 
let $(\cM_n)_{n \in \N}$ be a sequence of meshes 
with the underlying graphs  $(\cG_n)_{n \in \N}$. We equip $\cG_n$ with the volume vertex weights defined in \autoref{subsec:VolumeWeights}. 
These define a discrete measure $\mu_n$ on $M$ supported by the set of vertices $\cV_n = \cG_n^{(0)}$.

\begin{theorem} \label{thm:ConvergenceVolumeWeights}
If $M$ is any Riemannian manifold and $(\cM_n)_{n \in \N}$ is any fine sequence of meshes, then the measures $(\mu_n)_{n\in\N}$ on $M$
defined by the volume vertex weights
converge weakly-* to the volume density on $M$:
\begin{equation}
 \int_M f \, \upd \mu_n \stackrel{n\to +\infty}{\longrightarrow} \int_M f \, \upd \mu 
\end{equation}
for any $f \in \cC_c^0(M, \R)$ (continuous function with compact support), 
where $\mu$ denotes the measure on $M$ induced by the volume form $v_g$.
\end{theorem}

\begin{proof}
Recall that a \emph{continuity set} $A \subseteq M$ is a Borel set such that $\mu(\partial A) = 0$. Since
any compact set has finite $\mu$-measure, it is well-known that the weakly-* convergence of $\mu_n$ to $\mu$ is equivalent to 
\begin{equation} \label{eq:ConvergenceContinuitySet}
  \mu_n(A) \stackrel{n\to +\infty}{\longrightarrow} \mu(A)
\end{equation}
for any bounded continuity set $A$. Let thus $A$ be any bounded continuity set. Denote by
$B_n$ the union of all simplices that are entirely contained in $A$, and by $C_n$ the union of all simplices that have at least one vertex in $A$.
We obviously have $B_n \subseteq A \subseteq C_n$, and by definition of $\mu_n$ we have:
\begin{equation} \label{eq:ComparisonMeasures}
 \mu(B_n) \leqslant \mu_n(A) \leqslant \mu(C_n)
\end{equation}
On the other hand, clearly we have $C_n - B_n \subseteq N_{\varepsilon_n}(\partial A)$, where we have denoted $ N_{\varepsilon_n}(\partial A)$ the $\varepsilon_n$-neighborhood of $\partial A$, with $\varepsilon_n = 2r$ here. (As usual we denote $r = r_n$ the maximal edge length in $\cM_n$.) 
By continuity of the measure $\mu$, we know that $\lim_{n\to +\infty} \mu(N_{\varepsilon_n}(\partial A)) = 
\mu (\partial A) = 0$. Note that we used the boundedness of $A$, which guarantees that $\mu(N_{\varepsilon_n}(\partial A)) < +\infty$. It follows:
\begin{equation} \label{eq:CnMinusBn}
 \lim_{n \to +\infty} \mu(C_n - B_n) = 0~.
\end{equation}
Since $B_n \subseteq A \subseteq C_n$, \eqref{eq:CnMinusBn} implies that $\lim_{n \to +\infty} \mu(B_n) =  \lim_{n \to +\infty} \mu(C_n) = \mu(A)$, and we conclude with 
\eqref{eq:ComparisonMeasures} that $\lim_{n \to +\infty} \mu_n(A) = \mu(A)$.
\end{proof}

\subsection{Convergence of the tension field}
\label{subsec:ConvergenceDiscreteTensionField}

Now we consider another Riemannian manifold $N$ and a smooth function $f \colon M \to N$.

Consider a fine and crystalline sequence of meshes $(\cM_n)_{n\in \N}$ on $M$, with mesh size (\ie{} maximum edge length) $r = r_n$, and underlying graph $\cG_n$.

\begin{theorem}
\label{thm:ConvergenceTensionField}
Assume that the sequence of meshes $(\cM_n)_{n \in \N}$ on $M$ is fine and crystalline. If $\cG_n$ is equipped with a system of vertex and edge weights
that is Laplacian at some vertex $x$, then
\begin{equation} \label{eq:ConvergenceTensionField}
\tau_{\cG_n}(f_n)_x - \tau(f)_x = \bigO\left(r^2\right)\,.
\end{equation}
\end{theorem}

\begin{notation}
We denote $f_n \coloneqq \pi_{\cG_n}(f)$, the discretization of $f$ along $\cG_n$ (\ie{} restriction to $\cG_n^{(0)}$).
\end{notation}

\begin{remark} \label{rem:BigOPrecision}
The proof below shows that in \eqref{eq:ConvergenceTensionField}, the $\bigO(r^2)$ function depends on $f$, but may be chosen independent of $x$ if $M$ is compact.
\end{remark}

\begin{proof}
Consider $F:=\exp_{f(x)}^{-1} \circ f \circ \exp_x : \upT_x M \to \upT_{f(x)}N$.
For $y \sim x$, denote $v = v_y \coloneqq \exp_x^{-1}y$. By Taylor's theorem we have
\begin{equation}
\label{eq:TaylorFexp}
\exp_{f(x)}^{-1}f(y) = F(v) = (\upd F)_{|0}(v) + \frac{1}{2} (\upd^2 F)_{|0}(v,v) + 
\frac 16 (\upd^3 F)_{|0}(v,v,v) + \bigO\left(r^4\right)\,.
\end{equation}
This implies
\begin{equation}
 \begin{aligned}
\tau_\cG(f)(x) & = \frac 1{\mu_x} \sum_{y\sim x} \omega_{xy} \, \exp_{f(x)}^{-1}f(y) \\
& = \frac 1{\mu_x} \sum_{y\sim x}\omega_{xy} \, (\upd F)_{|0}(v) + \frac 1{2\mu_x} \sum_{y\sim x}\omega_{xy}\,  (\upd^2 F)_{|0}(v,v)\\
& \quad + \frac 1{6 \mu_x} \sum_{y\sim x}\omega_{xy} \, (\upd^3 F)_{|0}(v,v,v)+\frac{1}{\mu_x}\sum_{y\sim x}\omega_{xy} \, \bigO\left(r^4\right)
\end{aligned}
\end{equation}

By conditions \ref{itemprop:LinearCondition} and \ref{itemprop:CubicCondition} of \autoref{prop:LaplacianWeights}, the first and third sums above vanish, while the second
sum is rewritten with condition \ref{itemprop:QuadraticCondition}: 
\begin{equation}
\tau_\cG(f_\cG)(x) = \tr\left(\upd^2F_{|0}\right) + \frac{1}{\mu_x}\sum_{y\sim x}\omega_{xy} \, \bigO\left(r^4\right)\,.
\end{equation}
Note that $\tr\left(\upd^2F_{|0}\right) = \tr\left(\nabla^2f_{|x}\right) =\tau(f)(x)$, and conclude with \autoref{lem:SumWeightsCrystalline}.
\end{proof}

\subsection{Convergence of the energy}
\label{subsec:ConvergenceEnergy}

We keep the setting of \autoref{subsec:ConvergenceDiscreteTensionField}: $f \colon M \to N$ is a smooth function between
Riemannian manifolds, and $M$ is equipped with a sequence of meshes $(\cM_n)_{n \in \N}$ that is fine and crystalline. 

\subsubsection{Convergence of the energy density}

\begin{theorem}
\label{thm:ConvergenceEnergyDensity}
Assume that the sequence of meshes $(\cM_n)_{n \in \N}$ on $M$ is fine and crystalline.
Assume $\cG_n$ is equipped with a system of vertex and edge weights. Then 
\begin{equation}
 e_{\cG_n}(f_n) = e(f) + \bigO\left(r^2\right)
\end{equation}
on the set of vertices where $\cG_n$ is Laplacian.
\end{theorem}

Recall that we denote $f_n \coloneqq \pi_n(f)$ the discretization of $f$ along $\cG_n$.

\begin{remark}
 \autoref{rem:BigOPrecision} holds again for \autoref{thm:ConvergenceEnergyDensity}.
\end{remark}

\begin{proof}
Assume $\cG_n$ is Laplacian at $x$.
Using \eqref{eq:TaylorFexp} again, denoting $v_y = \exp_x^{-1}y$, we find that
\begin{align}
e_\cG(f)_x & = \frac1{4\mu_x} \sum_{y\sim x} \omega_{xy} \, \Vert F(v_y) \Vert ^2 \\
& = \frac1{4 \mu_x} \sum_{y\sim x} \omega_{xy} \, \left\Vert ( \upd F)_{|0}(v_y) + \frac12 (\upd^2 F)_{|0}(v_y) + \bigO\left( r^3\right) \right\Vert^2 \\
& = \frac1{4\mu_x} \sum_{y\sim x} \omega_{xy} \,\Vert (\upd F)_{|0}(v_y) \Vert^2 + 
\frac1{4\mu_x} \sum_{y\sim x} \omega_{xy}\, \langle (\upd F)_{|0}(v_y) , (\upd^2 F)_{|0}(v_y) \rangle \\
& \quad + \frac1{4\mu_x} \sum_{y\sim x} \omega_{xy} \,\bigO\left(r^4 \right)\,.
\end{align}
Condition \ref{itemprop:CubicCondition} of \autoref{prop:LaplacianWeights} implies that the second sum vanishes.
\autoref{lem:SumWeightsCrystalline} implies that the third sum is $\bigO\left(r^2 \right)$.
By condition \ref{itemprop:QuadraticCondition} of \autoref{prop:LaplacianWeights}, the remaining first sum is rewritten
\begin{equation}
 \frac1{4\mu_x} \sum_{y\sim x} \omega_{xy} \,\Vert (\upd F)_{|0}(v_y) \Vert^2 = \frac12 \Vert (\upd F)_{|0}\Vert^2 = e(f)_x
\end{equation}
since $\tr(L^2) = \Vert L \Vert^2$ for any linear form $L$. We thus get
\begin{equation}
e_\cG(f)_x = e(f)_x + \bigO\left(r^2\right)\,. 
\end{equation}
\end{proof}

\subsubsection{Convergence of the energy}

Recall that  the energy is $E(f) \coloneqq \int_M e(f) \upd \mu$. The convergence of the discrete energy is now an easy consequence of the weakly-* convergence of measures
$\mu_n \to \mu$ and the uniform convergence of the energy densities $e_{\cG_n}(f_n) \to e(f)$. This is the classical combination of weak convergence and strong convergence.

\begin{definition} \label{def:LaplacianSequence}
Let $(M,g)$ be a Riemannian manifold. Consider a sequence of geodesic meshes $(\cM_n)_{n\in \N}$, and equip the underlying graphs $\cG_n$ with a system
of positive vertex and edge weights. We call the sequence of biweighted graphs $(\cG_n)_{n\in \N}$ \emph{Laplacian} provided that:
\begin{enumerate}[(i)]
 \item The sequence of meshes $(\cM_n)_{n\in \N}$ is fine and crystalline.
 \item For every $n\in \N$, the vertex weights on $\cG_n$ are given by the volume weights (see \autoref{subsec:VolumeWeights}).
 \item For every $n\in \N$, the system of vertex and edge weights on $\cG_n$ is Laplacian.
\end{enumerate}
\end{definition}

\begin{theorem} \label{thm:ConvergenceEnergy}
 Let $M$ be a Riemannian manifold and let $(\cM_n)_{n\in\N}$ be a Laplacian sequence of meshes. For any smooth $f \colon M \to N$ with compact support:
 \begin{equation}
  \lim_{n \to +\infty} E_{\cG_n}(f_n) = E(f)\,.
 \end{equation}
\end{theorem}

Recall that we denote $f_n \coloneqq \pi_{\cG_n}(f)$ the discretization of $f$ along $\cG_n$.

\begin{proof}
 By \autoref{thm:ConvergenceVolumeWeights},
 \begin{equation}
  E(f) = \lim_{n \to + \infty} \int_M e(f) \upd \mu_n\,.
 \end{equation}
 By \autoref{thm:ConvergenceEnergyDensity}, on the support of $\mu_n$, $e(f) = e_{\cG_n}(f_n) + \bigO\left(r^2\right)$. It follows that
  \begin{equation}
  E(f) = \lim_{n \to + \infty} \int_M e_{\cG_n}(f_n) \upd \mu_n\,,
 \end{equation}
 in other words $E(f) = \lim_{n \to +\infty} E_{\cG_n}(f_n)$.
\end{proof}

\begin{remark}
\label{rem:HyperbolicEnergies}
The proof of \autoref{thm:ConvergenceEnergy} hints that $E(f) = E_{\cG_n}(f_n) + \bigO\left(r^2\right)$, provided that the convergence
of $\mu_n$ to $\mu$ is sufficiently fast. 
Improvements of this estimate can occur in more restricted situations: for instance, when both the target and the domain are hyperbolic
surfaces:
\begin{equation}
E(f) = E_{\cG_n}(f_n) + \bigO\left(r^4\right)\,.
\end{equation}
This can be proven by carrying out involved calculations in the hyperbolic plane, which we spare.
\end{remark}

\subsection{Weak Laplacian conditions}
\label{subsec:WeakLaplacianConditions}

It is clear from the proofs of the main results in the previous subsections that the Laplacian conditions for sequences of meshes can be weakened and still produce the same results, or at least some of them, with minimal changes in the proofs. This is a useful generalization, for it is very stringent to require a sequence of weighted graphs $(\cG_n)$ to be Laplacian for all $n$. Instead we start by asking that the sequence is merely asymptotically Laplacian in the following sense.

\begin{definition} \label{def:asymptoticVolumeWeights}
Let $M$ be a Riemannian manifold. Consider a sequence of geodesic meshes $(\cM_n)_{n\in \N}$, and equip the underlying graphs $\cG_n$ with a system
of positive vertex weights $\{\mu_{x, n}\}$. We call the sequence of weight systems $(\{\mu_{x, n}\})_{n \in \N}$ \emph{asymptotic volume weights
} provided that:
\begin{equation}
\mu_{x, n} = (1 + \littleo(1)) \, \widehat{\mu}_{x, n} 
\end{equation}
for some function $\littleo(1)$ independent of $x$, where $\widehat{\mu}_{x, n}$ denote the volume weights (see \autoref{subsec:VolumeWeights}).
\end{definition}

The following proposition is an immediate consequence of \autoref{thm:ConvergenceVolumeWeights}:

\begin{proposition} \label{prop:ConvergenceAsymptoticVolumeWeights}
If $M$ is any Riemannian manifold and $(\cM_n)_{n \in \N}$ is any fine sequence of meshes, then the measures $(\mu_n)_{n\in\N}$ on $M$
defined by any system of asymptotic volume vertex weights
converge weakly to the volume density on $M$.
\end{proposition}

It is immediate to show that for asymptotic volume weight, \autoref{thm:InterpolationL2} holds with a Lipschitz constant $L_n  = \sqrt{1+\dim M} + \littleo(1)$.
Although this is sufficient for the needs of this paper (see \autoref{lem:ConvergenceL2Lemma3}), let us state in the next theorem that the result can be improved to 
$L_n  = 1 + \littleo(1)$. The proof follows from \autoref{thm:InterpolationL2} by writing an expansion of the volume form in normal coordinates, we skip it for brevity.

\begin{theorem} \label{thm:InterpolationL2AS}
Let $M$ be a compact Riemannian manifold and let $(\cM_n)_{n \in \N}$ be a fine sequence of meshes equipped with a system of asymptotic volume vertex weights.
For any complete Riemannian manifold $N$ of nonpositive sectional curvature, the center of mass interpolation map 
$\iota_n  \colon  \Map_{\cG_n}(M, N) \to \cC(M,N)$
is $L_n$-Lipschitz with respect to the $\upL^2$ distance on both spaces, with $L_n = 1 + \littleo(1)$.
\end{theorem}

\begin{definition} \label{def:asymptoticallyLaplacian}
Let $M$ be a Riemannian manifold. Consider a sequence of geodesic meshes $(\cM_n)_{n\in \N}$ with mesh size $r = r_n$, and equip the underlying graphs $\cG_n$ with a system
of positive vertex and edge weights. We call the sequence of biweighted graphs $(\cG_n)_{n\in \N}$ \emph{asymptotically Laplacian} provided that:
\begin{enumerate}[(i)]
 \item \label{item:AsymptoticallyLaplaciani} The sequence of meshes $(\cM_n)_{n\in \N}$ is fine and crystalline.
 \item \label{item:AsymptoticallyLaplacianii} The vertex weights are asymptotic volume weights (see \autoref{def:asymptoticVolumeWeights}).
 \item \label{item:AsymptoticallyLaplacianiii} The system of vertex and edge weights on $\cG_n$ is Laplacian up to $\bigO\left(r^2\right)$ at all vertices.
\end{enumerate}
\end{definition}

Explicitly, \ref{item:AsymptoticallyLaplacianiii} means that for all $x \in \cV_n$ and $L \in \upT_x^* M$:
\begin{enumerate}[(1)]
\item 
\label{item:pL1}
\begin{equation}
\frac{1}{\mu_x} \sum_{y \sim x} \omega_{xy} \, \overrightarrow{xy} = \bigO\left(r^2\right)
\end{equation}
\item 
\label{item:pL2}
\begin{equation}
 \frac{1}{\mu_x} \sum_{y\sim x} \omega_{xy} \, L(\overrightarrow{xy})^2  = 2\| L \|^2 \left(1 + \bigO\left(r^2\right)\right)
\end{equation}
\item 
\label{item:pL3}
\begin{equation}
 \frac{1}{\mu_x} \sum_{y \sim x} \omega_{xy} \, L(\overrightarrow{xy}) ^3   = \| L \|^3 \bigO\left(r^2\right)
\end{equation}
\end{enumerate}
The $\bigO(r^2)$ functions above should be independent of $x$ and $L$. 
Note again that to alleviate notations, we drop the dependence in $n$ when writing $r$, $\mu_x$, and $\omega_{xy}$.

It is immediate to check that the proofs of \autoref{thm:ConvergenceTensionField}, \autoref{thm:ConvergenceEnergyDensity}, and \autoref{thm:ConvergenceEnergy}
apply to asymptotically Laplacian sequences of graphs. 
Alas, it is still unreasonable to expect to be able to construct asymptotically Laplacian sequences in general. Fortunately,
the notion may be further slightly weakened while keeping the validity of the most important theorems, 
and allowing the systematic construction of such sequences in \autoref{sec:Construction} (at least in the $2$-dimensional case).

\begin{definition}
\label{def:almostAsymptoticallyLaplacian}
Let $M$ be a compact Riemannian manifold of dimension $m$.
We say that the sequence of biweighted graphs $(\cG_n)_{n\in \N}$ is \emph{almost asymptotically Laplacian}
if it satisfies conditions \ref{item:AsymptoticallyLaplaciani} and \ref{item:AsymptoticallyLaplacianii} of \autoref{def:asymptoticallyLaplacian},
and the modified version of \ref{item:AsymptoticallyLaplacianiii}:
\begin{enumerate}[label=(iii')]
 \item \label{item:AsymptoticallyLaplacianiiiprime}
 There is a decomposition $\cV_n = \bigsqcup_{k=0}^2 \cV_n^{(k)}$, with $\mu_n\left(\cV_n^{(k)}\right) = \bigO(r^{k})$, 
 so that the system of vertex and edge weights on $\cG_n$ is Laplacian up to $\bigO\left(r^{2-k}\right)$ on $\cV_n^{(k)}$.
\end{enumerate}
\end{definition}

\begin{remark}
Any asymptotically Laplacian sequence of meshes is almost asymptotically Laplacian: take $\cV_n^{(0)} = \cV_n$ and $\cV_n^{(1)} = \cV_n^{(2)} = \emptyset$.
\end{remark}

\begin{remark}
In application, the set $\cV_n^{(k)}$ will be the vertices contained in the codimension $k$-skeleton of a fixed triangulation of $M$ (and not contained
in $\cV_n^{(k+1)}$).
\end{remark}

The following theorems are generalized or weakened versions of \autoref{thm:ConvergenceTensionField}, \autoref{thm:ConvergenceEnergyDensity},
and \autoref{thm:ConvergenceEnergy}.

\begin{theorem}
\label{thm:ConvergenceTensionFieldGeneralized}
Let $M$ be a compact Riemannian manifold. Consider a sequence of geodesic meshes $(\cM_n)_{n\in \N}$, with mesh sizes $r=r_n$, and equip the underlying graphs $\cG_n$ with a system
of vertex and edge weights. Let $f \colon M \to N$ be any smooth map to another Riemannian manifold.
\begin{enumerate}[(1)]
 \item \label{item:ConvergenceTensionFieldGeneralized1} If $(\cG_n)_{n\in\N}$ is asymptotically Laplacian, then $\left\Vert \; \tau(f) - \tau_{\cG_n}(f_n) \; \right\Vert_\infty = \bigO\left(r^2 \right)$. A fortiori, 
 \begin{equation}
   \left\Vert \tau(f) - \tau_{\cG_n}(f_n) \right\Vert_2 = \bigO\left(r^2 \right)\,.
 \end{equation}
 \item \label{item:ConvergenceTensionFieldGeneralized2}
If $(\cG_n)_{n\in\N}$ is almost asymptotically Laplacian, then
\begin{equation} \label{eq:ConvergenceTensionFieldGeneralized2a}
\left\Vert \tau(f) - \tau_{\cG_n}(f_n) \right\Vert_2 = \bigO\left(r\right)\,.
\end{equation}
Furthermore, if $\vec{V} \in \upT_{f_n} \Map_{\cG_n}(M, N)$ is a unit tangent vector such that $\Vert \vec{V} \Vert_{\cV_n^{(2)}}  = \littleo(1)$, then
\begin{equation} \label{eq:ConvergenceTensionFieldGeneralized2b}
\left \langle \; \tau(f) - \tau_{\cG_n}(f_n ) \; , \vec{V} \; \right \rangle = \littleo(r)\, .
\end{equation}
\end{enumerate}
\end{theorem}

Note that we use the discrete measure $\mu_n$ on the vertex set of $\cG_n$ in order to define the $\upL^2$-norm on spaces of discrete maps along $\cG_n$.

\begin{proof}
 When $(\cG_n)_{n\in\N}$ is Laplacian, \ref{item:ConvergenceTensionFieldGeneralized1} is an immediate consequence of \autoref{thm:ConvergenceTensionField}.
 When $(\cG_n)_{n\in\N}$ is merely asymptotically Laplacian, the proof of \autoref{thm:ConvergenceTensionField} is still valid up to $\bigO\left(r^2 \right)$.
 
For the proof of \ref{item:ConvergenceTensionFieldGeneralized2}, let $\cV_n^{(k)}$ be the subset of $\cV_n$ of mass $\bigO(r^{k})$ 
where $\cG_n$ is Laplacian up to $\bigO\left(r^{2-k}\right)$. By tracing the proof of \autoref{thm:ConvergenceTensionField}, one quickly sees that $\tau(f) = \tau_{\cG_n}(f_n) + \bigO\left(r^{2-k}\right)$ on $\cV_n^{(k)}$, for each $k \in \{0, 1, 2\}$. The decomposition $\cV_n = \bigsqcup_{k=0}^2 \cV_n^{(k)}$ implies
\begin{equation}
\begin{split}
\Vert \tau(f) - \tau_{\cG_n}(f_n ) \Vert^2 & = \sum_{k=0}^2 \Vert \tau(f) - \tau_{\cG_n}(f_n ) \Vert_{\cV_n^{(k)}}^2 \\
&\leqslant \sum_{k=0}^2 \Vert \tau(f) - \tau_{\cG_n}(f_n ) \Vert_{\infty, \cV_n^{(k)}}^2 ~\mu(\cV_n^{(k)}) \\
&\leqslant \sum_{k=0}^2 \bigO\left(r^{4-2k}\right) \bigO\left(r^{k}\right) = \bigO(r^2)\,.
\end{split}
\end{equation}
For the second estimate, write similarly
\begin{equation}
\begin{split}
\left\langle \tau(f) - \tau_{\cG_n}(f_n ) \, , \, \vec{V} \right\rangle & = \sum_{k=0}^2 \left\langle \tau(f) - \tau_{\cG_n}(f_n ) \, , \, \vec{V} \right\rangle_{\cV_n^{(k)}} \\
&\leqslant \sum_{k=0}^2 \Vert \tau(f) - \tau_{\cG_n}(f_n ) \Vert_{\cV_n^{(k)}} ~ \Vert \vec{V} \Vert_{\cV_n^{(k)}} \\
&\leqslant \bigO\left(r^{2}\right) \cdot 1 + \bigO\left(r^{3/2}\right) \cdot 1 + \bigO\left(r\right) \cdot \littleo(1) = \littleo(r)\,.
\end{split}
\end{equation}
\end{proof}

\begin{theorem}
\label{thm:ConvergenceEnergyDensityGeneralized}
We keep the setup of \autoref{thm:ConvergenceTensionFieldGeneralized}.
\begin{enumerate}[(1)]
 \item \label{item:ConvergenceEnergyDensityGeneralized1} If $(\cG_n)_{n\in\N}$ is Laplacian or asymptotically Laplacian, then 
 \begin{equation}
  \left\Vert e(f) - e_{\cG_n}(f_n) \right\Vert_\infty = \bigO\left(r^2\right)\,.
 \end{equation}
 \item \label{item:ConvergenceEnergyDensityGeneralized2} If $(\cG_n)_{n\in\N}$ is almost asymptotically Laplacian, 
 with decomposition $\cV_n = \bigsqcup_{k=0}^2 \cV_n^{(k)}$, then 
  \begin{equation}
  \left| e(f)(x) - e_{\cG_n}(f_n)(x) \right|= \bigO\left(r^{2-k}\right)
 \end{equation}
for every $x\in \cV_n^{(k)}$.
\end{enumerate}
\end{theorem}

\begin{proof}
The proof is easily adapted from the proof of \autoref{thm:ConvergenceEnergyDensity}.
\end{proof}

\begin{theorem}
\label{thm:ConvergenceEnergyGeneralized}
We keep the setup of \autoref{thm:ConvergenceTensionFieldGeneralized}.
If $(\cG_n)_{n\in\N}$ is almost asymptotically Laplacian,  
 \begin{equation}
  \lim_{n \to +\infty} E_{\cG_n}(f_n) = E(f)~.
 \end{equation}
\end{theorem}

\begin{remark}
 Of course, \autoref{thm:ConvergenceEnergyGeneralized} also holds for Laplacian and asymptotically Laplacian sequences of meshes,
 given the hierarchy between these conditions.
\end{remark}

\begin{proof}[Proof of \autoref{thm:ConvergenceEnergyGeneralized}]
 By definition of almost asymptotically Laplacian, the sequence of measures $(\mu_n)_{n\in\N}$ converges weakly-* to the measure $\mu$ on $M$,
 therefore 
 \begin{equation} 
 \begin{split}
  E(f) = \int_M e(f) \upd \mu = \lim_{n \to +\infty} \int_M e(f) \upd \mu_n\,.  \label{eq:ProofCEG1}
\end{split}
 \end{equation}
Let $\cV_n = \bigsqcup_{k=0}^2 \cV_n^{(k)}$ be the decomposition of the vertices of $\cG_n$ granted by \autoref{def:almostAsymptoticallyLaplacian}.
 By \autoref{thm:ConvergenceEnergyDensityGeneralized},
  \begin{equation} 
 \begin{split}
  \int_M e(f) \upd \mu_n = \sum_{k=0}^2 \int_{ \cV_n^{(k)} } e(f) \upd \mu_n
  = \sum_{k=0}^2 \int_{\cV_n^{(k)}} e_{\cG_n}(f_n) + \bigO\left( r^{2-k} \right) \upd \mu_n \,.
  \end{split}
 \end{equation}
 It follows:
  \begin{equation} 
 \begin{split}  
  \int_M e(f) \upd \mu_n &= \int_M e_{\cG_n}(f_n) \upd \mu_n + \sum_{k=0}^2 \bigO\left(r^{k}\right) \, \bigO\left(r^{2-k}\right)  = E_{\cG_n}(f_n) + \bigO\left(r^2\right) \,. \label{eq:ProofCEG2}
  \end{split}
 \end{equation}
 In particular, we find that $\int_M e(f) \upd \mu_n = E_{\cG_n}(f_n) + \littleo\left(1\right)$. Injecting this into \eqref{eq:ProofCEG1} yields 
 the desired result $E(f) = \lim_{n \to +\infty} E_{\cG_n}(f_n)$.
\end{proof}

\section{Convergence to smooth harmonic maps}
\label{sec:Convergence}

Let $(M,g)$ be a compact Riemannian manifold and let $(N,h)$ be a Riemannian manifold of nonpositive sectional curvature
which does not contain any flats (totally geodesic flat submanifolds).
Consider a connected component $\cC$ of the space of smooth maps $\cC^\infty(M,N)$ that does not contain any map of rank everywhere $\leqslant 1$. For instance, take any connected component of maps whose topological degree is nonzero when $\dim M = \dim N$. 
When $N$ is compact, a celebrated theorem of Eells-Sampson implies that $\cC$ contains a harmonic map $w$ \cite{MR0164306}, and by Hartman  \cite{MR0214004} the harmonic map $w$ is unique.

In this section we show that one can obtain the harmonic map $w \in \cC$ as the limit of discrete harmonic maps
$u_n$ along a sequence of meshes $(\cM_n)_{n\in N}$, provided that:
\begin{enumerate}[(i)]
 \item The discrete energy functional
$E_n$ is sufficiently convex on the discrete homotopy class $\cC_n$. We expect that this is the case
when $N$ is compact and has negative sectional curvature, and have showed it in the $2$-dimensional case in our previous work \cite{Gaster-Loustau-Monsaingeon1} (see \autoref{subsec:ApplicationHS}).
 \item The sequence of meshes is Laplacian (\autoref{def:LaplacianSequence}), 
or one of the weaker versions (\autoref{def:asymptoticallyLaplacian}, \autoref{def:almostAsymptoticallyLaplacian}). In the next and final section
\autoref{sec:Construction}, we systematically construct such sequences.
\end{enumerate}
We then show convergence of the discrete heat flow $u_{kn}$ to the smooth harmonic map $w$, 
when the time and space discretization indices $k$ and $n$ simultaneously run to $+\infty$, provided the adequate CFL condition is satisfied (see \autoref{subsec:CVTimeSpace}).

\subsection{Strong convexity of the discrete energy}
\label{subsec:StrongConvexity}

Please refer to \cite[\S 3.1]{Gaster-Loustau-Monsaingeon1} for the definition of convex, strictly convex, and strongly convex functions
on Riemannian manifolds. In a nutshell, these notions are generalized from the one-dimensional case by restricting to geodesics; the convexity [resp. $\alpha$-strong convexity] of a smooth function is characterized by its Hessian being $\geqslant 0$ [resp. $\geqslant \alpha g$ where $g$ is the Riemannian metric].

Keeping the same setup as above, assume moreover that $N$ is compact and has negative sectional curvature. In this case, we expect that the discrete 
energy functional $E_\cG \colon \cC_\cG \to \R$ is $\alpha_\cG$-strongly convex for any biweighted graph $\cG$ on $M$ underlying a mesh, 
for some $\alpha_\cG >0$. 
In our previous paper, we proved this statement when $M$ and $N$ are $2$-dimensional. 
The estimates we obtained (see \cite[Thm.~3.20, Prop.~3.14]{Gaster-Loustau-Monsaingeon1}) imply that, when $\cG$ is equipped with volume weights, $\alpha_\cG = \Omega \left( \diam(\cG)^{-1} \right)$. 
Further, when $(\cM_n)_{n \in \N}$ is a fine and crystalline sequence of meshes of $M$ and mesh sizes $r=r_n$, with underlying graphs $\cG_n$, discrete energy functionals $E_n:=E_{\cG_n}$, and moduli of convexity $\alpha_n:=\alpha_{\cG_n}$, \autoref{thm:CrystallineProperties} implies that we have the estimate $\alpha_n = \Omega \left( r \right)$.

In fact, we conjecture that the smooth energy $E \colon \cC \to \R$ is $\alpha$-strongly convex for some $\alpha > 0$ 
(see \cite[\S 3.2]{Gaster-Loustau-Monsaingeon1} for a discussion), 
and we expect that $\alpha = \lim_{n \to +\infty}{\alpha_n}$ for any asymptotically Laplacian sequence of meshes $(\cM_n)_{n \in \N}$.
In particular, the sequence $(\alpha_n)_{n \in \N}$ 
should be $\Omega(1)$ in great generality (see \autoref{not:BigONotations} for the notations $\Omega$ and $\Theta$). 

\subsection{\texorpdfstring{$\upL^2$ convergence}{L2 convergence}}

The main theorem of this section is:
\begin{theorem} \label{thm:ConvergenceL2}
Let $M$ and $N$ be Riemannian manifolds, with $M$ compact and $N$ complete with nonpositive sectional curvature.
Let $\cC$ be a connected component of $\cC^\infty(M,N)$ containing a harmonic map $w$.
Consider a sequence of meshes $(\cM_n)_{n\in \N}$ of $M$ with mesh size $r=r_n$ and underlying graphs $(\cG_n)_{n\in\N}$ that satisfy: 
\begin{enumerate}[(i)]
\item The sequence $(\cG_n)_{n\in \N}$ is almost asymptotically Laplacian.
\item The discrete energy $E_n \colon \Map_{\cG_n}(M,N) \to \R$ is $\alpha_n$-strongly convex on $\cC_n$, with $\alpha_n = \Omega \left( r^c \right)$. 
\end{enumerate}
Denote $v_n \in \Map_{\cG_n}(M,N)$, the minimizer of $E_n$ on $\cC_n$ and $\widehat{v_n}$ its center of mass interpolation.

If $c<1$, then
\begin{equation}
 \widehat{v_n} \xrightarrow[n \to +\infty]{} w \quad \text{in } \upL^2(M, N)\,.
\end{equation}
Moreover, the conclusion still holds if $c=1$ and $\dim M=2$, assuming $(\cG_n)_{n\in \N}$ has uniformly bounded ratio between edge weights.
\end{theorem}

\begin{remark}
Under the assumptions of \autoref{thm:ConvergenceL2} $w$ must be the unique smooth harmonic map in $\cC$, the minimizer of the energy functional.
\end{remark}

\begin{remark}
The case $c=1$ and $\dim M=2$ is especially salient in light of \cite{Gaster-Loustau-Monsaingeon1}, which guarantees that
 $c=1$ does hold when $\dim M=2$ in a broad setting: see \autoref{subsec:ApplicationHS} for details.
\end{remark}

\begin{proof}
The proof is a combination of a few key ideas that we emphasize using in-proof lemmas. The bulk of the hard work has been done in the previous sections,
which we will refer to for the proof of these lemmas.

Let $w_n \coloneqq \pi_n(w)\in \Map(\cG_n, N)$ denote the discretization of $w$ (restriction of $w$ to the vertex set of $\cG_n$).
We also denote $\widehat{w_n}$ the center of mass interpolation of $w_n$.

\begin{lemma} \label{lem:ConvergenceL2Lemma1}
We have $\widehat{w_n} \to w$ in $\upL^2(M,N)$ when $n \to +\infty$, moreover $E(\widehat{w_n}) \to E(w)$.
\end{lemma}

\begin{proof}[Proof of \autoref{lem:ConvergenceL2Lemma1}]
 This is an immediate consequence of \autoref{cor:CVFunctionsCrystallineCompact}, which we can invoke since $M$ is compact 
 and the sequence of meshes $(\cM_n)_{n\in \N}$ is fine and crystalline.
\end{proof}

\begin{lemma} \label{lem:ConvergenceL2Lemma3}
There exists a constant $L>0$ such that
\begin{equation}
d(\widehat{w_n}, \widehat{v_n}) \leqslant L\, d(w_n, v_n)
\end{equation}
where $d(\widehat{w_n}, \widehat{v_n})$ and $d(w_n, v_n)$ indicate the $\upL^2$ distances in $\cC(M, N)$ and 
$\Map_{\cG_n}(M, N)$.
\end{lemma}

\begin{proof}[Proof of \autoref{lem:ConvergenceL2Lemma3}]
This follows immediately from \autoref{thm:InterpolationL2AS}.
\end{proof}

\begin{lemma} \label{lem:StrongConvexity}
Let $R$ be a complete Riemannian manifold and $F \colon R \to \R$ be a $\cC^2$ $\alpha$-strongly convex function. Then
$F$ has a unique minimizer $x^*$, and for all $x\in R$
\begin{equation} \label{eq:StrongConvexityIneq1}
d(x,x^*) \leqslant \frac{ \left| \left \langle \; \grad F(x) \;, \vec{V} \; \right \rangle \right|}{\alpha}
\end{equation}
where $\vec{V} $ is a unit tangent vector in the direction $\exp_x^{-1}(x^*)$, in particular
\begin{equation} \label{eq:StrongConvexityIneq1b}
d(x,x^*) \leqslant \frac{\Vert \grad F(x) \Vert}{\alpha}\,.
\end{equation}
We also have
\begin{equation} \label{eq:StrongConvexityIneq2}
0 \leqslant F(x) - F(x^*) \leqslant \frac{\left\Vert\grad F(x)\right\Vert^2}{\alpha}\,.
\end{equation}
\end{lemma}

\begin{proof}[Proof of \autoref{lem:StrongConvexity}]
Recall that on a complete Riemannian manifold $R$, there exists a length-minimizing geodesic between any two points.
It is not hard to show that a strongly convex function on a complete (finite-dimensional) Riemannian manifold is proper, hence existence of the minimizer,
and uniqueness follows from strict convexity.

The first inequality \eqref{eq:StrongConvexityIneq1}  is easy to prove for a function $f \colon \R \to \R$ by integrating $f''(x) \geqslant \alpha$. 
For the general case, take a length-minimizing unit geodesic $\gamma \colon \R \to R$ with $\gamma(0) = x^*$ and $\gamma(L) = x$, and apply the previous result to $f = F \circ \gamma$. The second inequality \eqref{eq:StrongConvexityIneq2} follows with Cauchy-Schwarz. For \eqref{eq:StrongConvexityIneq2}, 
the one-dimensional case is readily obtained via the mean value theorem,
and the general case quickly follows.
\end{proof}

\begin{lemma} \label{lem:ConvergenceL2Lemma2}
We have
\begin{equation} \label{eq:ConvergenceL2Lemma2a}
d(w_n, v_n) \leqslant \frac{ \left| \left \langle \; \tau_{\cG_n}(w_n) \; , \vec{V} \; \right \rangle \right| }{\alpha_n} \quad \text{ where } \quad
\vec{V} = \frac{\exp_{w_n}^{-1}v_n}{\Vert \exp_{w_n}^{-1}v_n \Vert} \,.
\end{equation}
where $d$ denotes the $\upL^2$ distance in $\Map_{\cG_n}(M, N)$. 
In particular, 
\begin{equation} \label{eq:ConvergenceL2Lemma2b}
d(w_n, v_n) \leqslant \frac{ \Vert \tau_{\cG_n}(w_n) \Vert }{\alpha_n}\,.
\end{equation}
\end{lemma}

\begin{proof}[Proof of \autoref{lem:ConvergenceL2Lemma2}]
Apply \autoref{lem:StrongConvexity} \eqref{eq:StrongConvexityIneq1} and \eqref{eq:StrongConvexityIneq1b} to $R = \Map_{\cG_n}(M,N)$ and $F = E_n$.
\end{proof}

At this point, we would like to apply \autoref{lem:ConvergenceL2Lemma2} and \autoref{thm:ConvergenceTensionFieldGeneralized} to conclude that 
\begin{equation}
d(w_n, v_n) \to 0\,.
\end{equation}
Indeed, 
 \eqref{eq:ConvergenceL2Lemma2b} together with \eqref{eq:ConvergenceTensionFieldGeneralized2a} 
imply that $d(w_n, v_n) = \bigO(r^{1-c})$. If $c<1$, we thus clearly have $d(w_n, v_n) \to 0$.
The equality case $c=1$ is much more subtle. In theory, we can still conclude that $d(w_n, v_n) \to 0$ with
\eqref{eq:ConvergenceL2Lemma2a} and \eqref{eq:ConvergenceTensionFieldGeneralized2b}, which together yield $d(w_n, v_n) = \littleo(1)$.
However, to apply \eqref{eq:ConvergenceTensionFieldGeneralized2b}, we need to know that $\Vert \vec{V} \Vert_{\cV_n^{(2)}}  = \littleo(1)$.
Although we believe this is always true, we only show it when $\dim M = 2$ in this paper.
\begin{lemma} \label{lem:BootstrappingControlNew}
Assume $\dim M = 2$. We have $\Vert \vec{V} \Vert_{\cV_n^{(2)}}  = \littleo(1)$.
\end{lemma}

\begin{proof}[Proof of \autoref{lem:BootstrappingControlNew}]
Clearly, $\Vert \vec{V} \Vert_{\cV_n^{(2)}}^2 \leqslant \Vert \vec{V} \Vert_{\infty}^2 ~ \mu(\cV_n^{(2)})$, that is
\begin{equation}
 \Vert \vec{V} \Vert_{\cV_n^{(2)}}^2 \leqslant \frac{d_\infty(w_n, v_n)^2}{d(w_n, v_n)^2} ~ \bigO(r^2)\,.
\end{equation}
It appears that we win if we can show that $\frac{d_\infty(w_n, v_n)}{d(w_n, v_n)} = \littleo(r^{-1})$. Unfortunately, the comparison between 
the $\upL^\infty$ distance and the $\upL^2$ distance on $\Map_{\cG_n}(M, N)$ only satisfies $\frac{d_\infty(u, v)}{d(u, v)} = \bigO(r^{-1})$ 
in general. However, this inequality may be slightly improved when $v$ is the discrete energy minimizer. In order to avoid burdening our exposition, we relegate this
technical estimate to \autoref{sec:ComparingL2Linfty}. The desired comparison is given in \autoref{cor:bootstrap} (which requires the uniform bound assumption on ratios
of edge weights).
\end{proof}

%
%

We can now smoothly wrap up the proof of \autoref{thm:ConvergenceL2}: write
\begin{align}
 d(\widehat{v_n}, w) & \leqslant d(\widehat{v_n}, \widehat{w_n}) + d(\widehat{w_n}, w) && \text{(triangle inequality)}\\
 &\leqslant  L\, d(w_n, v_n) + \littleo(1) && \text{(by \autoref{lem:ConvergenceL2Lemma3} and \autoref{lem:ConvergenceL2Lemma1})}
\end{align}
We proved that $d(w_n, v_n) \to 0$ if $c<1$ or $c=1$ and $\dim M = 2$, so we are done.
\end{proof}

\begin{remark}
 We believe that the restriction $\dim M = 2$ when $c = 1$ is superfluous.
 Indeed, we expect that \autoref{lem:BootstrappingControlNew} is true in any dimension. However, proving it requires generalizations of the technical estimates 
 of \autoref{sec:ComparingL2Linfty} when $\dim M > 2$. We reserve this (possibly) for a future paper, as well as discussing cotangent weights
 and the constructions of \autoref{sec:Construction} to dimensions $>2$. 
\end{remark}

\subsection{\texorpdfstring{$\upL^\infty$ convergence}{L-infinity convergence}}
\label{subsec:LinftyCV}

Under stronger assumptions, we are able to prove uniform convergence in the $2$-dimensional case by comparing the $\upL^2$
and $\upL^\infty$ distances on the space of discrete maps $\Map_{\cG_n}(M,N)$ (and using \autoref{cor:CVFunctionsCrystallineCompact}). 
See \autoref{sec:ComparingL2Linfty} for details about this comparison.


\begin{theorem} \label{thm:ConvergenceLinfinity}
In the setup of \autoref{thm:ConvergenceL2}, if $\dim M=2$ and $c=0$, then $\widehat{v_n} \to w$ in $\upL^\infty(M,N)$.
\end{theorem}

\begin{proof}
 Write
\begin{equation}
 d_\infty(\widehat{v_n}, w) \leqslant d_\infty(\widehat{v_n}, \widehat{w_n}) + d_\infty(\widehat{w_n}, w)\,.
\end{equation}
The second term $d_\infty(\widehat{w_n}, w)$ converges to zero by \autoref{cor:CVFunctionsCrystallineCompact}. It remains to show that 
$d_\infty(\widehat{v_n}, \widehat{w_n}) \to 0$. By \autoref{thm:CoMOInterpolation} \ref{thm:CoMOInterpolationitemiii}, 
$d_\infty(\widehat{v_n}, \widehat{w_n}) \leqslant d_\infty(v_n, w_n)$.
Using \autoref{cor:bootstrap}, we find that $d_\infty(v_n,w_n) = \littleo\left( r^{2-\dim M} \right)$, and we conclude that 
$d_\infty(v_n, w_n) = \littleo(1)$.
\end{proof}

\begin{remark}
We believe that $c=0$ holds in great generality (see \autoref{subsec:StrongConvexity}). 
\end{remark} 

\begin{remark}
We believe that the restriction $\dim M =2$ (also possibly $c=0$) is superfluous, but are unable to omit it in the current stage of our work.
See \autoref{rem:Dickless} for a related discussion.
\end{remark}

\subsection{Convergence of the energy}

One would like to discuss convergence of the discrete minimizer $\widehat{v_n}$ to the smooth harmonic map $w$ in the Sobolev space $\upH^1(M,N)$, say, under the assumptions of \autoref{thm:ConvergenceL2}, but this function space (or rather its topology) is not well-defined, see \autoref{rem:Sobolev}. It is however still reasonable to ask whether the energy of $v_n$ converges to the energy of $w$.

We shall see that it does not cost much to prove that the discrete energy $E_n(v_n)$ converges to $E(w)$, however it is much more difficult to show that the energy of the interpolation
$E(\widehat{v_n})$ also converges to $E(w)$. While we believe that $E_n(v_n)$ and $E(\widehat{v_n})$ are asymptotic, proving it is too hard in the current state of our work. 
We will thus be content with stating the desired convergence result under very restrictive assumptions.

\begin{remark} \label{rem:Dickless}
The obstacle to show that $E_n(v_n)$ and $E(\widehat{v_n})$ are asymptotic would be lifted by showing that the sequence $(\widehat{v_n})_{n\in \N}$ has a uniformly bounded Lipschitz constant, but this would be a very strong result. It would in fact enable us to prove
\autoref{thm:ConvergenceL2} for any asymptotically Laplacian sequence of meshes, with no assumption involving $c$, with a completely
different method involving a Rellich–Kondrachov theorem. In the smooth setting, a uniform Lipschitz bound is achieved by using the Bochner formula and
Moser's Harnack inequality (see \eg{} \cite{MR756629}, \cite[\S2.2.2]{LoustauHarmonicKahler}). This is an essential feature of the heat flow and the theory of harmonic maps.
While developing a discrete Bochner formula and a discrete Moser's Harnack inequality is certainly a worthwhile project, it is also beyond the scope of this paper.
\end{remark}

\begin{theorem} \label{thm:ConvergenceInEnergy}
In the setup of \autoref{thm:ConvergenceL2}, if $c<2$, then $E_n(v_n) \to E(w)$. 
If moreover $\dim M = 2$, $c = 0$, and the sequence of meshes is asymptotically Laplacian, then we also have $E(\widehat{v_n}) \to E(w)$.
\end{theorem}

\begin{proof}
First write that $E(w) = \lim_{n \to +\infty} E_n(w_n)$ by \autoref{thm:ConvergenceEnergyGeneralized}. Thus it is sufficient to show that $E_n(w_n)$ and $E_n(v_n)$ are asymptotic. 
By \autoref{lem:StrongConvexity} \eqref{eq:StrongConvexityIneq2} applied to $F = E_n$, we find that
\begin{equation}
 0 \leqslant E_n(w_n) - E_n(v_n) \leqslant \frac{\left\Vert \tau_{\cG_n}(w_n) \right\Vert^2 }{\alpha_n}
\end{equation}
so with \eqref{eq:ConvergenceTensionFieldGeneralized2a} we find that $\left|E_n(w_n) - E_n(v_n)\right| = \bigO\left(r^{2-c}\right)$ and the claim follows.

For the second claim, first write that $E(w) = \lim_{n \to +\infty} E(\widehat{w_n})$ by \autoref{cor:CVFunctionsCrystallineCompact}. Thus it is sufficient to show that $E(\widehat{w_n})$ and $E(\widehat{v_n})$ are asymptotic. One can derive from \autoref{thm:CoMOInterpolation} \ref{thm:CoMOInterpolationitemiii} and 
\autoref{prop:CrystallineAngles} \ref{item:CrystallineAnglesii} that for a fine and crystalline sequence of meshes, 
\begin{equation}
 \left| \left\Vert \upd \widehat{f} (x) \right\Vert - \left\Vert \upd \widehat{g} (x) \right\Vert \right| = \bigO\left(\frac{d_\infty(f,g)}{r}\right)~.
\end{equation}
uniformly in $f, g \in \Map_{\cG_n}(M,N)$ and in $x \in M$ in the interior of the triangulation, from which it follows 
$\left| E( \widehat{f}) - E( \widehat{g}) \right| = \bigO\left(\frac{d_\infty(f,g)}{r}\right)$. In our case this gives
$\left| E( \widehat{w_n}) - E( \widehat{v_n}) \right| = \bigO\left(\frac{d_\infty(w_n,v_n)}{r}\right)$. By \autoref{lem:ConvergenceL2Lemma2}, \autoref{thm:ConvergenceTensionFieldGeneralized} \ref{item:ConvergenceTensionFieldGeneralized1}, and \autoref{cor:bootstrap},
we have $d_\infty(w_n, v_n) = \littleo\left(r^{2-c-\frac{\dim M}{2}}\right)$, so we find 
$\left| E( \widehat{w_n}) - E( \widehat{v_n}) \right| = \littleo\left(r^{1-c-\frac{\dim M}{2}}\right)$
hence $\left| E( \widehat{w_n}) - E( \widehat{v_n}) \right| = \littleo(1)$ when $c = 0$ and $\dim M = 2$.
\end{proof}

\subsection{Convergence in time and space of the discrete heat flow}
\label{subsec:CVTimeSpace}

We turn to more practical considerations about how to compute harmonic maps. 
In the previous subsections, we established that, under suitable assumptions, 
the discrete harmonic map $v_n$ converges to the smooth harmonic map $w$. In our previous work \cite{Gaster-Loustau-Monsaingeon1}, we showed
that for each fixed $n \in \N$, $v_n$ may be computed as the limit of the discrete heat flow $u_{k, n}$ when $k \to +\infty$.
While this is relatively satisfactory, in practice one cannot wait for the discrete heat flow to converge for each $n$. Hence it is preferable to let
both indices $k$ and $n$ run to $+\infty$ simultaneously. 
In the theory of PDEs, this situation with a double discretization in time and space is typical--they call it \emph{full discretization}, and one expects
convergence to the solution provided that the time step and the space step satisfy a constraint, called a \emph{CFL condition}.
We are happy to report a similar result.

We keep the same setup as in the beginning of the section. Let $u \in \cC$ be a smooth map, denote by $u_n \in \Map_{\cG_n}(M,N)$ its discretization.
For each $n \in \N$, denote by $(u_{k,n})_{k\in \N}$ the sequence in $\Map_{\cG_n}(M,N)$ obtained by iterating the discrete heat flow from the initial
map $u_{0,n} = u_n$. We recall that the discrete heat flow is defined by
\begin{equation}
 u_{k+1,n} = u_{k, n} + t_n \tau_{\cG_n}(u_{k, n})
\end{equation}
where $t_n$ is a suitably chosen time step and we use the notation $x+ v$ for the Riemannian exponential map $\exp_x (v)$ in $N$.
We recall that the discrete heat flow is just a fixed stepsize gradient descent method for the discrete energy functional $E_n$ on the Riemannian manifold
$\Map_{\cG_n}(M,N)$. In particular, strong convexity of the $E_n$ implies convergence of the discrete heat flow to the unique discrete harmonic map $v_n$
with exponential convergence rate. We refer to \cite{Gaster-Loustau-Monsaingeon1} for more details.

\begin{theorem} \label{thm:DoubleIndexCV}
Consider the same setup and assumptions as in \autoref{thm:ConvergenceL2}. Also assume that for any constant $K>0$, the discrete energy $E_n$ has Hessian
bounded above by $\beta_{n,K} = \bigO(r^{-d})$ on its sublevel set $\{E_n \leqslant K\}$, for some $d \geqslant 0$ independent of $K$. Then
\begin{equation}
 \widehat{u_{k,n}} \xrightarrow[k,n \to +\infty]{} w \quad \text{in } \upL^2(M, N)
\end{equation}
provided the CFL condition:
\begin{equation} \label{eq:CFL}
 k = \Omega\left( \frac{\log(r^{-1})}{r^{c+d}} \right)\,.
\end{equation}
\end{theorem}

\begin{remark}
 The assumption on the upper bound of the Hessian is reasonable when compared to the Euclidean setting due to scaling considerations. 
 When $N$ is a hyperbolic surface, we have $\beta_{n,K} = \bigO(r^{-2})$ by \cite[Prop. 3.17]{Gaster-Loustau-Monsaingeon1}, 
 which satisfies the assumption but is surely not optimal.
\end{remark}

\begin{remark}
The CFL condition \eqref{eq:CFL} is most likely far from optimal.
\end{remark}

\begin{proof}[Proof of \autoref{thm:DoubleIndexCV}]
Let us break the proof into a few key steps.

\begin{lemma} \label{lem:BoundEnergy}
 There exists a constant $K>0$ such that
 \begin{equation}
  E_n(u_{k,n}) \leqslant K
 \end{equation}
 for all $k,n \in \N$.
\end{lemma}

\begin{proof}[Proof of \autoref{lem:BoundEnergy}]
The proof of this lemma is a favorite of ours. For each fixed $n \in \N$, the discrete energy $E_{n}(u_{k,n})$ is nonincreasing with $k$,
since the discrete heat flow is a gradient descent for the discrete energy. In particular $E_n(u_{k,n}) \leqslant E_n(u_{0,n})$.
To conclude, we must argue that the sequence $(E_n(u_n))_{n\in \N}$ is bounded. This is true since it converges to $E(u)$ by \autoref{thm:ConvergenceEnergyGeneralized}.
\end{proof}

\begin{lemma} \label{lem:DiscreteHFRate}
For every $k, n$, we have
\begin{equation}
 d(u_{k,n}, v_n) \leqslant c_n q_n^n
\end{equation}
where $c_n = \bigO\left(r^{-c/2}\right)$ and $q_n = 1 - C r^{c + d} + \littleo\left(r^{c+d}\right)$ with $C>0$.
\end{lemma}

\begin{proof}[Proof of \autoref{lem:DiscreteHFRate}]
This is an immediate consequence of \cite[Theorem 4.1]{Gaster-Loustau-Monsaingeon1}. Note that for the estimate of $c_n$,
we need to use the fact that $E_n(u_{0,n}) = \bigO(1)$, which we showed in \autoref{lem:BoundEnergy}.
\end{proof}

We now finish the proof of  \autoref{thm:DoubleIndexCV}. For every $k, n \in \N$, we have
\begin{equation}
 d(\widehat{u_{k,n}}, w) \leqslant d(\widehat{u_{k,n}}, \widehat{v_n}) + d(\widehat{v_n}, w)\,.
\end{equation}
The second term $d(\widehat{v}_n, w)$ converges to zero by \autoref{thm:ConvergenceL2}. As for the first term, 
we have $d(\widehat{u_{k,n}}, \widehat{v_n}) \leqslant L\, d(u_{k,n}, v_n)$ for some constant $L>0$ by \autoref{thm:InterpolationL2AS}.
Thus it is enough to show that $d(u_{k,n}, v_n) \to 0$ under the appropriate CFL condition.

Let $(\varepsilon_n)_{n\in N}$ be a sequence of positive real numbers converging to zero to be chosen later.
Since $(u_{k,n})$ converges to $v_n$ when $k \to +\infty$, there exists $k_0(n)$ such that $d(u_{k,n}, v_n) \leqslant \varepsilon_n$ for all $k \geqslant k_0(n)$.
Note that the inequality $k \geqslant k_0(n)$ is the CFL condition that we are after, for a/any choice of $(\varepsilon_n)$.
It is possible to compute $k_0(n)$ explicitly with \autoref{lem:DiscreteHFRate}; one finds that
\begin{equation}
 k_0(n) = \frac{\log(c_n) + \log(\varepsilon_n^{-1})}{\log(q_n^{-1})}
\end{equation}
is sufficient. With our estimates we get $\log(c_n) = \Theta(\log(r^{-1}))$ and $\log(q_n^{-1}) \sim C r^{c+d}$. It is easy to choose $\varepsilon_n$
so that  $\log(\varepsilon_n^{-1})$ is negligible compared to $\log(r^{-1})$, \eg{} $\varepsilon_n = \log(r^{-1})$. We thus find
$ k_0(n) = \Theta \left( \frac{\log(r^{-1})}{r^{c+d}} \right)$ as desired.
\end{proof}

\begin{remark}
 We could similarly show $\upL^\infty$ convergence (resp. convergence of the energy) of
 $\widehat{u_{k,n}}$ to $w$ under the assumptions of \autoref{thm:ConvergenceLinfinity} (resp. \autoref{thm:ConvergenceInEnergy})
 and suitable CFL conditions.
\end{remark}

\subsection{Application to surfaces} \label{subsec:ApplicationHS}

When $M$ and $N$ are both $2$-dimensional, our previous work \cite{Gaster-Loustau-Monsaingeon1}
gives estimates for the strong convexity of the discrete energy. More precisely, consider the following setup:

Let $S = M$ and $N$ be closed Riemannian surfaces of negative Euler characteristic. Assume $N$ has negative sectional curvature. 
Assume that $S$ is equipped with a fine and crystalline sequence of meshes $(\cM_n)_{n\in \N}$, equipped with asymptotic volume weights
and positive edge weights such that the ratio of any two edge weights is uniformly bounded.
Consider a homotopy class of maps $\cC \subset \cC^\infty(M,N)$ of nonzero degree, and its discretization $\cC_n$ along each mesh.

\begin{lemma} \label{lem:SCSurfaces}
The discrete energy functional $E_n \colon \cC_n \to \R$ has Hessian bounded below by $\alpha_n$ and above by $\beta_{n,K}$ on any sublevel set $\{E_n \leqslant K\}$,
with
\begin{equation}
\begin{aligned}
 \alpha_n &= \Omega(r) \\ \beta_{n,K} &= \bigO(r^{-2})\,.
\end{aligned}
\end{equation}
\end{lemma}

\begin{proof}
The estimate for $\alpha_n$ is an immediate consequence of \cite[Theorem 3.20]{Gaster-Loustau-Monsaingeon1}.
The estimate for $\beta_n$ is an immediate consequence of \cite[Prop. 3.17]{Gaster-Loustau-Monsaingeon1}. Note that \cite[Prop. 3.17]{Gaster-Loustau-Monsaingeon1}
is only stated for a hyperbolic metric, but it can be extended to any Riemannian metric of curvature bounded below, which is always the case on a compact manifold.
\end{proof}

\begin{remark} \label{rem:Improvingc}
The estimate $\alpha_n = \Omega(r)$ based on \cite[Theorem 3.20]{Gaster-Loustau-Monsaingeon1} only assumes that $N$ has nonpositive sectional curvature.
When $N$ has negative curvature (bounded away from zero by compactness), we expect that a better bound $\alpha_n = \Omega(r^c)$ with $c<1$ is possible to achieve,
in fact we conjecture that $\alpha_n = \Omega(1)$.
\end{remark}

As a consequence of \autoref{lem:SCSurfaces} and the previous theorems of this section, we obtain the following theorem for surfaces.

\begin{theorem} \label{thm:MainThmSurfaces}
If the sequence of meshes $(\cM_n)_{n\in \N}$ is almost asymptotically Laplacian, then the sequence of interpolations $(\widehat{v_n})_{n\in \N}$ 
of the discrete harmonic maps $(v_n)$ converges 
to the unique harmonic map $w \in \cC$ in $\upL^2(M,N)$, and $E(w) = \lim_{n \to +\infty} E_n(v_n)$.

Furthermore, the discrete heat flow $(\widehat{u_{k,n}})_{k, n \in \N}$ from any initial condition $u \in \cC$ converges to $w$ in $\upL^2(M,N)$ when both $k, n \to +\infty$,
provided the CFL condition $k = \Omega\left(\log(r^{-1}) r^{-3} \right)$ holds.
\end{theorem}

\begin{remark}
 \autoref{thm:MainThmSurfaces} could be considered one of the main results of both our previous paper \cite{Gaster-Loustau-Monsaingeon1} and the present paper combined,
 except for the fact that we have yet to produce almost asymptotically Laplacian sequences of meshes on surfaces. In the final section \autoref{sec:Construction}, we systematically construct such sequences.
\end{remark}

The previous theorems of this section (\autoref{thm:ConvergenceLinfinity} and \autoref{thm:ConvergenceInEnergy}) 
also show that under the stronger assumption $\alpha_n = \Omega(1)$ (which we believe holds in a very general setting), 
the conclusions of the previous theorem may be strengthened:

\begin{theorem}
 In the setup of \autoref{thm:MainThmSurfaces}, assuming $\alpha_n = \Omega(1)$, the convergence of $\widehat{v_n}$ to $w$ is uniform.
 If moreover the sequence of meshes is asymptotically Laplacian, then we also have $E(w) = \lim_{n \to +\infty} E(\widehat{v_n})$.
\end{theorem}

\section{Construction of Laplacian sequences}
\label{sec:Construction}

Most of our convergence theorems in \autoref{sec:SequencesOfMeshes} and \autoref{sec:Convergence} require
a Laplacian sequence of meshes (\autoref{def:LaplacianSequence}), 
or one of the weaker variants
(\autoref{def:asymptoticallyLaplacian}, \autoref{def:almostAsymptoticallyLaplacian}). 
Indeed, 
one should only expect convergence for weighted graphs that reasonably capture the geometry of $M$.

In this section, we construct a sequence of weighted meshes on any Riemannian surface and prove that it is always almost asymptotically Laplacian,
and discuss cases where more can be said. This construction is very explicit: in fact, 
it is implemented in our software \Harmony{}
in the case of hyperbolic surfaces. 
The construction can simply be described: take a sequence of meshes obtained by midpoint subdivision 
(\autoref{subsec:MidpointSubdivision}) and equip it with the volume vertex weights (\autoref{subsec:VolumeWeights}) and the 
cotangent weights (\autoref{subsec:CotangentWeights}).

\begin{remark}
It is possible to generalize this construction to higher-dimensional manifolds, most likely with similar results. We reserve this analysis maybe as part of a future paper.
In Euclidean space, the formula for higher-dimensional cotangent weights is given in  \cite{KeenanCraneCotangent}.
\end{remark}

\subsection{Description}

Let $S = M$ be a $2$-dimensional compact Riemannian manifold. One could consider complete metrics with punctures and/or geodesic boundary,
but for simplicity we assume $S$ is closed.

Consider a sequence of meshes $(\cM_n)_{n\in \N}$ with underlying graphs $(\cG_n)_{n\in \N}$ defined by:
\begin{itemize}
 \item $\cM_0$ is any acute triangulation.
 \item $\cM_{n+1}$ is obtained from $\cM_n$ by midpoint subdivision (see \autoref{subsec:MidpointSubdivision}).
\end{itemize}
Furthermore, equip $\cG_n$ with the volume vertex weights (\autoref{subsec:VolumeWeights}) and the 
cotangent weights (\autoref{subsec:CotangentWeights}).

\begin{remark}
 Finding an initial triangulation of $S$ that is acute is far from an easy task, even for a flat surface.
 The reader may refer to \cite{MR3016971} for more background on this active subject.
\end{remark}

\begin{definition}
A \emph{$\Delta$-sequence} is a sequence of meshes $(\cM_n)_{n\in N}$ with the associated biweighted graphs $(\cG_n)_{n\in \N}$ constructed as above.  
\end{definition}

\begin{remark}
We think of ``$\Delta$'' here as standing for either ``Laplacian'' or ``simplex''.
\end{remark}

\subsection{Angle properties}
\label{subsec:AngleProperties}

In order for $\Delta$-sequences to be crystalline and have reasonable edge weights systems, we need to address some questions about the behavior of angles
when iterating midpoint subdivision:
\begin{enumerate}[(1)]
 \item \label{anglequ1} Do all the angles of the triangulation remain bounded away from zero? 
 \item \label{anglequ2} Do all angles remain acute?
 \item \label{anglequ3} Do all angles remain bounded away from $\frac\pi{2}$?
\end{enumerate}
These questions, which are surprisingly hard to answer, are crucial since: \ref{anglequ1} is necessary and sufficient for the sequence of meshes to be crystalline (see \autoref{prop:CrystallineAngles}),
\ref{anglequ2} is sufficient for the edge weights to remain positive, and
\ref{anglequ3} is necessary for the ratio of any two edge weights to remain uniformly bounded, a requirement 
to apply \autoref{thm:MainThmSurfaces}.

\begin{lemma} \label{lem:simplexSubdivisionCrystalline}
Let $(M,g)$ be a compact Riemannian manifold of dimension $m$. Let $(\Delta_n)_{n \in \N}$ be a sequence of simplices with geodesic edges such that 
for every $n\in \N$, $\Delta_{n+1}$ is one of the $2^m$ simplices obtained from $\Delta_n$ by midpoint subdivision.
Then all edge lengths of $\Delta_n$ are $\Theta(2^{-n})$.
\end{lemma}

\begin{proof}
To avoid burdening our presentation with technical Riemannian geometry estimates, 
we postpone this proof to the appendix: see \autoref{prop:simplexSubdivisionCrystallineAppendix} in \autoref{subsec:SubdividingSimplex}.
\end{proof}

\autoref{lem:simplexSubdivisionCrystalline}, together with compactness of $M$ and \autoref{prop:CrystallineAngles}, immediately imply a positive answer to question \ref{anglequ1}:
\begin{theorem} \label{thm:SubdivisionCrystalline}
Let $(M,g)$ be a compact Riemannian manifold. Any sequence of meshes $(\cM_n)_{n\in\N}$ obtained by geodesic subdivision is fine and crystalline.
\end{theorem}


The answer to questions \ref{anglequ2} and \ref{anglequ3} is more nuanced: it is 
not true that refinements of an acute triangulation stay acute, even for fine triangulations in $\H^2$.
However, refinements of a sufficiently fine and sufficiently acute triangulation do remain acute with angles bounded away from $\frac \pi 2$.
This a consequence of \autoref{prop:simplexSubdivisionAcuteAppendix} in \autoref{subsec:SubdividingSimplex}, whose proof we postpone to the appendix.
\begin{theorem} \label{thm:AcuteRefinement}
Let $(M,g)$ be a compact Riemannian manifold. Let $\delta >0$. The iterated refinements of any sufficiently fine initial triangulation of $M$ whose angles 
are all $\leqslant \frac{\pi}{2} - \delta$ remain acute and with angles bounded away from $\frac \pi 2$.
\end{theorem}



We say a sequence of acute triangulations is \emph{strongly acute} if the angles remain uniformly bounded away from $\frac{\pi}{2}$.
Thus any sequence of triangulations
obtained from iterated refinement as in \autoref{thm:AcuteRefinement} is strongly acute.

We record the following easy consequence of \autoref{thm:SubdivisionCrystalline} and \autoref{thm:CrystallineProperties}.

\begin{proposition} \label{prop:DeltaEdgeWeights}
Let $(\cM_n)_{n\ in \N}$ be a strongly acute $\Delta$-sequence in $(S,g)$. Then all edge weights of $\cG_n$ are $\Theta(1)$.
\end{proposition}


\subsection{Laplacian qualities}
\label{subsec:LaplacianQualities}

Let $(\cM_n)_{n\ in \N}$ be a $\Delta$-sequence in $(S,g)$, denote $(\cG_n)_{n\in \N}$ the underlying graphs. 

\begin{definition}
Recall that $\cV_n \subseteq S$ denotes the set of vertices of $\cG_n$. Consider the decomposition $\cV_n = \cV_n^{(0)} \sqcup \cV_n^{(1)} \sqcup \cV_n^{(2)}$, where:
\begin{itemize}
 \item $\cV_n^{(2)}$ consists of the vertices that are also elements of $\cV_0$, called \emph{initial vertices}.
 \item $\cV_n^{(1)}$ consists of the vertices that are located on the edges of the initial triangulation $\cM_0$, and are not elements of $\cV_n^{(2)}$, called \emph{boundary vertices}.
 \item $\cV_n^{(0)}$ consists of all other vertices, called \emph{interior vertices}.
\end{itemize}
\end{definition}

\begin{lemma} \label{lem:WeightsDecomp}
We have $\mu\left(\cV_n^{(k)}\right) = \Theta(r^k)$ for $k \in \{0, 1, 2\}$.
\end{lemma}

\begin{proof}
The cardinal $\left| \cV_n^{(2)} \right|$ is clearly constant, while it is easy to show by induction that $\left| \cV_n^{(1)} \right| = \Theta(2^n)$
and $\left| \cV_n^{(2)} \right| = \Theta(4^n)$. We also have $r_n = \Theta(2^{-n})$ by \autoref{prop:simplexSubdivisionAcuteAppendix} and
$\mu_{x, n} = \Theta\left(r_n^2\right)$ for any $x \in \cV_n$ by \autoref{thm:CrystallineProperties}. The desired estimates follow.
\end{proof}

The decomposition $\cV_n = \cV_n^{(0)} \sqcup \cV_n^{(1)} \sqcup \cV_n^{(2)}$ thus makes any $\Delta$-sequence a candidate to be almost asymptotically Laplacian: see \autoref{def:almostAsymptoticallyLaplacian}. The main theorem of this section provides a positive answer:

\begin{theorem} \label{thm:DeltaSequenceLaplacian}
Any strongly acute $\Delta$-sequence in a closed Riemannian surface $(S,g)$ is almost asymptotically Laplacian.
\end{theorem}

\begin{proof}
 There are several conditions to check: see \autoref{def:almostAsymptoticallyLaplacian}. 
 Condition \ref{item:AsymptoticallyLaplaciani} is satisfied by \autoref{thm:SubdivisionCrystalline}. Condition \ref{item:AsymptoticallyLaplacianii}
 is trivially satisfied by definition of a $\Delta$-sequence. 
 
 It remains to check the Laplacian qualities stated in \ref{item:AsymptoticallyLaplacianiiiprime}, namely that 
 $\cG_n$ is Laplacian up to $\bigO\left(r^{2-k}\right)$ on $\cV_n^{(k)}$ for $k\in \{0,1,2\}$. For each $k$, there are three conditions to check:
 the first-order, second-order, and third-order Laplacian conditions
 , up to $\bigO\left(r^{2-k}\right)$ (see \autoref{def:asymptoticallyLaplacian} \autoref{item:AsymptoticallyLaplacianiii}).
There are thus nine conditions to check, 
some of which can be grouped together.
 
The first lemma is straightforward:
 \begin{lemma} \label{lem:LapFree}
  At any vertex $x \in \cV_n$, the $j$-th order Laplacian condition (for $j \in \{1, 2, 3\}$) holds up to $\bigO(r^{j-2})$.
 \end{lemma}

 \begin{proof}
 We have $\mu_x = \Theta\left(r_n^2\right)$ (\autoref{thm:CrystallineProperties}), $\omega_{xy} = \Theta(1)$ (\autoref{prop:DeltaEdgeWeights})
 and $\overrightarrow{xy} = \bigO(r)$ for any $y \sim x$, therefore
 \begin{equation}
  \frac{1}{\mu_x} \sum_{y \sim x} \omega_{xy} L(\overrightarrow{xy})^j = \Vert L \Vert^j \bigO(r^{j-2})\,.
 \end{equation}
The conclusion easily follows for each $j \in \{1, 2, 3\}$.
 \end{proof}
 
In what follows, we will frequently need to compare our present Riemannian setting to its ``Euclidean counterpart''. Let us clarify what we typically mean by that.
Consider a vertex $x \in \cV_n \subseteq S$ and its neighbors $\{y_i\} \subseteq S$. By working in the normal chart at $x$, we can imagine that $x$ and $\{y_i\}$ live in the Euclidean
plane $\upT_x S$. In this plane, each edge of the triangulation, which is a Riemannian geodesic, may be replaced by a Euclidean straight segment, yielding a Euclidean triangulation.
One can then define, for instance, the Euclidean cotangent weights associated to this Euclidean triangulation.
We shall call the Euclidean cotangent weights $\omega_{xy}^\tE$ the \emph{Euclidean counterparts} of the cotangent weights $\omega_{xy}$.

\begin{lemma} \label{lem:EuclideanCounter}
 The cotangent weights $\omega_{xy}$ are within $\bigO(r^2)$ of their Euclidean counterparts $\omega_{xy}^\tE$.
\end{lemma}

\begin{proof}
This immediately follows from the first-order expansion of the cotangent given in \autoref{prop:ExpansionAngles}. Note that we need to know
that all angles are bounded away from $0$ and $\frac{\pi}{2}$, which is guaranteed respectively by \autoref{thm:SubdivisionCrystalline}
and by definition of a strongly acute $\Delta$-sequence.
\end{proof}
%

The fact that the cotangent weights are \emph{exactly} Laplacian to first order in the Euclidean setting (\autoref{prop:EuclideanCotangentWeightsLaplacianFirstOrder})
and the previous lemma allow us to upgrade the $j=1$ case of \autoref{lem:LapFree}:

\begin{lemma} \label{lem:FirstOrderLap}
 At any vertex $x \in \cV_n$, the first-order Laplacian condition holds up to $\bigO(r)$.
\end{lemma}

\begin{proof}
Write
\begin{equation}
\sum_{y \sim x} \omega_{xy} \overrightarrow{xy} = \sum_{y \sim x} \left(\omega_{xy} - \omega_{xy}^\tE\right) \overrightarrow{xy} ~+ ~\sum_{y \sim x}  \omega_{xy}^\tE \overrightarrow{xy}\,.
\end{equation}
The first sum is $\bigO(r^3)$ by \autoref{lem:EuclideanCounter} and the second vanishes by \autoref{prop:EuclideanCotangentWeightsLaplacianFirstOrder} (note that
$\overrightarrow{xy} \coloneqq \exp_x^{-1}(y)$ is equal to its ``Euclidean counterpart'' $\overrightarrow{xy}^\tE$, since we are looking at the normal chart at $x$).
Since $\mu_x = \bigO(r^2)$, conclude that $\frac{1}{\mu_x} \sum_{y \sim x} \omega_{xy} \overrightarrow{xy} = \bigO(r)$.
\end{proof}

As far as the first-order Laplacian condition is concerned, \autoref{lem:FirstOrderLap} is good enough for vertices $x \in \cV_n^{(2)}$ and $x \in \cV_n^{(1)}$.
However for $x \in \cV_n^{(0)}$, we need to upgrade the estimate to $\bigO(r^2)$. Essentially, this follows from the fact that interior vertices have 
``almost central symmetry'', and second-order Riemannian estimates. The computations are tedious but fairly straightforward, we condensed them in the
proof of the next lemma:

\begin{lemma} \label{lem:FirstOrderLapInterior}
 At any interior vertex $x \in \cV_n^{(0)}$, the first-order Laplacian condition holds up to $\bigO(r^2)$.
\end{lemma}

\begin{proof}
We need to push one step further the asymptotic expansion of the cotangent weights mentioned in \autoref{lem:EuclideanCounter}. Order the neighbors
of $x$ cyclically, and given a neighbor $y$, denote $y'$ and $y''$ the previous and the next neighbors. By \autoref{prop:RiemEstimateAngle3Points}, we have
\begin{equation}
 \omega_{xy} = \omega_{xy}^\tE + \lambda_{xy} + \bigO(r^3) \quad \text{ with } \quad \lambda_{xy} = \frac{1}{2} \left(\varepsilon_{x y' y} + \varepsilon_{x y'' y} \right)
\end{equation}
where the notation $\varepsilon_{O A B}$ is defined in \autoref{prop:RiemEstimateAngle3Points}. It follows that
\begin{equation}
  \sum_{y \sim x} \omega_{xy} \overrightarrow{xy} = \sum_{y \sim x} \omega_{xy}^\tE \overrightarrow{xy} ~+ ~\sum_{y \sim x} \lambda_{xy} \overrightarrow{xy} ~+ ~\bigO(r^4)\,.
\end{equation}
The first sum vanishes as in \autoref{lem:FirstOrderLap}. Since $\mu_x = \bigO(r^2)$, we need to show that $\sum_{y \sim x} \omega_{xy} \overrightarrow{xy} = \bigO(r^4)$.
Hence we win if we show that $\sum_{y \sim x} \lambda_{xy} \overrightarrow{xy} = \bigO(r^4)$.

We note that any interior vertex $x \in \cV_n^{(0)}$ has ``almost central symmetry'' up to $\bigO(r^3)$, meaning that its set of neighbors
may be divided into pairs $\{y_+, y_-\}$ such that $\overrightarrow{x y_+} + \overrightarrow{x y_-} = \bigO(r^3)$ (equivalently, the central symmetry
at $x$ preserves the set of neighbors up to $\bigO(r^3)$). This immediately follows from the fact that $x \in \cV_n^{(0)}$ has in fact ``almost hexaparallel symmetry'',
as we shall see in (\autoref{lem:IntBound2ndLap}).

Now write
\begin{equation}
 \begin{split}
  \sum_{y \sim x} \lambda_{xy} \overrightarrow{xy} &= \sum_{\{y_+, y_-\}} \lambda_{xy_+} \overrightarrow{xy_+} + \lambda_{xy_-} \overrightarrow{xy_-}\\
  &= \sum_{\{y_+, y_-\}} \left(\lambda_{xy_+} - \lambda_{xy_-}\right) \overrightarrow{xy_+} + \lambda_{xy_-} \left(\overrightarrow{xy_-} +  \overrightarrow{xy_+} \right) \,.
 \end{split}
\end{equation}
It is not hard to see from the expression of $\lambda_{xy}$ that $\lambda_{xy} = \bigO(r^2)$, and, due to the almost central symmetry, $\lambda_{xy_+} - \lambda_{xy_-} = \bigO(r^4)$. (To be fair,
it is a few lines of calculations, but let us skip the unnecessary details.) We also have $\overrightarrow{xy_-} +  \overrightarrow{xy_+} = \bigO(r^3)$,
we thus derive from the previous identity that $\sum_{y \sim x} \lambda_{xy} \overrightarrow{xy} = \bigO(r^5)$, which is better than the $\bigO(r^4)$ desired result.
\end{proof}

At this point, it is good to pause and see that we have proved that the first-order Laplacian condition holds up to $\bigO(r^{2-k})$ on $\cV_n^{(k)}$ for all $k\in \{0,1,2\}$,
as required. Let us now turn to the second-order condition. On $\cV_n^{(2)}$, we have already proved that it holds up to $\bigO(1)$ as required: see \autoref{lem:LapFree}.
Let us now show that it holds up to $\bigO(r^2)$ on $\cV_n^{(1)}$ (better than the required $\bigO(r)$) and on $\cV_n^{(0)}$ (as required).
Along with \autoref{lem:FirstOrderLap}, this is the most difficult part of the proof.

\begin{lemma} \label{lem:IntBound2ndLap}
 At any interior vertex $x \in \cV_n^{(0)}$ or boundary vertex $x \in \cV_n^{(1)}$, the second-order Laplacian condition holds up to $\bigO(r^2)$.
\end{lemma}

\begin{proof}
Let $x \in \cV_n^{(0)}$ be an interior vertex. Using Riemannian estimates, we shall prove that 
the second-order Laplacian condition holding up to $\bigO(r^2)$ is a consequence of the fact that $x$ has ``almost hexaparallel symmetry''.
We defined hexaparallel symmetry in the Euclidean setting: see \autoref{def:Hexaparallel}. This definition
naturally extends to the Riemannian setting, using the normal chart at $x$ to bring $x$ and its neighbors back to the Euclidean setting. We further say that
$x$ has \emph{almost hexaparallel symmetry (up to $\bigO(r^3)$)} provided that the neighbors of $x$ are within $\bigO(r^3)$ of a hexaparallel configuration.
Using \autoref{prop:midpoint}, one quickly shows that any interior vertex has almost hexaparallel symmetry.

Denote $\hat{y}_i$ the hexaparallel configuration around $x$ such that $\hat{y}_i - y_i = \bigO(r^3)$, and denote $\omega_{x\hat{y}}^\tE$
the Euclidean counterparts of the cotangent weights $\omega_{x\hat{y}}$. As in \autoref{lem:FirstOrderLap}, one shows that
$\omega_{xy}$, $\omega_{x\hat{y}}$, and $\omega_{x\hat{y}}^\tE$ are all within $\bigO(r^2)$.
Now write
\begin{align}
 \frac{1}{\mu_x} \sum_{y\sim x} \omega_{xy} L\left(\overrightarrow{xy}\right)^2 
 &= \frac{1}{\mu_x} \sum_{y\sim x} (\omega_{xy} - \omega_{x\hat{y}}^\tE) L\left(\overrightarrow{xy}\right)^2
 + \frac{1}{\mu_x} \sum_{y\sim x} \omega_{x\hat{y}}^\tE L\left(\overrightarrow{xy} - \overrightarrow{x \hat{y}}\right)
 L\left(\overrightarrow{xy} + \overrightarrow{x \hat{y}}\right)\\
 & + \frac{1}{\mu_x} \sum_{y\sim x} \omega_{x\hat{y}}^\tE L\left(\overrightarrow{x \hat{y}}\right)^2
\end{align}
One quickly sees that the and second sums $\Vert L \Vert^2 \bigO\left(r^2\right)$. As for the third sum, first note that
denoting $\mu_x^\tE$ the Euclidean area weight at $x$,
we have 
\begin{equation}
\frac{1}{\mu_x^\tE} \sum_{y\sim x} \omega_{x\hat{y}}^\tE L\left(\overrightarrow{x \hat{y}}\right)^2 = 2 \Vert L \Vert^2
\end{equation}
by \autoref{lem:HexaparrallelVertices}. Since $\mu_x = \mu_x^\tE \left(1 + \bigO\left(r^2\right)\right)$ by \autoref{prop:VolumeDensity}, we find
\begin{equation}
 \frac{1}{\mu_x} \sum_{y\sim x} \omega_{x\hat{y}}^\tE L\left(\overrightarrow{x \hat{y}}\right)^2 =
 2 \Vert L \Vert^2 \left(1 + \bigO\left(r^2\right)\right) \,.
\end{equation}
Gathering all three sums, we find $\frac{1}{\mu_x} \sum_{y\sim x} \omega_{xy} L\left(\overrightarrow{xy}\right)^2 = 2 \Vert L \Vert^2 \left(1 + \bigO\left(r^2\right)\right)$
as desired. 

One conducts a similar proof when $x$ is a boundary vertex: in that case, it has almost semi-hexaparallel symmetry up to $\bigO(r^3)$, and the proof
is similarly derived from the Euclidean case.
\end{proof}

This concludes the proof that the second-order Laplacian condition holds up to $\bigO(r^{2-k})$ on $\cV_n^{(k)}$ for all $k\in \{0, 1, 2\}$. 
Let us finally examine the third-order condition. We already proved in \autoref{lem:LapFree} that it holds up to $\bigO(r)$ at any vertex,
which is good enough for $\cV_n^{(2)}$ and $\cV_n^{(1)}$. It remains to prove that it holds up to $\bigO(r^2)$ on $\cV_n^{(0)}$.
It actually holds up to $\bigO(r^3)$:

\begin{lemma}
 At any $x \in \cV_n^{(0)}$, the third-order Laplacian condition holds up to $\bigO(r^3)$.
\end{lemma}

\begin{proof}
This is an easy consequence of the almost central symmetry: write
\begin{equation}
 \begin{split}
  \sum_{y \sim x} \omega_{xy} L(\overrightarrow{xy})^3 &= \sum_{\{y_+, y_-\}} \omega_{xy_+} L(\overrightarrow{xy_+})^3 + \omega_{xy_-} L(\overrightarrow{xy_-})^3\\
  &= \sum_{\{y_+, y_-\}} \left(\omega_{xy_+} - \omega_{xy_-}\right) L(\overrightarrow{xy_+})^3 + \omega_{xy_-} \left(L(\overrightarrow{xy_-})^3 +  L(\overrightarrow{xy_+})^3 \right) \,.
 \end{split}
\end{equation}
By almost central symmetry, we have $\omega_{xy_+} - \omega_{xy_-} = \bigO(r^2)$ and $\overrightarrow{xy_-} +  \overrightarrow{xy_+} = \bigO(r^3)$.
It follows that the first term is $\Vert L \Vert^3 \bigO(r^5)$, as is the second term. (For the second term, write $L(\overrightarrow{xy_-}) = L(\overrightarrow{xy_+}) + \Vert L \Vert
\bigO(r^3)$ and expand the third power of this identity.) Thus we find that $\sum_{y \sim x} \omega_{xy} L(\overrightarrow{xy})^3 = \Vert L \Vert^3 \bigO(r^5)$,
therefore $\frac{1}{\mu_x}\sum_{y \sim x} \omega_{xy} L(\overrightarrow{xy})^3 = \Vert L \Vert^3 \bigO(r^3)$ as required.
\end{proof}
This concludes the proof that the third-order Laplacian condition holds up to $\bigO(r^{2-k})$ on $\cV_n^{(k)}$ for all $k\in \{0,1,2\}$. The proof of \autoref{thm:DeltaSequenceLaplacian} is now complete.
\end{proof}

\begin{remark}
In retrospect, it is remarkable--almost miraculous--how the conditions for a $\Delta$-sequence to be almost asymptotically Laplacian are barely met,
and in turn how these conditions are barely sufficient for the main convergence theorem (\autoref{thm:ConvergenceL2}) to hold, at least in the $c=1$ case.
Seeing how delicate the analysis is, the reader should not be too surprised that it took us many failed attempts until we were able to achieve the right definitions and results.
\end{remark}

\appendix

\section{Riemannian estimates}
\label{sec:RiemannianEstimates}

Many proofs in this paper can be summarized in two steps:
First, the claim is shown to be true in the Euclidean (flat) setting, by direct proof.
Subsequently, it is also true in the Riemannian setting on first approximation (\eg{}, provided the mesh is fine).
The moral justification for the second step is that locally, a Riemannian manifold looks Euclidean. Of course, one should not use this aphorism too liberally, 
since there are local Riemannian invariants such as curvature.
In some cases, one can make this type of proof rigorous with a soft argument using only first-order approximation. 
In others, one should be more cautious and examine the next order terms, which involve curvature.

A standard way to obtain estimates in Riemannian geometry is to compute Taylor expansions in normal coordinates, \ie{}
using the exponential map at some point as a chart, and picking an orthonormal basis of the tangent space to have an $n$-tuple of coordinates.
For example, the Taylor expansion of the Riemannian metric in normal coordinates reads
\begin{equation} \label{eq:RiemannExpansion}
 g_{ij} = \delta_{ij} - \frac{1}{3} R_{ikjl}x^k x^l + \bigO(r^3)
\end{equation}
where $R_{ijkl}$ is the Riemann curvature tensor. This foundational fact of Riemannian geometry goes back to Riemann's 
1854 habilitation \cite{MR3525305}. From this estimate, many other geometric quantities can be similarly approximated:
distances, angles, geodesics, volume, etc.

In \autoref{subsec:RiemannianExpansionsNormalChart}, we establish Riemannian estimates of the most relevant geometric quantities.
These are used implicitly or explicitly throughout the paper, especially \autoref{subsec:LaplacianQualities}.
In \autoref{subsec:SubdividingSimplex}, we study iterated midpoint subdivisions of a simplex in a Riemannian manifold, proving two key
lemmas for \autoref{subsec:AngleProperties}.

\subsection{Riemannian expansions in a normal chart}
\label{subsec:RiemannianExpansionsNormalChart}

Let $(M,g)$ be a Riemannian manifold and let $x_0 \in M$. We consider the normal chart given by the exponential map $\exp_{x_0} \colon \upT_{x_0} M \to M$, 
which is well-defined and a diffeomorphism near the origin. We do not favor the unnecessary introduction of local coordinates, so we will abstain
from choosing an orthonormal basis of $\upT_{x_0} M$ (in other words fixing an identification $\upT_x M \approx \R^m$), and instead work in the Euclidean vector space
$(T_{x_0} M, \langle \cdot, \cdot \rangle_\tE)$ where the inner product $\langle \cdot, \cdot \rangle_\tE$ is just $g_x$.

We implicitly identify objects in $M$ and in $\upT_{x_0} M$ via the exponential map $\exp_{x_0}$, \eg{} $x_0 = 0$, 
and tangent vectors to some point $x \in M$ to vectors (or points) in $\upT_{x_0} M$ via the derivative of the exponential map.
Let $r>0$. In what follows, all points considered (typically denoted $x$, $A$, $B$) are within distance $\leqslant r$ of $x_0$. 
With this setup, \eqref{eq:RiemannExpansion} is written:
\begin{theorem}[Second-order expansion of the metric.] \label{thm:ExpansionInnerProduct}
Let $u, v$ be tangent vectors at some point $x \in M$. Then
\begin{equation}
\label{eq:ExpansionInnerProduct}
\langle u,v \rangle = \langle u,v \rangle_{\tE} - \frac13 \left\langle R(u,x)x,v \right\rangle_{\tE} + \bigO\left(r^3\|u\|_{\tE} \|v\|_{\tE}\right)\,.
\end{equation}
where $R$ is the Riemann curvature tensor at $x_0 = 0$.
\end{theorem}
Note that when writing $R(u,x)x$, we think of the point $x$ as an element of $\upT_{x_0} M$. From this fundamental estimate, it is elementary to show the following series of estimates.

\begin{remark}
All the $\bigO\left(\cdot \right)$ functions in this section are locally uniform in $x \in M$.
\end{remark}

\begin{proposition}[Second-order expansion of the norm]
\label{prop:ExpansionNorm}
\begin{align}
\label{eq:ExpansionNorm}
\|u\|^2 &= \|u \|^2_{\tE} - \frac13 \left\langle R(u,x)x,u \right\rangle + \bigO\left(r^3\|u\|^2 \right)\\
\|u\| &= \|u \|_{\tE} - \frac16 \frac{\left\langle R(u,x)x,u \right\rangle}{\|u \|^2_{\tE}} + \bigO\left(r^3\right)\,.
\end{align}
\end{proposition}

\begin{proposition}[Second-order expansion of cosine]
\label{prop:ExpansionCosine}
\begin{equation}
\label{eq:ExpansionCosine}
\cos \angle(u,v) = \cos \angle_{\tE}(u,v) \left[ 1+ \frac{\left\langle R(u,x)x,u \right\rangle_{\tE}}{6\|u\|_{\tE}^2}
 + \frac{\left \langle R(v,x)x,v\right\rangle_{\tE} }{6\|v\|_{\tE}^2}
 - \frac {\left \langle R(u,x)x,v \right\rangle_{\tE}}{3 \langle u,v\rangle_{\tE}} + \bigO\left(\frac{r^3}{\cos \angle_{\tE}(u,v)}\right) \right]
\end{equation}
\end{proposition}

The previous proposition implies the less accurate estimates:
\begin{proposition}[First-order expansions of angles]
\label{prop:ExpansionAngles}
\begin{equation}
\label{eq:ExpansionCosineFirstOrder}
\cos \angle(u,v) = \cos \angle_{\tE}(u,v) + \bigO\left(r^2\right)\,.
\end{equation}
If $\angle(u,v)$ (equivalently $\angle_{\tE}(u,v)$) is bounded away from $0$ and $\frac{\pi}{2}$ modulo $\pi$, then
\begin{align}
 \sin \angle(u,v) &= \sin \angle_{\tE}(u,v) + \bigO\left(r^2\right)\\
 \cot \angle(u,v) &= \cot \angle_{\tE}(u,v) + \bigO\left(r^2\right)\,.
\end{align}
\end{proposition}

Let $A$, $B$ be points in our normal chart: they can either be thought of as elements of $M$ or $\upT_{x_0} M$.
We denote as usual $\overrightarrow{AB}$ the vector $\exp_{A}^{-1} (B)$, which is an element of $\upT_A M$, or of $\upT_{x_0} M$ via our chart.
We also denote $\overrightarrow{AB}^\tE$ the Euclidean vector $B - A \in \upT_{x_0} M$. 

\begin{proposition}[Geodesic through two points]
\label{prop:GeodesicThroughTwoPoints}
Let $\gamma$ be the geodesic with $\gamma(0)=A$ and $\gamma(1)=B$.
\begin{equation}
\label{eq:GeodesicThroughTwoPoints}
\gamma(t) = \gamma_{\tE}(t) + \frac{t(t-1)}{3} R\left(A , B\right) \overrightarrow{AB}^\tE + O\left(tr^4\right)\,.
\end{equation}
\end{proposition}

\begin{proposition}[Vector between two points]
\label{prop:VectorExpansion}
\begin{equation}
\overrightarrow{AB} = \overrightarrow{AB}^{\tE} + \frac13 R\left(A,B\right ) \overrightarrow{AB}^\tE + \bigO\left(r^4\right)\,.
\end{equation}
\end{proposition}

\begin{proposition}[Midpoint]
\label{prop:midpoint}
Let $I$ denote be the midpoint of midpoint of $A$ and $B$ in $M$, and let $I_{\tE} = \frac{A+B}{2}$ denote their Euclidean midpoint in $\upT_{x_0} M$. 
\begin{equation}
I = I_{\tE} + \frac1{12} R\left(A,B\right)\overrightarrow{AB}^\tE + \bigO\left(r^4\right)\,.
\end{equation}
\end{proposition}

\begin{proposition}[Distance between two points] \label{prop:DistanceExpansion}
\begin{equation}
 d(A,B)^2 = d_{\tE}(A,B)^2 -\frac{1}{3} \left\langle R\left(B,A\right )A, B \right\rangle + \bigO(r^5)\,.
\end{equation}
\end{proposition}
\begin{remark}
 Note that $\left\langle R\left(B,A\right )A, B \right\rangle = K \Vert B \wedge A \Vert^2$ where $K$ is the sectional curvature at $x_{0} = 0$.
 In particular, we see from \autoref{prop:DistanceExpansion} that $d > d_\tE$ near $x_0$ if and only if $M$ has negative sectional curvature at $x_0$, which should be expected.
\end{remark}

We recover the well-known expansion of the volume density:
\begin{proposition}[Volume density] \label{prop:VolumeDensity}
The volume density at $x$ is given by
\begin{equation}
 v_g(x) = v_\tE \left(1 - \frac{\Ric(x,x)}{6} + \bigO(r^3)\right)
\end{equation}
where $v_E$ is the Euclidean volume density in $\upT_x M$ and $\Ric$ is the Ricci curvature tensor at $x_0$.
\end{proposition}

Let us finish with the following estimate that we use in \autoref{sec:Construction} (see \autoref{lem:FirstOrderLapInterior}): 
\begin{proposition} \label{prop:RiemEstimateAngle3Points}
 Let $A$, $B$ be two points such that all three sides of the triangle $OAB$ are $\bigO(r)$ (where $O = x_0$). Denote $\alpha$ the unoriented angle $\widehat{BAO}$ 
 and $\alpha_\tE$ its Euclidean counterpart in the normal chart at $x_0$. Then we have the second-order expansion
 \begin{equation}
  \cot \alpha = \cot \alpha_\tE + \varepsilon_{OAB} + \bigO(r^3) \quad \text{ with } \quad \varepsilon_{OAB} = \frac{K}{6}\left(\frac{2 \Vert OA \Vert_\tE \, \Vert AB \Vert_\tE}{\sin \alpha_\tE} + \Vert OA \Vert_\tE^2 \cot \alpha_\tE\right)
 \end{equation}
where $K$ denotes the sectional curvature at $x_0$.
\end{proposition}

\begin{proof}
Write $\widehat{BAO} = \angle (\overrightarrow{AB}, \overrightarrow{A0})$. Use \autoref{prop:VectorExpansion} to replace
$\overrightarrow{AB}$ and $\overrightarrow{AO}$ by their Euclidean counterparts, and \autoref{prop:ExpansionCosine} to compare the Riemannian angle
to its Euclidean counterpart. We spare the lengthy but straightforward details.
\end{proof}

\subsection{Iterated subdivision of a simplex}
\label{subsec:SubdividingSimplex}

In this subsection, we estimate the edge lengths and angles in the iterated midpoint subdivision (see \autoref{subsec:MidpointSubdivision}) of a simplex in a Riemannian manifold.
We prove two propositions, which are the key to \autoref{thm:SubdivisionCrystalline} and \autoref{thm:AcuteRefinement} respectively.

\begin{proposition} \label{prop:simplexSubdivisionCrystallineAppendix}
Let $(M,g)$ be a compact Riemannian manifold of dimension $m$. Let $(\Delta_n)_{n \in \N}$ be a sequence of simplices with geodesic edges such that 
for every $n\in \N$, $\Delta_{n+1}$ is one of the $2^m$ simplices obtained from $\Delta_n$ by midpoint subdivision.
Then all edge lengths of $\Delta_n$ are $\Theta(2^{-n})$.
\end{proposition}

\begin{remark} \label{rem:simplexSubdivisionCrystallineAppendix}
 In \autoref{prop:simplexSubdivisionCrystallineAppendix}, the $\Theta(2^{-n})$ function is uniform in the choice of the sequence $(\Delta_n)$: more precisely,
 there exists constants $C_1, C_2>0$ depending only $(M,g)$ such that any edge length $x_n$ of the triangulation obtained by $n$-th refinement 
 of $\Delta_0$ satisfies $C_1 2^{-n} \leqslant x_n \leqslant C_2 2^{-n}$. 
\end{remark}

\begin{proof}
For comfort, we write the proof when $\dim M = 2$, but it works in any dimensions. We thus have a sequence of geodesic triangles $\Delta_n$
in a Riemannian surface $(S, g)$. Choose a labelling of the side lengths of $\Delta_n$ by $a_n$, 
$b_n$, $c_n$. Given the labelling of $\Delta_0$, there is a unique sensible way to do this for all $n$ so that $\Delta_{n+1}$ is ``similar'' to $\Delta_n$. For instance, in the Euclidean setting, one should have $a_n = 2^{-n} a_0$, etc.
In order to show that $a_n$, $b_n$, and $c_n$ are $\Theta(2^{-n})$, we would like to use Riemannian estimates, 
but we must first show that $\diam(\Delta_n)$ converges to zero. 

Let us prove the stronger claim that $r_n \to 0$, where $r_n$
is the maximum edge length of the whole triangulation obtained by $n$-th refinement of $\Delta_0$.
Notice that $(r_n)$ is nonincreasing: this follows easily from the triangle inequality in each simplex. Moreover $r_n < r_{n+1}$ unless one of the simplices is reduced to a point, which cannot happen unless $\Delta_0$ is a point. One can conclude that $r_n \to 0$ by compactness: if not, we could find a converging sequence of simplices with diameter bounded below, etc.

Now we can use the estimates of \autoref{subsec:RiemannianExpansionsNormalChart}. It is not hard to derive from 
\autoref{prop:midpoint} and \autoref{prop:DistanceExpansion} that
\begin{equation} \label{eq:Ineqan}
 \left|a_n - 2a_{n+1}\right| = \bigO(r_n^3)
\end{equation}
and we have similar estimates for $b_n$ and $c_n$. This means that there exists a constant $B>0$ such that for all $n$ sufficiently large, 
$\left|a_n - 2a_{n+1}\right| \leqslant B r_n^3$. 
Applying this inequality repeatedly, we find
\begin{equation} 
\begin{aligned}
\left|a_n - 2^k a_{n+k} \right| &= \left|(a_n - 2a_{n+1}) + 2(a_{n+1}- 2a_{n+2}) + \ldots + 2^{k-1}(a_{n+k-1} - 2 a_{n+k})\right| \\
& \label{eq:Ineqan2} \leqslant B\left ( r_n^3 + 2r_{n+1}^3 + \ldots + 2^{k-1} r_{n+k-1}^3 \right )\,.
\end{aligned}
\end{equation}
Now, note that $r_n$ must satisfy the same inequality \eqref{eq:Ineqan}, so in particular
\begin{equation}
2r_{n + 1} \leqslant r_n + B r_n^2 \leqslant C r_n
\end{equation}
for any constant $C>1$ chosen in advance, provided $n$ is sufficiently large. Therefore 
we obtain
\begin{equation}
\left|a_n - 2^k a_{n+k} \right| \leqslant B r_n^3 \left( 1 + \frac{C^3}{2} + \ldots + \left(\frac{C^3}{2}\right)^{k-1} \right) \,.
\end{equation}
Provided we chose $1 < C^3 < 2$, the sum $1 + \frac{C^3}{2} + \ldots + \left(\frac{C^3}{2}\right)^{k-1}$ is bounded, as a truncated convergent geometric series.
In particular, we find that the sequence $(2^k a_{n+k})_{k \in \N}$ is bounded, in other words $a_{n+k} = \bigO(2^{-k})$. Of course this is the same as 
saying that $a_n = \bigO(2^{-n})$. We similarly show the other inequality $a_n = \Omega(2^{-n})$, and conclude that $a_n = \Theta(2^{-n})$.
Obviously, the same argument works for $(b_n)$ and $(c_n)$.

Note that the claim of \autoref{rem:simplexSubdivisionCrystallineAppendix} is justified by the fact that the sequence $(r_n)$ and the constant $C$
are independent of the choice of the sequence $(\Delta_n)$.
\end{proof}

\begin{proposition} \label{prop:simplexSubdivisionAcuteAppendix}
Let $(M,g)$ be a compact Riemannian manifold of dimension $m$. Let $\delta>0$. There exists $R>0$ and $\eta>0$ such that the following holds. Let 
$(\Delta_n)_{n \in \N}$ be a sequence of simplices with geodesic edges where for every $n\in \N$, $\Delta_{n+1}$ is one of the $2^m$ simplices obtained from $\Delta_n$ by midpoint subdivision. If the longest edge length of $\Delta_0$ is $\leqslant R$ and all angles of $\Delta_0$ are $\leqslant \frac{\pi}{2} - \delta$, then all angles of $\Delta_n$ are $\leqslant \frac{\pi}{2} - \eta$ for all $n\in \N$.
\end{proposition}

\begin{proof}
We have seen in \autoref{prop:simplexSubdivisionAcuteAppendix} that the diameter of $\Delta_n$ is $\leqslant r_n$, with $r_n = \Theta(2^{-n})$.
in particular, $r_n \to 0$ so we can use the Riemannian estimates of \autoref{subsec:RiemannianExpansionsNormalChart}. 

Label $\alpha_n$, $\beta_n$, and $\gamma_n$ the angles of $\Delta_n$. Of course, one should do this labelling in the only sensible way:
for instance in the Euclidean setting we should have $\alpha_n = \alpha_{n+1}$, etc.
It is not hard to derive from \autoref{prop:midpoint} and \autoref{prop:ExpansionAngles} that for all $n \in \N$,
\begin{equation}
 \cos \alpha_{n+1} = \cos \alpha_n + \bigO(r_n^2)\,,
\end{equation}
in other words there exists a constant $C$ depending only on $(M,g)$ such that 
\begin{equation}
\left| \cos \alpha_{n+1} - \cos \alpha_n \right| \leqslant C r_0 2^{-2n}\,.
\end{equation}
Using a telescopic sum, we find that
\begin{align}
 \left| \cos \alpha_{n} - \cos \alpha_0 \right| &\leqslant \sum_{k=0}^{n-1} \left| \cos \alpha_{k+1} - \cos \alpha_k \right| \\
 & \leqslant C r_0 \sum_{k=0}^{n-1} 2^{-2k} \leqslant C r_0 \sum_{k=0}^{\infty} 2^{-2k} = C r_0 \frac{4}{3}\,.
\end{align}
We therefore have the bound
\begin{equation}
\cos \alpha_{n} \geqslant \cos \alpha_0 - C' r_0                  
\end{equation}
where $C' = 4C/3$. By assumption, $\cos \alpha_0 \geqslant \cos(\pi/2 - \delta) = \sin(\delta)$.
Clearly $\cos \alpha_n$ is bounded away from zero if $r_0$ is sufficiently small, for instance
$r_0 \leqslant \frac{\sin \delta}{2C'}$ yields $\cos \alpha_n \geqslant \frac{\sin \delta}{2}$.
It follows that $\alpha_{n}$ is bounded away from $\pi/2$.
\end{proof}

\section{\texorpdfstring{Comparing the discrete $\upL^2$ and $\upL^\infty$ distances}{Comparing the discrete L2 and L-infinity distances}}
\label{sec:ComparingL2Linfty}

Let $M$ be compact Riemannian manifold, let $N$ be a complete Riemannian manifold of nonpositive sectional curvature.
Let $\cM$ be a mesh on $M$ and equip the underlying graph $\cG$ with vertex weights $(\mu_x)_{x\in \cV}$ and $(\omega_{xy})_{x \sim y}$.
Recall the $\upL^2$ distance on the space of discrete maps $\Map_\cG(M,N)$:
\begin{equation}
 d(u,v)^2 = \sum_{x \in \cV} \mu_x \, d(u(x), v(x))^2
\end{equation}
while the $\upL^\infty$ distance is
\begin{equation}
 d_\infty(u,v) = \max_{x \in \cV} d(u(x), v(x))\,.
\end{equation}
Clearly, these distances satisfy the inequality $m d_\infty^2 \leqslant d^2 \leqslant W d_\infty^2$, where
$m \coloneqq \min_{x \in \cV} \mu_x$ is the minimum vertex weight and $W \coloneqq \sum_{x \in \cV} \mu_x$ is the sum of the vertex weights.
Typically, $W$ is equal to $\Vol (M)$ or asymptotic to it for a fine mesh, so the second inequality is fairly robust. On the other hand, the first
inequality $m d_\infty^2 \leqslant d^2$, which we rewrite
\begin{equation} \label{eq:ComparingDistances}
 d_\infty(u,v) \leqslant m^{-1/2} \,d(u,v)\,,
\end{equation}
is less attractive since typically $m \to 0$ for a fine mesh. This should be expected though, as
the $\upL^2$ and $\upL^\infty$ distances are not equivalent on the space of continuous maps $M \to N$.
The goal of this section is to find an improvement of \eqref{eq:ComparingDistances} when $v$ is a discrete harmonic map. 
This step is crucial in our proof of \autoref{thm:ConvergenceL2}.

\begin{proposition} \label{prop:bootstrap}
Let $\cG$ be a biweighted graph embedded in $M$, and let $N$ be a complete Riemannian manifold of nonpositive sectional curvature. Let $v\in \Map_\cG(M, N)$ be the minimizer of the discrete energy. Denote by $r$ the maximum edge length of $\cG$, $V$ the maximum valence of a vertex of $\cG$, 
$m = \min_{x \in \cV} \mu_x$ the smallest vertex weight, and $\omega = \frac{\omega_{\max}}{\omega_{\min}}$ the ratio of the largest and smallest
edge weights. Let $L>0$. There exists constants $A = A(\omega,V) > 0$ and $B = B(\omega, L,V) \in \R$ such that for any $L$-Lipschitz map $u \in \Map_\cG(M,N)$:
\begin{equation}
 d_\infty(u,v) \leqslant \max \left\{(\kappa m)^{-1/2} d(u,v) ~, ~ r^{1/2}\right\} 
\end{equation}
with $\kappa \coloneqq \min  \left(A \log \left(r^{-1}\right) + B \, , \, 
\surj \rad \cG -1\right)$.
\end{proposition}

We recall that the combinatorial 
surjectivity radius 
$\surj\rad\cG$ is defined below \autoref{thm:CrystallineProperties}.

\begin{proof}
Let $\rho \coloneqq Lr$. Notice that $\rho$ is an upper bound for the length of any edge in $N$ that is the image of an edge of $\cG$ by $u$.
\autoref{prop:bootstrap} is a consequence of the following ``bootstrapping'' lemma: if some distance $d(u(x),v(x))$ is large, then $d(u(y),v(y))$ will also be large, for many vertices $y$ that are near $x$. More precisely:
\begin{lemma}
\label{lem:bootstrap}
Let $x_0$ be a vertex which achieves $d_\infty(u,v) \eqqcolon D$.  Let $K$ be given by
\begin{equation}
K \coloneqq \min\left\{ \ \left \lfloor \underline \log_{\delta}(D/\rho) \right \rfloor \ , \ 
\surj\rad \cG  \ \right\}\,,
\end{equation}
where $\delta = 2\left(1+ \omega V \right)$.
For each $k=1,2,\ldots,K$ there exists a vertex $x_k$ satisfying:
\begin{enumerate}[(1)]
\item The combinatorial distance in $\cG_n$ is given by $d_{\cG_n}(x,x_k) = k$, and
\item $d(u(x_k),v(x_k)) \geqslant D - \delta^{k-1}\rho$.
\end{enumerate}
\end{lemma}

\begin{remark}
The $\underline \log$ above is the cutoff function $\underline \log_b (x) \coloneqq \max\{\, \log_b x\, ,\,0\,\}$.
\end{remark}

Let us postpone the proof of \autoref{lem:bootstrap} until after the end of this proof.
Now we find
\begin{align}
d(u,v)^2 &= \sum_{x \in \cV} \mu_x d(u(x), v(x))^2 \geqslant m \, \sum_{k=0}^K d(u(x_k),v(x_k))^2 \\
& \geqslant m D^2 +  m \, \sum_{k=0}^{K-1} \left(D- \delta^{k}\rho\right)^2 \\
& \geqslant m D^2 (K+1) -2m D\rho \delta^{K} \\
& \geqslant m D^2 (K+1) -2m D\rho \delta^{\log_\delta(D/\rho)} = m D^2(K-1)\,.
\end{align}
The conclusion follows by noting that if $D \leqslant r^{1/2}$ \ie{} $d_{\infty}(u,v) \leqslant r^{1/2}$, then we are done, and if $D \geqslant r^{1/2}$ then 
$D / \rho \geqslant r^{-1/2}/L$, therefore $ K - 1 \geqslant \min  \left(A \log \left(r^{-1}\right) + B \, , \, 
\surj\rad \cG -1\right)$
where $A = \frac{1}{2 \log \delta}$ and $B = -\log_\delta(L) - 1$.
\end{proof}

\begin{proof}[Proof of \autoref{lem:bootstrap}]
We make repeated use of the following fact (see \cite[Prop.~2.22]{Gaster-Loustau-Monsaingeon1}): since $v$ is a discrete harmonic map its discrete tension field is zero: $\sum_{y\sim x} \omega_{xy} \overrightarrow{v(x)v(y)} = 0$.
In other words $v(x)$ is the weighted barycenter of its neighbor values in $N$. We refer to this as the \emph{balanced condition} of $v$ at $x$.

We prove \autoref{lem:bootstrap} by induction on $k$. For the base case $k=1$, consider the 
unit geodesic $\gamma$ through $v(x_0)$ and $u(x_0)$, parametrized with a coordinate $t$ chosen by requiring $\gamma(0) = v(x_0)$ and $\gamma(-D) = u(x_0)$.
Define the orthogonal projection $\mathit{pr}_\gamma$ as a map $\upT_{v(x_0)}N \to \gamma \approx \R$.
If $\mathit{pr}_\gamma(v(y))< 0$ for all $y\sim x_0$ then $v$ would not be balanced at $x_0$, therefore there exists some neighbor vertex $x_1 \sim x_0$ so that $\mathit{pr}_\gamma(v(x_1))\geqslant 0$. Moreover, by assumption $u(x_1)$ is within $\rho$ of $u(x_0)$, so that $\mathit{pr}_\gamma(u(x_1))\leqslant \mathit{pr}_\gamma(u(x_0))+\rho = -D + \rho$. 
We conclude that $d(v(x_1),u(x_1)) \geqslant \mathit{pr}_\gamma(v(x_1)) - \mathit{pr}_\gamma(u(x_0)) \geqslant D- \rho$.

For the inductive step, we follow the above argument with $x_k$ in place of $x_0$. That is, we have the unit geodesic $\gamma$ through $u(x_k)$ and $v(x_k)$, 
with $\gamma(0) = v(x_k)$, $\gamma(t) = u(x_k)$ for some $t<0$, and the projection $\mathit{pr}_\gamma:T_{v(x_k)}N \to \gamma \approx \R$. 
Split up the neighbors of $x_k$ into $\cA$, those vertices at combinatorial distance at most $k$ from $x_0$ in $\cG$, and $\cB$, those vertices at distance $k+1$ from $x_0$.
For each of the vertices $y\in \cA$, observe that $\mathit{pr}_\gamma(v(y))\leqslant -d(v(x_k),u(x_k)) + \rho + D \leqslant (1+\delta^{k-1})\rho$. 
Now the balanced condition for $v$ at $x_k$ gives
\begin{align}
0 &= \sum_{y\sim x_k} \omega_{x_ky} \ \mathit{pr}_\gamma \left(\overrightarrow{v(x_k)(v(y)}\right)  \\
&= \sum_{y\in \cA} \omega_{x_ky} \ \mathit{pr}_\gamma \left(\overrightarrow{v(x_k)v(y)} \right)
+ \sum_{y\in \cB} \omega_{x_ky} \ \mathit{pr}_\gamma \left(\overrightarrow{v(x_k)v(y)} \right)\\
&\leqslant  \omega_{\max} \sum_{y\in \cA} (1+\delta^{k-1})\rho
+ \sum_{y\in \cB} \omega_{x_ky} \ \max_{y'\in \cB} \ \mathit{pr}_\gamma\left(\overrightarrow{v(x_k)v(y')}\right)\,.
\end{align}
If $\mathit{pr}_\gamma\left( \overrightarrow{v(x_k)v(y)} \right) \geqslant 0$ for some $y\in \cB$, then $d(v(y),u(y)) \geqslant d(v(x_k),u(x_k)) -\rho$, so we may let $x_{k+1}=y$. Otherwise, each of these coordinates are negative, and we have
\begin{equation}
\label{eq:ProjBound}
0 < \omega_{\max} V (1+\delta^{k-1})\rho + \omega_{\min} \ \max_{y\in \cB} \ \mathit{pr}_\gamma\left( \overrightarrow{v(x_k)(v(y)} \right)\,.
\end{equation}
Let $x_{k+1}\in \cB$ satisfy $\mathit{pr}_\gamma\left( \overrightarrow{v(x_k)v(x_{k+1})} \right) =  \max_{y\in \cB} \ \mathit{pr}_\gamma\left( \overrightarrow{v(x_k)v(y)} \right)$.
Rearranging \eqref{eq:ProjBound}, 
\begin{equation}
\mathit{pr}_\gamma\left( \overrightarrow{v(x_k)v(x_{k+1})} \right) >  - \omega V(1+\delta^{k-1})\rho\,.
\end{equation}
Because $u(x_{k+1})$ is within $\rho$ of $u(x_k)$, we find that $\mathit{pr}_\gamma(u(x_{k+1}))\leqslant \mathit{pr}_\gamma(u(x_k))+\rho$. 
By the induction hypothesis,
\begin{align}
d(v(x_{k+1}),u(x_{k+1})) & \geqslant \mathit{pr}_\gamma\left( \overrightarrow{v(x_k)v(x_{k+1})} \right) - \left( u(x_k) +\rho \right) \\
& > - \omega V(1+\delta^{k-1})\rho + d(u(x_k),v(x_k)) - \rho \\
& \geqslant D - \rho\left( 1 + \omega V\right)(1+\delta^{k-1}) \,.
\end{align}
Finally, we have
\begin{equation}
\begin{split}
(1 + \omega V) (1+\delta^{k-1}) &= \frac \delta 2 = (1+\delta^{k-1}) = \frac \delta 2 + \frac{\delta^k} 2\\
&\leqslant \frac {\delta^k} 2 + \frac{\delta^k} 2 = \delta^k
\end{split}
\end{equation}
so that we conclude $d(v(x_{k+1}),u(x_{k+1}) \geqslant D - \delta^k \rho$.
\end{proof}

As an application of \autoref{prop:bootstrap} we get:

\begin{corollary} \label{cor:bootstrap}
Let $M$ be a compact manifold and let $N$ be a complete manifold of nonpositive sectional curvature.
Equip $M$ with a sequence of meshes $(\cM_n)_{n\in\N}$ that is fine and crystalline, let $r=r_n$ denote the mesh size of $\cM_n$, and equip the underlying graphs $\cG_n$
with asymptotic vertex weights and positive edge weights. Assume that there are uniform upper bounds for the ratio of any two edge weights. 

Let $w \colon M \to N$ be a smooth map, denote by $w_n$ its discretization along $\cG_n$, and let $v_n$ be a discrete harmonic map.
Then there is a constant $C>0$ so that 
\begin{equation}
d_\infty(w_n,v_n) \leqslant C \max \left\{ \;  r^{-\dim M/2} \log\left( \frac 1r\right)^{-1/2}  d(w_n,v_n) \; , \sqrt r \; \right\}\, .
\end{equation}
\end{corollary}

\begin{remark}
In the setting above, \eqref{eq:ComparingDistances} would yield only
\begin{equation}
d_\infty(w_n,v_n) \leqslant \bigO\left( r^{-\dim M/2} \right) \cdot d(w_n,v_n) \, .
\end{equation}
\autoref{cor:bootstrap} represents a slight improvement when $v_n$ is discrete harmonic.
\end{remark}

\begin{proof}
 Note that since $w$ is $\cC^1$ on a compact manifold, it must be $L$-Lipschitz for some $L>0$,
 and for all $n \in \N$ the discretization $w_n$ is also $L$-Lipschitz. \autoref{prop:bootstrap} yields
 \begin{equation}
 d_\infty(w_n,v_n) \leqslant \max \left\{(\kappa_n m_n)^{-1/2} d(w_n,v_n) ~, ~ r^{1/2}\right\} \,,
\end{equation}
where $\kappa_n = \min  \left(A \log \left(r_n^{-1}\right) + B \, , \, 
\surj \rad \cG_n -1\right)$, for some uniform constants $A>0$ and $B\in \R$.
In our setting, $m_n = \Theta(r^{\dim M})$ by \autoref{thm:CrystallineProperties} \ref{item:CrystallinePropertiesi} 
and $
\surj \rad \cG_n = \Theta(r^{\dim M})$ by \autoref{thm:CrystallineProperties} \ref{item:CrystallinePropertiesiv}. 
Therefore, there is some constant $C>0$ so that, for $n$ sufficiently large, $m_n \geqslant C r^{\dim M}$ and $\kappa_n \geqslant C \log \left(r^{-1}\right)$, and it follows that
\begin{equation}
d_\infty(w_n,v_n) \leqslant \max \left\{ \; C r^{-\dim M/2} \log\left( \frac 1r\right)^{-1/2} \cdot d(w_n,v_n) \; , \sqrt r \; \right\}\, . 
\end{equation}
\end{proof}

\cleardoublepage\phantomsection 
\bibliographystyle{alpha}
{\small \bibliography{biblio}}

\end{document}